\documentclass[10pt]{amsart}
\usepackage{latexsym}
\usepackage{amscd,amssymb,amsmath,amsxtra}
\usepackage[margin=3cm]{geometry}
\usepackage{graphicx}
\usepackage[curve]{xypic}
\usepackage{stmaryrd}
\newcommand{\myfatslash}{\mathbin{\mkern-6mu\fatslash}}

\makeatletter
\let\@wraptoccontribs\wraptoccontribs
\makeatother
\usepackage{amsmath}
\usepackage{epsfig}
\usepackage{graphicx}
\usepackage{eufrak}
\usepackage{dsfont}
\usepackage[english]{babel}
\usepackage{color}
\usepackage{mathabx}

\usepackage{palatino, mathpazo}
\usepackage{amscd}      % to be able to use "CD" environment
\usepackage{amssymb}
\usepackage{dsfont}
\usepackage{xypic}      % to be able to use "diagram" environment of
\LaTeXdiagrams          % xypic to typeset commutative diagrams
\usepackage[all,v2]{xy}
\xyoption{2cell} \UseAllTwocells \xyoption{frame} \CompileMatrices
\usepackage{latexsym}
\usepackage{epsfig}
\usepackage{enumerate}

\usepackage{latexsym}
\usepackage{epsfig}
\usepackage{amsfonts}
\usepackage{enumerate}
\usepackage{mathrsfs}

\allowdisplaybreaks[3]

\newtheorem{prop}{Proposition}[section]
\newtheorem{lem}[prop]{Lemma}

\newtheorem{cor}[prop]{Corollary}
\newtheorem{thm}[prop]{Theorem}
\newtheorem{rmk}[prop]{Remark}

\newtheorem{defn}[prop]{Definition}

\newtheorem{con}[prop]{Conjecture}

            {\nolinebreak $\Box$ \end{trivlist}}

\newcommand{\noprint}[1]{}

\newcommand{\Ext}{\mbox{Ext}}

\newcommand{\Hom}{\mbox{Hom}}

\newcommand{\tw}{\mbox{\tiny tw}}

\newcommand{\eF}{\mathop{\sf F}\nolimits}

\newcommand{\Z}{\mathop{\sf Z}\nolimits}
\newcommand{\N}{\mathcal{N}}

\newcommand{\XX}{{\mathfrak X}}

\renewcommand{\SS}{{\mathfrak S}}

\newcommand{\YY}{{\mathfrak Y}}

\newcommand{\zz}{{\mathbb Z}}
\newcommand{\hh}{{\mathbb H}}

\newcommand{\aaa}{{\mathbb A}}

\renewcommand{\ll}{{\mathbb L}}
\newcommand{\qq}{{\mathbb Q}}
\newcommand{\pp}{{\mathbb P}}
\newcommand{\cc}{{\mathbb C}}
\newcommand{\rr}{{\mathbb R}}

\newcommand{\sT}{{\mathcal T}}
\newcommand{\sD}{{\mathcal D}}
\newcommand{\sV}{{\mathcal V}}
\newcommand{\sK}{{\mathcal K}}
\newcommand{\sC}{{\mathcal C}}
\newcommand{\sE}{{\mathcal E}}
\newcommand{\sI}{{\mathcal I}}

\newcommand{\sL}{{\mathcal L}}

\newcommand{\sS}{{\mathcal S}}

\newcommand{\sP}{{\mathcal P}}
\newcommand{\sU}{{\mathcal U}}
\newcommand{\sO}{{\mathcal O}}

\newcommand{\sW}{{\mathcal W}}
\newcommand{\sY}{{\mathcal Y}}
\newcommand{\sM}{{\mathcal M}}
\newcommand{\hH}{{\mathcal H}}

\newcommand{\sZ}{{\mathcal Z}}

\newcommand{\sF}{{\mathcal F}}
\newcommand{\sA}{{\mathcal A}}

\newcommand{\Coh}{\mbox{Coh}}

\newcommand{\rE}{\mathscr{E}}

\newcommand{\sB}{\mathscr{B}}

\DeclareMathOperator{\Var}{Var}

\DeclareMathOperator{\id}{id}

\DeclareMathOperator{\St}{St}
\DeclareMathOperator{\Sch}{Sch}

\DeclareMathOperator{\DT}{DT}
\DeclareMathOperator{\Hilb}{Hilb}

\DeclareMathOperator{\Aut}{Aut}

\DeclareMathOperator{\At}{At}
\DeclareMathOperator{\Cone}{Cone}
\DeclareMathOperator{\vir}{vir}

\DeclareMathOperator{\Pic}{Pic}

\DeclareMathOperator{\Ch}{Ch}
\DeclareMathOperator{\CR}{CR}

\DeclareMathOperator{\even}{even}

\DeclareMathOperator{\Br}{Br}
\DeclareMathOperator{\per}{per}

\DeclareMathOperator{\Td}{Td}

\DeclareMathOperator{\NS}{NS}

\DeclareMathOperator{\SU}{SU}

\DeclareMathOperator{\PGL}{PGL}
\DeclareMathOperator{\coker}{coker}
\DeclareMathOperator{\sr}{sr}
\DeclareMathOperator{\rel}{rel}
\DeclareMathOperator{\ex}{ex}
\DeclareMathOperator{\td}{td}
\DeclareMathOperator{\Stab}{Stab}
\DeclareMathOperator{\PT}{PT}
\DeclareMathOperator{\GW}{GW}
\DeclareMathOperator{\na}{na}
\DeclareMathOperator{\Hodge}{Hodge}
\DeclareMathOperator{\free}{free}

\newcommand{\rk}{\mathop{\rm rk}}

\newcommand{\cl}{\mathop{\rm cl}\nolimits}

\newcommand{\tr}{\mathop{\rm tr}\nolimits}

\renewcommand{\Re}{\mathop{\rm Re}}
\renewcommand{\Im}{\mathop{\rm Im}}

\newcommand{\red}{\mathop{\rm red}\nolimits}
\newcommand{\supp}{\mathop{\rm supp}}

\newcommand{\spec}{\mathop{\rm Spec}\nolimits}

\newcommand{\tor}{\mathop{\rm tor}\nolimits}

\numberwithin{equation}{subsection}

\newcommand {\mat}      [1] {\left(\begin{array}{#1}}
\newcommand {\rix}          {\end{array}\right)}

\title[Multiple cover formula for local K3 gerbes]{On multiple cover formula for local  K3 gerbes}
\author{Yunfeng Jiang}
\address{Department of Mathematics\\ University of Kansas\\ 405 Snow Hall 1460 Jayhawk Blvd\\Lawrence KS 66045 USA} 
\email{y.jiang@ku.edu}
\author{Hsian-Hua Tseng}
\address{Department of Mathematics\\ Ohio State University\\ 100 Math Tower 231 West 18th
Ave.\\ Columbus OH 43210 USA}
\email{hhtseng@math.ohio-state.edu}

\begin{document}
\sloppy \maketitle
\begin{abstract}
We generalize the multiple cover formula of Y. Toda  (proved by Maulik-Thomas)  for  counting invariants for semistable coherent sheaves on local K3 surfaces to semistable twisted sheaves over twisted local K3 surfaces.   The formula has an application  to prove any rank S-duality conjecture for K3 surfaces. 
\end{abstract}

%%% -----------------------------------------------------------------------
\maketitle
%%% ----------------------------------------------------------------------

\tableofcontents

\section{Introduction}

Let $S$ be a smooth projective K3 surface, and $X=S\times \cc$ the local K3 surface. 
Toda's multiple 
cover formula for the counting invariants for semistable coherent sheaves on local K3 surface $X$ is a powerful tool for calculating the Donaldson-Thomas  invariants,  see \cite[Corollary 6.8, 6.10]{MT}, \cite[Conjecture 1.3]{Toda_JDG}.   The formula was proved in \cite{MT} and used by Tanaka-Thomas in \cite{TT2} to calculate any rank Vafa-Witten invariants for K3 surfaces. 
In this paper we generalize and prove a multiple cover formula for the counting  invariants of semistable twisted sheaves on the local  optimal K3 gerbe 
$\XX:=\SS\times\cc$.  This twisted version of  multiple cover formula is used in \cite{Jiang-Tseng-K3} to calculate the $\SU(r)/\zz_r$-Vafa-Witten invariants  and prove the 
S-duality conjecture for K3 surfaces in any rank.  The twisted multiple cover formula is also useful for the calculation of Donaldson-Thomas invariants for local K3 gerbes.

Let $S$ be a smooth projective K3 surface.  
The isomorphism classes of $\mu_r$-gerbes on $S$ are classified by the \'etale cohomology group
$H^2(S,\mu_r)$. 
Let  $\SS\to S$ be a  $\mu_r$-gerbe on $S$ which is given by a class $[\SS]\in H^2(S,\mu_r)$.
Then $\SS$ determines a  class 
$\varphi([\SS])=\alpha\in H^2(S,\sO_S^*)_{\tor}$, where $\varphi$
is the map in the long  exact sequence 
$$\cdots\to H^1(S,\sO_S^*)\to H^2(S,\mu_r)\stackrel{\varphi}{\longrightarrow} H^2(S,\sO_S^*)\to \cdots$$
induced by the short exact sequence:
$$1\to \mu_r\to \sO_S^*\stackrel{(\cdot)^r}{\longrightarrow} \sO_S^*\to 1.$$
The cohomology $H^2(S,\sO_S^*)_{\tor}$ is, by definition, the cohomological Brauer group 
$\Br^\prime(S)$, and from de Jong's theorem \cite{de_Jong} the Brauer group 
$\Br(S)=\Br^\prime(S)$. 
A $\mu_r$-gerbe $\SS\to S$ is called {\em optimal} if the class $\varphi([\SS])=\alpha\in H^2(S,\sO_S^*)_{\tor}$ is nonzero and has order $r$. 
We call $(S,\alpha)$,  a K3 surface $S$ together with a Brauer class $\alpha\in \Br^\prime(S)$ a twisted K3 surface as in \cite{HS}. 
We also call the optimal gerbe $\SS\to S$ a twisted K3 surface since its class in  $\Br^\prime(S)$ is $\alpha$.  We always use these two notions. 
The twisted Mukai vectors were constructed in \cite{HMS} and \cite{Yoshioka2}. 
Let $\Coh(\SS)$ be the category of coherent sheaves on the gerbe $\SS$.  From \cite{TT-Adv}, the category has a decomposition 
$$\Coh(\SS)=\bigoplus_{0\leq i\leq r-1} \Coh(\SS)_i$$ 
where $\Coh(\SS)_i$ is the subcategory of coherent sheaves on $\SS$ with $\mu_r$-weight $i$.  The subcategory $\Coh(\SS)_1$ is the category of twisted sheaves and we use 
$\Coh(S,\alpha)$ or $\Coh^{\tw}(\SS)$ to represent this category on a twisted K3 surface. 

Let $\XX:=\SS\times\cc$ which is the local K3 gerbe.   It is an optimal $\mu_r$-gerbe over $X=S\times\cc$ and its class 
in $H^2(X,\mu_r)\cong H^2(S,\mu_r)$ is also given by $\alpha$. 
We let $\Coh(X,\alpha)$ or $\Coh^{\tw}(\XX)$ be the category of twisted sheaves on a local  twisted K3 surface $(X,\alpha)$. 
Let $H(\sA^{\tw}_{\XX})$ be the Hall algebra of the category $\sA_{\XX}^{\tw}=\Coh^{\tw}(\XX)$ as in \cite{Bridgeland10}, \cite{JS}. 
On our gerbes $\SS$ and $\XX$, we always fix a generating sheaf $\Xi=\bigoplus_{i=0}^{r-1}\xi^{i}$, where $\xi$ is a fixed  $\SS$ and $\XX$-twisted locally free sheaf.  The modified Gieseker stability is defined in \cite{Nironi}.
We use the geometric stability in \cite{Lieblich_Duke} or modified Gieseker stability for twisted sheaves and  the moduli stack was constructed in \cite{Lieblich_Duke} and \cite{Nironi}. 
Let $\Gamma_0:=\zz\oplus\NS(S)\oplus\qq$ measuring the Mukai vectors for  $S$. 
Then for each Mukai vector $v\in \Gamma_0$ there is a virtual indecomposable element 
$$
\epsilon_{\omega,\XX}(v):=
\sum_{\substack{\ell\geq 1, v_1+\cdots+v_{\ell}=v, v_i\in\Gamma_0\\
p_{\omega,v_i}=p_{\omega,v}(m)}}
\frac{(-1)^{\ell}-1}{\ell}\delta_{\omega,\XX}(v_1)\star\cdots\star \delta_{\omega,\XX}(v_{\ell})
$$
where $p_{\omega,v_i}(m)$ is the reduced modified Hilbert polynomial, and 
$$\delta_{\omega,\XX}(v):=[\sM_{\omega,\XX}(v)\hookrightarrow \widehat{\sM}(\XX) ]\in H(\sA^{\tw}_{\XX})$$
is an element in the Hall algebra. Here 
$\star$ is the Hall algebra product,  $\omega\in \NS(S)$ is an ample divisor,  and $\sM_{\omega,\XX}(v)\hookrightarrow \widehat{\sM}(\XX) $ is the moduli stack of semistable twisted sheaves 
with Mukai vector $v$ where $\widehat{\sM}(\XX)$ is the stack of coherent twisted sheaves on $\XX$. 
The Joyce invariant $J(v)$ is defined by the Poincar\'e polynomial of $\epsilon_{\omega,\XX}(v)$, see Definition \ref{defn_invariants_Jv_XX}.  
We prove a multiple cover formula for the invariant $J(v)$.

\begin{thm}\label{thm_twisted_multiple_cover_intro}
Let $\SS\to S$ be an optimal $\mu_r$-gerbe over a smooth projective K3 surface $S$.  Then for each Mukai vector $v\in \Gamma_0$, 
we have  a formula for the Joyce invariant: 
$$J(v)=\sum_{k\geq 1; k|v}\frac{1}{k^2}\chi(\Hilb^{\langle v/k,v/k\rangle+1}(S))$$
where $k|v$ means $k$ divides every component of $v$,  $\Hilb^{\langle v/k,v/k\rangle+1}(S)$ is the Hilbert scheme of $\langle v/k,v/k\rangle+1$ points on $S$, and $\langle v/k,v/k\rangle$ is the Mukai pairing. 
\end{thm}

The method to prove Theorem \ref{thm_twisted_multiple_cover_intro} is as follows.  
In \cite{Toda_JDG}, Toda actually defined the curve counting invariants $N_{n,\beta}$ for the rank zero semistable sheave on a K3 surface $S$. 
In \cite{Toda_JAMS} Toda has proved a wall crossing formula of Pandharipande-Thomas stable pair invariants 
$\PT(X)$ in \cite{PT}  for any Calabi-Yau threefold $X$ in terms of the invariants $N_{n,\beta}$ and some limit stable invariants $L_{n,\beta}$, see (\ref{eqn_Toda66}). 
We prove Toda's wall-crossing formula for $\mu_r$-gerbes $\XX\to X$ over a Calabi-Yau threefold $X$. 

In \cite{Toda_JDG}, Toda used wall crossing formula  of weak stability conditions on the derived category of  coherent sheaves on a local K3 surface 
$X=S\times\cc$ to conjecture a wall crossing formula for the stable pair invariants $\PT(X)$ in terms of the Gopakumar-Vafa invariants, see (\ref{eqn_Toda_68}). 
Then Toda showed that for a local K3 surface $X$, these two formula are equivalent if the multiple cover formula 
\begin{equation}\label{eqn_multiple_cover_Nnbeta}
N_{n,\beta}=\sum_{k\geq 1; k|v}\frac{1}{k^2}n_0^{\beta/k}
\end{equation}
holds, where $n_0^{\beta/k}$ is the genus zero  Gopakumar-Vafa invariant.  The formula  (\ref{eqn_multiple_cover_Nnbeta}) now is true since in \cite{MT}, Maulik-Thomas proved the 
multiple cover formula  using KKV formula for K3 surfaces \cite{PT_KKV}. 

For our purpose, it is not necessary to introduce the stable pair theory for twisted local K3 surfaces (which are optimal $\mu_r$-gerbes over  local K3 surfaces). We only need a stable pair theory for arbitrary  $\mu_r$ gerbes.
Thus we generalize and survey some results of stable pair theory for the cyclic $\mu_r$-gerbe $\SS\to S$ over a smooth projective K3 surface $S$, and the reduced stable pair theory on the  local K3 gerbe $\XX=\SS\times\cc$ over $X=S\times \cc$.  We also generalize Oberdieck's theorem for equating the reduced stable pair invariants to Behrend's weighted Euler characteristic  for  \'etale gerbes over surfaces, in which 
similar result also holds  for Joyce-Song twisted stable pairs in \cite{Jiang-Tseng-K3}.
Toda's wall crossing results for counting invariants $N_{n,\beta}$ are also generalized  to cyclic gerbes.  All of these results are routine generalizations for cyclic gerbes, but will have applications to the twisted multiple cover formula.

Now for a twisted K3 surface $(S,\alpha)$, and a local twisted K3 surface $(X=S\times\cc,\alpha)$,  we prove that the counting invariants $N_{n,\beta}$ for these one dimensional sheaves 
keep the same as for trivial $\mu_r$-gerbe $\SS_0$, and all the counting invariants for trivial $\mu_r$-gerbe $\SS_0$ are the same as the K3 surface $S$. 
This is because counting twisted one dimensional sheaves on a twisted local K3 surface is the same as counting usual untwisted one dimensional sheaves on a local K3 surface.
Then we still have the multiple cover formula (\ref{eqn_multiple_cover_Nnbeta}).  
The invariants $N_{n,\beta}$ keep the same for the optimal 
$\mu_r$-gerbe $\SS\to S$ and $\XX\to X$. 
Then we prove a similar result as in \cite[Theorem 1.2]{Toda_JDG} for the twisted sheaves on $\SS$ and 
$\XX$, which we call the twisted Hodge isometry theorem, see Corollary \ref{cor_Jv_Jgv}.  Therefore we get the multiple cover formula in Theorem \ref{thm_twisted_multiple_cover_intro}. 

\subsection{Outline} 

After   we set up basic notions and notations of \'etale gerbes and stacks in \S \ref{sec_gerbes_stacks_notations},  
we introduce stable pair theory of Pandharipande-Thomas  on \'etale gerbes over surfaces  in \S \ref{sec_stable_pair_threefold_DM}; and  prove Oberdieck's theorem for equating the reduced stable pair invariants to Behrend's weighted Euler characteristic  for  \'etale gerbes over surfaces
in \S \ref{sec_Behrend_Oberdieck}.  In \S \ref{sec_KKV_Toda} we generalize the wall crossing formula of Toda to Calabi-Yau \'etale gerbes and review the multiple cover formula of Toda.
In \S \ref{sec_twisted_K3} we use twisted sheaves on optimal  local K3 gerbes, and  study the twisted multiple cover formula for twisted sheaves.      In the Appendix, we generalize Toda's counting invariants 
for semistable objects for K3 surfaces to twisted K3 surfaces and prove that the invariants do not depend on the stability conditions, which is used in  \S \ref{sec_twisted_K3}.

\subsection{Convention}
We work over complex number  $\cc$   throughout of the paper.    We use Roman letter $E$ to represent a coherent sheaf on a projective DM stack or an \'etale gerbe  $\SS$, and use curly letter $\sE$ to represent the sheaves on the total space 
Tot$(\sL)$ of a line bundle $\sL$ over $\SS$. 
We reserve {\em $\rk$} for the rank of the torsion free coherent sheaves $E$.   

For a  K3 gerbe in this paper  we mean a cyclic $\mu_r$-gerbe $\SS\to S$ where $S$ is a K3 surface.

%%% ----------------------------------------------------------------------
\subsection*{Acknowledgments}

Y. J. would like to thank  Kai Behrend,  Amin Gholampour, Martijn Kool,  and   Richard Thomas for valuable discussions on the Vafa-Witten invariants, and Yukinobu Toda for the discussion of the multiple cover formula for K3 and twisted K3 surfaces.  Y. J. is partially supported by  NSF DMS-1600997. H.-H. T. is supported in part by Simons foundation collaboration grant.

%%%----------------------------------------------------------------------

\section{\'Etale gerbes and stacks, notations}\label{sec_gerbes_stacks_notations}

 Our main reference for stacks  is \cite{LMB} and \cite{Stack_Project}.  The notion of  \'etale $\mu_r$-gerbes over schemes and  $\cc^*$-gerbes on a scheme was reviewed in \cite[\S 2]{Jiang_2019}.
 We fix the following notations.
 
 $\bullet$  Let $S$ be a smooth K3 surface, and $p: \SS\to S$ a $\mu_r$-gerbe.  Let $X=S\times\cc$ be the local K3 surface. 
 
 $\bullet$ $p: \XX\to X$ always represents a $\mu_r$-gerbe over a smooth scheme $X$.  
 
 $\bullet$  The isomorphism classes of $\mu_r$-gerbes on $S$ or any other scheme are classified by $H^2(S,\mu_r)$.  The exact sequence 
 $$1\to \mu_r\to \sO_S^*\stackrel{(\cdot)^r}{\longrightarrow}\sO_S^*\to 1$$
 induces a long exact sequence:
 $$\cdots\to H^1(S, \sO_S^*)\stackrel{\psi}{\longrightarrow} H^2(S,\mu_r)\stackrel{\varphi}{\longrightarrow}H^2(S,\sO_S^*)\to \cdots$$
 We call a $\mu_r$-gerbe $p: \SS\to S$ ``essentially trivial" if it is lying in the image of the map $\psi$ in the above exact sequence; and 
 ``optimal" if the order $|\varphi(\SS)|=r$ in $H^2(S,\sO_S^*)_{\tor}$.

$\bullet$  For the surface $S$, the Brauer group $\Br(S)$ is by definition the group of isomorphism classes of Azumaya algebras over $S$; which is equal to its cohomological Brauer group 
$\Br^\prime(S):=H^2(S,\sO_S^*)_{\tor}$.

$\bullet$  In the \'etale cohomology group $H^2(S,\mu_r)$, we call a class $g\in H^2(S,\mu_r)$ ``algebraic" if it comes from a class in $H^1(S, \sO_S^*)$, i.e., a line bundle over $S$ in the above exact sequence. 
We call $g\in H^2(S,\mu_r)$ ``non-algebraic" if its image in $H^2(S, \sO_S^*)$  under $\varphi$ is nonzero.

%%%%%%%%%%%%%%%%%%
\section{Stable pair theory on some threefold DM stacks}\label{sec_stable_pair_threefold_DM}

In this section we list some basic materials on  the stable pair theory of counting curves by Pandharipande-Thomas on threefold DM stacks, and mainly focus on the threefold Calabi-Yau DM stack $\XX:=\SS\times \cc$ and 
$$\YY:=\SS\times E,$$ where $\SS\to S$ is a $\mu_r$-gerbe over a smooth projective surface $S$, and 
$E$ is an elliptic curve over $\cc$.  

\subsection{Stable pair theory}\label{subsec_stable_pair_invariants_DM}

Let $p: \XX\to X$ be a smooth threefold DM stack, which is a $\mu_r$-gerbe over a smooth threefold $X$.  Let  $\Xi$ be a generating sheaf on $\XX$.  The definition and property of generating sheaves can be found in \cite{Nironi}.

\begin{defn}\label{defn_stable_pair_threefold_DM}
A stable pair on $\XX$ is given by a morphism 
$$\sO_{\XX}\otimes \Xi\stackrel{s}{\longrightarrow} F$$
where 
\begin{enumerate}
\item $F$ is a pure dimension one sheaf on $\XX$;
\item $\coker(s)$ is zero dimensional. 
\end{enumerate}
\end{defn}

We explain the reason to make such a definition comparing with the stable pair theory on smooth threefold $X$ in \cite{PT}.  For the $\mu_r$-gerbe $\XX\to X$, recall the geometric stability for a coherent sheaf 
$E$ on $\XX$, which is given by the geometric Hilbert polynomial 
$$P^g(E,m)=\chi^g(E\otimes p^*\sO_X(m)):=[I\XX:\XX]\cdot \deg(\Ch(E(m))\cdot \Td_{\XX}))$$
is defined by the geometric Euler characteristic in \cite{Lieblich_Duke}. 
We assume that $E$ is supported in dimension one, then 
$$\Ch(E)=(0,0,\Ch_2(E)=-c_2(E), \Ch_3(E)=\frac{1}{6}c_3(E))\in H^*(\XX,\qq).$$
Thus let $c_1(p^*\sO_{\XX}(1))=x$, 
we calculate 
$$\chi^g(E\otimes p^*\sO_X(m))=
r\Big[\int_{\XX}(-c_2(E))mx+\deg(\Ch(E)\cdot \Td_{\XX})\Big].$$
Since the sheaf $E$ is supported on a dimension cycle $\beta\in H^4(\XX,\qq)$, the second Chern class 
$c_2(E)=-\beta$ as a class, and we can write the above as 
$$\chi^g(E\otimes p^*\sO_X(m))=
r\Big[m\int_{\beta}x+\deg(\Ch(E)\cdot \Td_{\XX}) \Big].$$

Then fixing a K-group class for a purely one dimensional sheaf 
$E$ in $K_0(\XX)$ is the same as fixing the cycle $\beta\in H^2(\XX,\qq)$ and 
a rational number $n\in \qq$ representing $\Ch_3(E)$.  Let 
$\rk(E):=r\cdot \int_{\beta}x$. The stability of pairs $(E,s)$ is defined as 
$q$-stable if 
\begin{enumerate}
\item $$\frac{\chi^g(F(m))}{\rk(F)}< \frac{\chi^g(E(m))+q(m)}{\rk(E)}, \quad  m>>0$$
for any proper subsheaf $F\subset E$; and 
\item $$\frac{\chi^g(F(m))+q(m)}{\rk(F)}< \frac{\chi^g(E(m))+q(m)}{\rk(E)}, \quad  m>>0$$
for any proper subsheaf $F\subset E$ such that $s$ factors through. 
\end{enumerate}
A similar argument as in \cite[Lemma 1.3]{PT} shows that for sufficient large $q$, semistablity coincides with stability and $(E,s)$ is stable if and only if 
the two conditions in Definition \ref{defn_stable_pair_threefold_DM} hold. 
Fix a cycle $\beta\in H^4(\XX,\qq)$, and a rational number $n\in \qq$ such that $n=\deg(\Ch(E)\cdot \Td_{\XX})$, we define the moduli stack $$P_n^{\Xi}(\XX,\beta)$$ of stable pairs of 
$\XX$ by fixing the generating sheaf $\Xi$. 

\subsection{Stable pairs on $\YY=\SS\times E$.}

In this section we consider the stable pairs on the Calabi-Yau threefold DM stack $\YY=\SS\times E$, where 
$p: \SS\to S$ is a $\mu_r$-gerbe over a smooth projective K3 surface $S$, and $E$ is an elliptic curve.  A curve class in $H_2(\SS,\qq)\subset H_2(\YY,\qq)$ is given by a line bundle in 
$\Pic(\SS)$, if $d\in\zz_{\geq 0}$ is a nonnegative integer, we let 
$(\beta,d)\in H_2(\YY,\qq)$ be the class 
$$(\beta,d):=\iota_{\SS*}(\beta)+\iota_{E*}(d[E])$$
where $\iota_{\SS}: \SS\hookrightarrow \YY$ and $\iota_{E}: E\hookrightarrow \YY$ are inclusions.  
The Calabi-Yau threefold DM stack 
$$p: \YY=\SS\times E\to Y:=S\times E$$ is also a $\mu_r$-gerbe over $Y$ and its class in $H^2(Y, \mu_r)$ is given by the class of the 
gerbe $[\SS]\in H^2(S,\mu_r)$. 
We consider the moduli stack $P_{\alpha}(\YY, (\beta,d))$ of stable pairs $(E,s)$ on $\YY$ with curve class $(\beta,d)$ and 
K-group class $\alpha\in K_0(\YY)$.  
Form the commutative diagram:
\begin{equation}\label{eqn_diagram_1}
\xymatrix{
\SS\ar@{^{(}->}[r]^{\iota_{\SS}}\ar[d]_{p}& \SS\times E\ar[d]^{p}\\
S\ar@{^{(}->}[r]^{\iota_{S}}& S\times E, 
},\quad \quad
\xymatrix{
S\times E\ar[d]& \SS\times E\ar[l]_{p}\ar[r]^{p_{E}}\ar[d]^{p_{\SS}}& E\\
S& \SS\ar[l]_{p}
}
\end{equation}
\begin{lem}\label{lem_stable_pair_rational_2-Chern_class}
The stable pair $(E,s)$ in $P_{\alpha}(\YY, (\beta,d))$ only depends on the curve class 
$(\beta,d)$ and a rational number  $n$ given by 
$$n=\int_{Y}\Ch(p_*E)\cdot \Td_{Y}.$$
\end{lem}
\begin{proof}
We calculate 
\begin{align*}
\int_{\YY}\Ch(\alpha)\cdot\Td_{\SS\times E}&=\int_{\YY}(-c_2(\alpha)+c_3(\alpha))(1+c_1(\YY))+\frac{1}{12}c_1(\YY)^2+c_2(\YY)\\
&=\int_{\YY}\left((-c_2(\alpha))\cdot c_1(\YY)+c_3(\alpha)\right)\\
&=\int_{\YY}\left((\iota_{\SS*}(\beta)+\iota_{E*}(d[E]))(c_1(\SS)+c_1(E))+c_3(\alpha)\right)\\
&=\int_{\YY}\left((\iota_{\SS*}(\beta)\cdot c_1(E)+\iota_{E*}(d[E])\cdot c_1(\SS))+c_3(\alpha)\right)\\
&=\int_{\YY}\iota_{\SS*}(\beta)\cdot c_1(E)+\int_{\YY}\iota_{E*}(d[E])\cdot c_1(\SS)+\int_{\YY}c_3(\alpha)\\
&=\frac{1}{r}\Big[\int_{Y}p_*(\iota_{\SS*}(\beta)\cdot c_1(E))+\int_{Y}p_*(\iota_{E*}(d[E])\cdot c_1(\SS))+\int_{Y}p_*c_3(\alpha)\Big]
\end{align*}
Here we let 
$$p_*[E]=p_*\alpha=\iota_{S*}\overline{\beta}+\iota_{E*}(d[E])\in K_0(Y)$$
and $\overline{\beta}=p_*\beta\in H_2(S,\qq)$, and the last equality is from the $\mu_r$-gerbe structure. 
Thus according to diagrams  in Diagram (\ref{eqn_diagram_1}), the above integral is equal to 
$$\int_{Y}\Ch(p_*E)\cdot \Td_{Y}$$
which we define as the rational number $n$.  
\end{proof}

\begin{rmk}
The number $n=\int_{Y}\Ch(p_*E)\cdot \Td_{Y}$ is actually an integer by Riemann-Roch theorem. 
In the following we let  $$P_{n}(\YY, (\beta,d))$$ be the moduli stack of stable pairs $(E,s)$ on $\YY$ with curve class $(\beta,d)$ and rational number 
$n$ above. 
\end{rmk}

\subsection{The elliptic curve $E$-action}

We also consider the elliptic curve $E$ action on $\YY=\SS\times E$ by the translation of the factor $E$, 
$$m_{\YY}: E\times \YY\to \YY$$
by:
$$(x, (s,e))\mapsto (s, e+x).$$
Denote by  $P:=P_{n}(\YY, (\beta,d))$. 
Let $t_x$ for $x\in E$ be the translation of $\YY$ determined by $x$.  Let 
$\iota: E\to E$ be the inverse map by $x\mapsto -x$, and let 
$$\Psi: E\times \YY\times P\stackrel{\iota\times \id_{\YY\times P}}{\longrightarrow}E\times \YY\times P
\stackrel{m_{\YY}\times \id_{P}}{\longrightarrow}\YY\times P
$$
be the composition.  We let 
$$\mathbb{I}=[\sO_{\YY\times P}\to \eF]$$ be the universal stable pair on $\YY\times P$.  Then 
$\Psi^*(\mathbb{I})$ defines a family of stable pairs on $\YY$ over $E\times P$.  By the universal property of $P$, this gives a map 
$$m_{P}: E\times P\to P$$
such that $\Psi^*(\mathbb{I})\cong m_{P}^*(\mathbb{I})$. 
The map $m_{P}$ defines a group action on $P$. Points $x\in E$ act on $I\in P$ by
$$I+x:=m_P(x,I)=t_{-x}^*(I)=t_{x*}(I).$$
Let 
$$m_{\YY\times P}: E\times \YY\times P\to \YY\times P$$
be the diagonal action by
$$(e,x,I)\mapsto (x+e, I+e).$$ Then $\Psi^*(\mathbb{I})\cong m_{P}^*(\mathbb{I})$ gives 
$$m_{\YY\times P}^*(\mathbb{I})= \pi_{E}^*(\mathbb{I})$$
where $\pi_E: E\times \YY\times P\to \YY\times P$ is the projection.  These data satisfy the cocycle condition which descends to the quotient:
$$\rho: \YY\times P\to (\YY\times P)/E$$
with the diagonal action $m_{\YY\times P}$.  The universal stable pair $\mathbb{I}$ descends along $\rho$ and gives $\overline{\mathbb{I}}$ on the quotient 
$ (\YY\times P)/E$ and 
$\rho^*(\overline{\mathbb{I}})\cong \mathbb{I}$. 

\subsection{The reduced perfect obstruction theory}\label{subsec_reduced_POT}

For the K3 surface $S$, the $\mu_r$-gerbe $\SS\to S$ is a holomorphic symplectic DM stack.  We follow Oberdieck \cite{Oberdieck} and Kool-Thomas \cite{KT} to construct a reduced perfect obstruction theory
on $P$.  Let $\pi_{\YY}: \YY\times P\to \YY$ and 
$\pi_{P}: \YY\times P\to P$ be projections. 
We consider the Atiyah class:
$$\At_{\YY}(\mathbb{I})\in \Ext^1_{\YY\times P}(\mathbb{I}, \mathbb{I}\otimes \ll_{\YY\times P})$$
in \cite{Illusie}.  Since 
$\ll_{\YY\times P}\cong \pi_{\YY}^*(\Omega_{\YY})\oplus \pi_{P}^*(\ll_{P})\to \pi_{P}^*\ll_{P}$, we use the composition maps
\begin{multline*}
\Ext^1_{\YY\times P}(\mathbb{I}, \mathbb{I}\otimes \ll_{\YY\times P})\to  \Ext^1_{\YY\times P}(R\Hom(\mathbb{I}, \mathbb{I}), \pi_{P}^*\ll_{P})
\to   \Ext^1_{\YY\times P}(R\Hom(\mathbb{I}, \mathbb{I})_0, \pi_{P}^*\ll_{P}) \\
\stackrel{\cong}{\longrightarrow}\Hom_{P}(R\pi_{P*}R\Hom(\mathbb{I},\mathbb{I})_0\otimes \omega_{P}[2], \ll_{P})
\end{multline*}
and the last isomorphism is from relative Verdier duality.  Then the image $\At_{P}(\mathbb{I})$ in the above image gives a perfect obstruction theory:
$$E^{\bullet}=R\pi_{P*}\left(R\Hom(\mathbb{I},\mathbb{I})_0\otimes \omega_{P}\right)[2]\to \ll_{P}.$$
As in \cite[\S 2.3]{KT}, we take cup product of the obstruction sheaf of $E^\bullet$ with the Atiyah class $\At_{\YY}(\mathbb{I})$ and get a semi-regularity map:
\begin{equation}\label{eqn_semi-regular}
\sr:  (E^{\bullet})^{\vee}\to H^{1,3}(\YY)\otimes \sO_{P}[-1]
\end{equation}
by 
\begin{multline*}
\Ext^2_{\pi_{P}}(\mathbb{I}, \mathbb{I})_0\hookrightarrow  \Ext^2_{\pi_{P}}(\mathbb{I}, \mathbb{I})\stackrel{\At(\mathbb{I})}{\longrightarrow}
 \Ext^3_{\pi_{P}}(\mathbb{I}, \mathbb{I}\otimes \ll_{\YY\times P}) 
\to \Ext^3_{\pi_{P}}(\mathbb{I}, \mathbb{I}\otimes \pi_{\YY}^*\Omega_{\YY}) \\
\stackrel{\tr}{\longrightarrow}R^3\pi_{P*}\pi_{\YY}^* \Omega_{\YY}\stackrel{\cong}{\longrightarrow}H^{1,3}(\YY)\otimes \sO_{P},
\end{multline*}
which is surjective since $\SS$ is a K3 gerbe. 
Taking the dual of the semi-regularity map $\sr$ we get:
$$H^{1,3}(\YY)^{\vee}\otimes \sO_{P}[1]\to E^{\bullet}$$
and let $E^{\bullet}_{\red}:=\Cone(H^{1,3}(\YY)^{\vee}\otimes \sO_{P}[1]\to E^{\bullet} )$, the cone the above morphism. 
Then from \cite[Theorem 2.7]{KT}: 

\begin{prop}\label{prop_reduced_POT}
There exists a perfect obstruction theory 
$$E^{\bullet}_{\red}\to \ll_{P}$$
for $P$ with virtual dimension $h^{1,3}(\YY)=h^{0,2}(\SS)$.
\end{prop}
\begin{proof}
We generalize the proof in \cite[Theorem 2.7]{KT} to the gerbe setting. 

\textbf{Step 1:} We study the twistor family for $\mu_r$-gerbes on a K3 surface $S$.  
First let us recall the twisted family for K3 surfaces.  Consider $H^1(T_S)$ and let $\mathfrak{m}$ be the maximal ideal of the origin 
$0\in H^1(T_S)$. The first order neighbourhood of the origin is 
$\spec(\sO_{H^1(T_S)}/\mathfrak{m}^2)$. The cotangent sheaf restricted to $0\in H^1(T_S)$  gives $H^1(T_S)^*$. Let 
$\overline{\sS}\to \spec(\sO_{H^1(T_S)}/\mathfrak{m}^2)$ be the tautological flat family of K3 surfaces.  As in \cite[\S 2.1]{KT}, there exists a subspace $V\subset H^1(T_S)$
such that 
$\cup_{\beta}: V\stackrel{\cong}{\longrightarrow}H^2(\sO_{S})$
is an isomorphism.  Then restricting $\overline{\sS}$ to $B:=\spec(\sO_{V}/\mathfrak{m}^2)$ gives the flat twistor family 
$\overline{\sS}_B$ of K3 surfaces. 
Since the twistor family  $\overline{\sS}_B\to B$ is a family over an affine scheme we have:
$$H^2( \overline{\sS}_B, \mu_r)\cong H^2(S,\mu_r).$$
Then the class of the gerbe $[\SS]\in H^2(S,\mu_r)$ determines a $\mu_r$-gerbe 
on $\overline{\sS}_B$. We denote $p: \sS_B\to \overline{\sS}_B$ for this $\mu_r$-gerbe whose central fiber gives 
$p: \SS\to S$.  This is the twistor family for the K3 gerbe $\SS$. 
We also let 
$$\overline{\YY}_B:= \overline{\sS}_B\times E; \quad  \YY_B:= \sS_B\times E$$
be the families of $\SS\times E$ over $B$. 

\textbf{Step 2:}  Let
\[
\xymatrix{
\SS\ar@{^{(}->}[r]^{j}\ar[d] & \sS_B\ar[d]\\
S\ar@{^{(}->}[r]^{j}& \overline{\sS}_B;
}
\quad 
\xymatrix{
\SS\times E\ar@{^{(}->}[r]^{j}\ar[d] & \YY_B\ar[d]\\
S\times E\ar@{^{(}->}[r]^{j}& \overline{\YY}_B;
}
\]
be the inclusions of the central fibers. 
We claim that the natural morphisms 
$$j_*: P_n(\SS,\beta)\stackrel{\cong}{\longrightarrow} P_n(\sS_B/B, \beta_B)$$
and 
$$j_*: P_n(\YY,(\beta,d))\stackrel{\cong}{\longrightarrow} P_n(\YY_B/B, (\beta_B,d))$$
are isomorphisms, where 
$\beta_B=\beta\otimes 1\in H^2(\sS_B/B)\cong H^2(\SS,\cc)\otimes \sO_B$.
We follow the proof in \cite[Proposition 2.2]{KT}. First 
$P_n(\sS_B,\beta_B)\otimes_{B}\{0\}\cong P_n(\SS,\beta)$. We need to prove that if there is an Artinian scheme 
$A$ with a morphism to $B$, and proper flat family 
$\mathbb{I}_A\to A$ of stable pairs over $\sS_B$ such that it pulls back to a stable pair 
$\mathbb{I}_0$ over $\SS$ and $h_{0*}[\beta]=\beta_B$, then $A\to B$ factors through $0\in B$. But this is from 
\cite[Lemma 2.1]{KT} since we fix $\beta\in H^2(\SS,\qq)$ such that 
$\overline{\beta}\in H^{1,1}(S)\cap (H^2(S,\zz)/\tor)$. The case of $\YY$ and $\YY_B$ is similar. 

\textbf{Step 3:}  We then generalize \cite[Theorem 2.7]{KT} to the gerbe case and get the perfect obstruction theory. 
Consider the algebraic twistor family 
$\sS_B\to B$ and the family  $\YY_B=\sS_B\times E\to B$, we have 
the moduli stack of stable pairs 
$P:=P_n(\YY_B/B, \iota_*\beta_B)\to B$ on the fibers is isomorphic to the moduli stack 
$P_n(\YY, \beta)$ on $\YY=\SS\times E$. 
For the stack $\YY_B$ over $B$, since for  the stable pair 
$$I^{\bullet}\in P_n(\YY_B/B, \iota_*\beta_B)$$
the deformation and obstruction are given by 
$\Ext^1_{\YY_B}(I^\bullet, I^\bullet)_0$ and $\Ext^2_{\YY_B}(I^\bullet, I^\bullet)_0$, the argument in \cite[\S 3]{MPT} works for this 
$\mu_r$-gerbe case and gives a relative perfect obstruction theory:
$$E^{\bullet}_{\rel}:=
R\pi_{P*}(R\Hom(\mathbb{I}, \mathbb{I})_0\otimes\omega_{P\times_{B}\YY_B/P})[2]\to \ll_{P/B}.
$$
This perfect obstruction theory fits into the following diagram:
\[
\xymatrix{
F^\bullet\ar[r]\ar[d]& E^{\bullet}_{\rel}\ar[r]\ar[d] &\Omega_{B}[1]\ar[d]^{=}\\
\ll_P\ar[r]& \ll_{P/B}\ar[r]&  \Omega_{B}[1]
}
\]
and defines a perfect absolute obstruction theory $F^\bullet$ over $P$.  All the stable pairs of $P$ lie scheme theoretically on the central fiber $\YY$, and 
$E^{\bullet}_{\rel}$ is the usual complex of stable pair theory on $\YY$. Hence 
$F^{\bullet}$ has virtual dimension $h^{2,0}(\SS)$. The same proof in  \cite[Theorem 2.7]{KT} shows that 
$F^{\bullet}\to E^{\bullet}\to  E^{\bullet}_{\red}$ is an isomorphism. 
\end{proof}

\subsection{Symmetric obstruction theory of the quotient}

\subsubsection{Global vector field on $P$ by the $E$-direction on $\YY$}

Let $m:E\times P\to P$ be the action, then we have 
the complex 
$T_E\otimes \sO_{E\times P}\hookrightarrow T_{E\times P}\stackrel{dm}{\longrightarrow}m^*T_P$ where $dm$ is the differential. We restrict it to 
$0_E\times P\hookrightarrow E\times P$ and get a global vector field:
\begin{equation}\label{eqn_vector_field}
\textbf{v}: H^0(T_E)\otimes \sO_P=T_{E,0_E}\otimes \sO_P\to T_P\cong \sE xt^{1}_{\pi_P}(\mathbb{I}, \mathbb{I})_0. 
\end{equation}
Similar arguments in \cite[Lemma 1]{Oberdieck} gives:
$$
H^0(T_E)\otimes\sO_P \stackrel{\cong}{\longrightarrow} H^0(\YY, T_{\YY})\otimes \sO_P\stackrel{\At_{\YY}(\mathbb{I})}{\longrightarrow}
\Ext^1(\mathbb{I}, \mathbb{I})\otimes\sO_P\to  \sE xt^{1}_{\pi_P}(\mathbb{I}, \mathbb{I})
$$
which is the same as (\ref{eqn_vector_field}). 
Since we have $H^0(E,T_E)=H^0(\YY, T_{\YY})$, let 
$$s: H^0(T_E)^{\vee}\otimes\sO_P\stackrel{\cong}{\longrightarrow} H^{1,3}(\YY)\otimes\sO_P$$
be the isomorphism induced by the non-degenerate pairing:
$$H^0(\YY,T_{\YY})\otimes H^3(\YY, \Omega_{\YY})\to H^3(\YY, \sO_{\YY})\cong \cc.$$
Then from  \cite[Lemma 2]{Oberdieck}, we have the following commutative diagram:
\[
\xymatrix{
E^{\bullet}\ar[r]^--{\partial^{\vee}}\ar[d]_{\theta}& H^0(T_E)^{\vee}\otimes \sO_P\ar[d]^{s}\\
(E^{\bullet})^{\vee}[1]\ar[r]^-{\sr[1]}& H^{1,3}(\YY)\otimes \sP_P
}
\]
where $\partial^{\vee}: E^{\bullet}\to h^0(E^{\bullet})=\Omega_P\to 
H^0(T_E)^{\vee}\otimes \sO_P$ is the composition map $E^{\bullet}\to h^0(E^{\bullet})$ with the dual of (\ref{eqn_vector_field})
and $\sr$ is the semi-regularity map. 

\subsubsection{Symmetric complex of the quotient}

We form the complex 
$$\sI^{\bullet}:=\Cone(E^{\bullet}\stackrel{\partial^{\vee}}{\longrightarrow} H^0(T_E)^{\vee}\otimes \sO_P)[-1]. $$
The reduced complex $E_{\red}^{\bullet}$ fits into the following:
$$H^{1,3}(\YY)^{\vee}\otimes \sO_P\stackrel{\sr^{\vee}}{\longrightarrow} E^{\bullet}\to E^{\bullet}_{\red}.$$
Consider the following diagram of exact triangles:
\[
\xymatrix{
\sI^{\bullet}\ar[r]\ar[d]^{\widetilde{\theta}}& E^{\bullet}\ar[d]^{\theta}\ar[r]^--{\partial^{\vee}}& H^0(T_E)^{\vee}\otimes \sO_P\ar[d]^{s}\\
(E^{\bullet}_{\red})^{\vee}[1]\ar[r]& (E^{\bullet})^{\vee}[1]\ar[r]& H^{1,3}(\YY)\otimes\sO_P
}
\]
Since $\theta$ and $s$ are isomorphisms, so is $\widetilde{\theta}$. We form the following diagram:
\[
\xymatrix{
H^{1,3}(\YY)^{\vee}\otimes\sO_P[1]\ar[r]^{=}\ar[d]^{\sr^{\vee}}& H^{1,3}(\YY)^{\vee}\otimes\sO_P[1]\ar[r]^{}\ar[d]^{\sr^{\vee}}&0\ar[d]^{}\\
\sI^{\bullet}\ar[r]\ar[d]^{}& E^{\bullet}\ar[d]^{}\ar[r]^--{\partial^{\vee}}& H^0(T_E)^{\vee}\otimes \sO_P\ar[d]^{=}\\
G^{\bullet}\ar[r]& E_{\red}^{\bullet}\ar[r]^--{\partial^{\vee}}& H^0(T_E)^{\vee}\otimes \sO_P
}
\]
where 
$G^{\bullet}:=\Cone(H^{1,3}(\YY)^{\vee}\otimes\sO_P[1]\to \sI^{\bullet})$ or 
$G^{\bullet}:=\Cone(E_{\red}^{\bullet}\stackrel{\partial^{\vee}}{\longrightarrow} H^0(T_E)^{\vee}\otimes \sO_P[-1])$. 
Then the above induces the following diagram:
\[
\xymatrix{
H^{1,3}(\YY)^{\vee}\otimes\sO_P[1]\ar[r]^--{\sr^{\vee}}\ar[d]^{s^{\vee}[1]}& \sI^{\bullet}\ar[r]\ar[d]^{\widetilde{\theta}}& G^{\bullet}\ar[d]^{\lambda}\\
H^0(T_E)\otimes \sO_{P}[1]\ar[r]^--{\partial}& (E^{\bullet}_{\red})^{\vee}[1]\ar[r]& (G^{\bullet})^{\vee}[1]
}
\]
where $\lambda$ is induced by the morphism in the left square. Then from \cite[Proposition 2]{Oberdieck}, $\lambda^{\vee}[1]\cong \lambda$ and 
$\lambda: G^{\bullet}\to (G^{\bullet})^{\vee}[1]$ is a non-degenerate symmetric bilinear form of degree one. 

\subsubsection{The morphism $q: P\to E$}

Since we work on $\SS\times E=\YY$, and $P=P_n(\YY, (\beta,d))$ is the moduli stack.  Let 
$\beta\in \Pic(\SS)$ be such that 
$\overline{\beta}\in \Pic(S)$ is a curve class.  We generalize the morphism $p_c$ in \cite[\S 2.2]{Oberdieck} (which we call it $q_c$) to this case. 
First we have for each $k\in \zz$, 
\begin{equation}\label{eqn_useful_morphism1}
\sigma_k: E\times E\to E
\end{equation}
given by $(x,e)\mapsto e+kx$ which is the action of $E$ on itself by $k$-times. Also 
\begin{equation}\label{eqn_useful_morphism2}
m_P: E\times P\to P
\end{equation}
is given by 
$(x,I)\mapsto I+x=t_{x*}I$.  We recall the construction of the morphism 
\begin{equation}\label{eqn_key_morphism}
q_c: P_n(\YY, (\beta, d))\to E
\end{equation}
We let $\widehat{E}:=\Pic^0(E)$ be the dual of $E$ which is isomorphic to $E$. 
Let $L\to E\times \widehat{E}$ be the Poincar\'e line bundle which is defined by 
\begin{enumerate}
\item $L_{\xi}=L|_{E\times\xi}$ is isomorphic to $\xi$ for any $\xi\in \Pic^0(E)$;
\item $L|_{0\times \widehat{E}}\cong \sO_{\widehat{E}}$. 
\end{enumerate}
So $L_X=L|_{x\times \widehat{E}}$ is isomorphic to $x\in E$ since $\Pic^0(\widehat{E})\cong E$. 
Let 
$\widehat{\YY}=\SS\times \widehat{E}$ and $L_{\YY}\to \YY\times\widehat{\YY}$ be the pullback of $L$ from 
$\YY\times\widehat{\YY}\to E\times\widehat{E}$. Let 
$$\Phi_{L_{\YY}}: D^b(\YY)\to D^b(\widehat{\YY})$$
be the equivalence defined by:
$$E\mapsto Rp_{\widehat{\YY}*}(L_{\YY}\otimes p_{\YY}^*(E)).$$
This is the Fourier-Mukai transform. For a line bundle $\sL\in \Pic(\SS)$, such that 
$c_1(\sL)=c$. Let 
$$\pi_{\SS}: \widehat{\YY}\to \SS$$
be the projection. Define
$$(\cdot)\otimes \pi_{\SS}^*\sL: D^b(\widehat{\YY})\to D^b(\widehat{\YY})$$
to be:
$$E\mapsto E\otimes \pi_{\SS}^*\sL.$$
Let $\pi_{\widehat{E}}: \widehat{\YY}\to \widehat{E}$ be the projection to  $\widehat{E}$. We define $q_c$ in the level of geometric points:
\begin{multline*}
P(\cc)=P_n(\YY, (\beta,d))(\cc)\hookrightarrow D^b(\YY)
\stackrel{\Phi_{L_{\YY}}}{\longrightarrow} D^b(\widehat{\YY})\stackrel{(\cdot)\otimes \pi_{\SS}^*\sL}{\longrightarrow}
D^b(\widehat{\YY})\stackrel{R\pi_{\widehat{E}*}}{\longrightarrow} D^b(\widehat{E})\stackrel{\det}{\longrightarrow}
\Pic^m(\widehat{E})=E
\end{multline*}
for some $m$. 

The algebraic morphism $q_c$ can be constructed by giving a line bundle on 
$P_n(\YY, (\beta,d))\times \widehat{E}$. The line bundle is given by the determinant of the following
$$R\pi_{\widehat{E}\times P *}(\Phi_{L_{\YY\times P}}(\mathbb{I})\otimes \pi_{\SS}^*\sL)$$
where $\mathbb{I}$ is the universal twisted stable pair on $\YY\times P$. 
We have a similar result as in \cite[Proposition 3]{Oberdieck}.

\begin{prop}\label{prop_morphism_q}
There exists a morphism
$$q_c: P_n(\YY, (\beta,d))\to E$$
such that for $k=\langle c,\beta\rangle+n$, it is $E$-equivariant with respect to $m_{P}$ and $\sigma_k$ defined before. 
Moreover, for $I\in P$, $x\in E$, $q_c(I+x)=q_c(I)+kx$.
\end{prop}
\begin{proof}
Since we consider the twisted stable pairs 
$I=(\Xi\otimes \sO_{\YY}\to F)$, where our $\mu_r$-gerbe twist means the $\mu_r$-action on the generating sheaf is given by the whole primitive action of the generator of $\mu_r$.
Then all the calculations in \cite[Lemma 4, Lemma 5, Lemma 6]{Oberdieck} keep the same except we need to multiply by  
$e^{2\pi i/r}$ on the Chern character. Therefore the calculation in \cite[Lemma 6]{Oberdieck} still holds for twisted stable pairs. Also 
for $c\in\Pic(\SS)$, we have $\langle c,\beta\rangle+n=\langle \overline{c}, \overline{\beta}\rangle+n$, where 
$\overline{c}$ and $\overline{\beta}$ are the images in $H^2(\SS,\qq)\to H^2(S,\qq)$. 
\end{proof}

\subsubsection{Virtual class}

Consider the morphism $q_c: P\to E$. We will construct a perfect obstruction theory on
$$K:=q_c^{-1}(0_E).$$ 
Let 
$0_E\times P\hookrightarrow E\times P$ be the inclusion, then the derivative map:
$$d: \ll_P\to \Omega_{E,0}\otimes \sO_P$$
is the second one in $\pi^*\ll_{P/E}\to \ll_P\to \ll_{\pi}\cong \sO_P$ induced from $\pi: P\to P/E$.  From the analysis of obstruction theory above we form the following diagram:
\[
\xymatrix{
G^\bullet\ar[r]\ar[d]^{\varphi}& E^{\bullet}_{\red}\ar[r]\ar[d]& \Omega_{E,0}\otimes\sO_P\ar[d]^{\cong}\\
\pi^*\ll_{P/E}\ar[r]& \ll_P\ar[r]& \Omega_{E,0}\otimes\sO_P
}
\]
where $h^0(T_E)^{\vee}=\Omega_{E,0}$. Then it indues a morphism:
$$\varphi: G^{\bullet}\to \pi^*P/E$$
As in \cite[Proposition 4]{Oberdieck}, let 
$\iota: L\hookrightarrow P$ be the inclusion, then 
$L\iota^*(\varphi): L\iota^*G^{\bullet}\to L\iota^*\pi^*\ll_{P/E}\cong \widetilde{\pi}^*(\ll_{P/E})$
where $\widetilde{\pi}: K\to P\to P/E$; and $\widetilde{\pi}$ is \'etale. 
Hence $\widetilde{\pi}^*(\ll_{P/E})\cong\ll_{K}$ and the composition 
$$L\iota^*G^{\bullet}\stackrel{L\iota^*(\varphi)}{\longrightarrow}\ll_K$$
defines a symmetric obstruction theory on $K$.  Thus this symmetric obstruction theory gives a virtual fundamental class 
$[K]^{\vir}\in A_0(K)$ such that 
$$\frac{1}{|G|}\pi_*([K]^{\vir})=[P/E]^{\vir}\in A_0(P/E)$$
where $G$ is the finite group of the \'etale map $\widetilde{\pi}$.

Now we show Oberdieck's theorem in \cite{Oberdieck} for $\mu_r$-gerbe $\SS$ and $\YY=\SS\times E$. The reduced perfect obstruction theory 
$E^{\bullet}_{\red}$ defined before gives a one-dimensional virtual fundamental cycle 
$[P]^{\red}\in A_1(P)$. 
Let $\omega\in H^2(E,\zz)$ be the class of a point and $\beta^{\vee}\in H^2(\SS,\qq)$ such that $\langle \beta, \beta^{\vee}\rangle=1$. 
\begin{defn}\label{defn_reduced_invariant_YY}
Define
$$\widetilde{N}^{\YY}_{n,(\beta,d)}:=\int_{[P^{\red}]}
\tau_0(\pi_{\SS}^*(\beta^{\vee})\cup \pi_{E}^*(\omega))$$
where $\tau_0(\cdot)$ is the insertion operator defined in \cite[\S 3.6]{PT}, and 
$\pi_{\SS}: \SS\times E\to \SS$, $\pi_E: \SS\times E\to E$ are projections. This is called the incidence DT-invariant. 
\end{defn}

There is another invariant given by the Behrend function \cite{Behrend}. The elliptic curve $E$ acts on $P$ and we have the quotient 
$P/E$. Let $\nu_P: P/E\to \zz$ be the Behrend function on the quotient.  

\begin{defn}\label{defn_reduced_invariant_YY_Behrend}
Define
$$N^{\YY}_{n,(\beta,d)}:=\chi(P/E,\nu_P)$$
the weighted Euler characteristic by the Behrend function. 
\end{defn}

\begin{thm}[\cite{Oberdieck}]\label{thm_oberdieck_reduced_Behrend}
We have:
$$\widetilde{N}^{\YY}_{n,(\beta,d)}=N^{\YY}_{n,(\beta,d)}.$$
\end{thm}
\begin{proof}
Let 
$\pi: P\to P/E$ be the projection. We first show that 
$\pi^*[P/E]^{\vir}=[P^{\red}]$.  We use the diagram
\[
\xymatrix{
E\times K\ar[r]^{\pi_K}\ar[d]_{m}& K\ar[d]^{\widetilde{\pi}}\\
P\ar[r]^{\pi}& P/E
}
\]
From the construction of reduced obstruction theory $E^{\bullet}_{\red}$ (which differs from $E^{\bullet}$ by $\sO_P$), we have:
$$[P]^{\red}=\{s((E^{\bullet})^{\vee})\cap c_F(P)\}_{1}$$
where $s((E^{\bullet})^{\vee})$ is the Segre class and $c_F(P)$ is the Fulton class of $P$. Here we use the description of virtual class of Siebert \cite{Siebert} using Fulton class. 
The map $m$ is \'etale, and $\deg(\widetilde{\pi})=|G|$, we have:
$$[P]^{\red}=\frac{1}{|G|}m_{*}
\{s((m^*E^{\bullet})^{\vee})\cap m^*c_F(P)\}_{1}.$$
Form the following diagram:
\[
\xymatrix{
\YY\times P\ar[r]^{\pi_P}\ar[d]_{\rho}& P\ar[d]^{\pi}\\
(\YY\times P)/E\ar[r]^{\quad\pi_{P/E}}& P/E
}
\]
there exists a universal pair $\overline{\mathbb{I}}\in (\YY\times P)/E$ such that $\rho^*\overline{\mathbb{I}}=\mathbb{I}$, and $\mathbb{I}$ is the universal stable pair 
on $\YY\times P$. Let 
$$\hH^{\bullet}:=R\pi_{P/E*}R\Hom(\overline{\mathbb{I}}, \overline{\mathbb{I}})_0[2].$$
Then $\hH^{\bullet}$ is of amplitude $[-1,0]$ and 
$$\pi^*\hH^{\bullet}=E^{\bullet}=R\pi_{P*}R\Hom(\mathbb{I}, \mathbb{I})_0[2].$$
From Fulton's Chern class explanation of virtual classes, if we let $h: Y\to P/E$ be any proper \'etale morphism of $\deg(h)$ with $Y$ a scheme, we have: 
$$[P/E]^{\vir}=\frac{1}{\deg(h)}h_*\{s((h^*\hH^{\bullet})^{\vee})\cap c_F(Y)\}_0.$$

Thus since Fulton's Chern classes are invariant under \'etale pullback,  and 
$m^*E^{\bullet}=m^*\pi^*\hH^{\bullet}=\pi_{K}^*\widetilde{\pi}^*\hH^{\bullet}$, we have 
$$[P]^{\red}=\frac{1}{|G|}m_*\{s((\pi_K^*\widetilde{\pi}^*\hH^{\bullet})^{\vee})\cap \pi_K^* c_F(K)\}_1$$
$$=\frac{1}{|G|}m_*\pi_K^*\{s((\widetilde{\pi}^*\hH^{\bullet})^{\vee})\cap \pi_K^* c_F(K)\}_0
=\pi^*[P/E]^{\vir}.$$
From the definition of the invariants in Definition \ref{defn_reduced_invariant_YY} and Definition \ref{defn_reduced_invariant_YY_Behrend}, and look at the diagram:
\[
\xymatrix{
\YY\ar[r]^{\pi_{\YY}}& \YY\times P\ar[r]^{\pi_P}\ar[d]_{\rho}& P\ar[d]^{\pi}\\
&(\YY\times P)/E\ar[r]^{\quad \pi_{P/E}}& P/E
}
\]
let 
$D\in H^2(\SS,\qq)$ be such that $\langle \beta, D\rangle=1$, then we have 
$\widetilde{N}^{\YY}_{n, (\beta,d)}$ is the degree of  
$$\textbf{D}:=(-\Ch_2(\mathbb{I})\cdot \pi_{\YY}^*(\pi_{\SS}^*(D)\cup \pi_E^*(\omega))\cap \pi_{P}^*[P]^{\red}.$$
We need to show $(\pi_{P/E}\circ \rho)_*\textbf{D}=[P/E]^{\vir}$.
We have 
$\Ch_2(\mathbb{I})=\rho^*\Ch_2(\overline{\mathbb{I}})$, therefore 
$\rho_*\textbf{D}=\rho_*(\pi_E^*(\omega)\cap\rho^*(\alpha))$ for 
$\alpha=(-\Ch_2(\overline{\mathbb{I}})\cup\pi_{\SS}^*(D))\cap \pi_{P/E}^*[P/E]^{\vir}$.
Then the same argument in \cite[\S 4.4]{Oberdieck} shows that 
$\rho_*\textbf{D}=\alpha$, and 
$(\pi_{P/E}\circ\rho)_*\textbf{D}=\pi_{P/E*}\alpha=\langle D, \beta\rangle [P/E]^{\vir}
=[P/E]^{\vir}$.
\end{proof}

\section{Behrend equals to Reduced invariants of Oberdieck}\label{sec_Behrend_Oberdieck}

In this section we perform the following invariants on the Calabi-Yau threefold DM stack 
$\XX:=\SS\times \cc$ for the $\mu_r$-gerbe $\SS\to S$ over a K3 surface $S$. 

For the Calabi-Yau DM stack $\XX=\SS\times \cc$, we choose a generating sheaf  on $\XX$ such that it is the pullback of the generating sheaf $\Xi$ on $\SS$ by the projection $\XX\to \SS$. 
We still denote this generating sheaf by $\Xi$. 
We consider the stable pairs on $(\sO_{\XX}\otimes\Xi\stackrel{s}{\rightarrow}F)$ on $\XX$. 
For $n\in \qq$, we let $P_n(\XX,\iota_*\beta)=P_n(\SS\times\cc, \iota_*\beta)$, where  
$\iota: \SS\hookrightarrow \XX$ is the inclusion,  be the moduli stack of stable pairs on $\XX$ with number
$n$ given by 
$$n=\int_{\SS\times\pp^1}\Ch(F)\cdot \Td_{\SS\times\pp^1}$$
and class $[F]=\iota_*\beta$. 
Since $\XX$ is also a Calabi-Yau DM stack, $P_n(\XX,\iota_*\beta)$ admits a symmetric obstruction theory and there is virtual fundamental class 
$[P_n(\XX,\beta)]^{\vir}$.  
There is a $\cc^*$-action on $\XX$ by scaling the fiber $\cc$ and induces an action on the moduli stack $[P_n(\XX,\beta)]^{\vir}$.  From \cite{GP}, there is an induced virtual fundamental class on the fixed locus 
$(P_n(\XX,\beta))^{\cc^*}$.  
\begin{defn}\label{defn_red_invariant_XX}
$$P_{n,\beta}^{\red}(\XX):=\int_{[(P_n(\XX,\beta))^{\cc^*}]^{\vir}}\frac{1}{e(N^{\vir})}\in\qq$$
where $e(N^{\vir})$ is the Euler class of the virtual normal bundle of the $\cc^*$-action.
\end{defn}

On the other hand, we have the Behrend function 
$$\nu_{P}: P_n(\XX,\iota_*\beta)\to \zz$$
on the moduli stack and the invariant 
$\chi(P_n(\XX,\iota_*\beta), \nu_P)$ is Behrend's weighted Euler characteristic.  Similar to \cite{MT}, 
let us form the generating function:
\begin{equation}\label{eqn_generating_function1}
\Z_{P}^{\red}(\XX; q, v):=\sum_{\alpha=(\beta,n)}P_{n,\beta}^{\red}(\XX)q^n v^{\beta}
\end{equation}
which is the generating function of rational residue reduced stable pair invariants of $\XX$. 

We also have
\begin{equation}\label{eqn_generating_function2}
\Z_{P}^{\chi}(\XX; q, v):=\sum_{\alpha=(\beta,n)}\chi(P_n(\XX,\iota_*\beta), \nu_P)q^n v^{\beta}
\end{equation}
which is the generating function of Behrend invariants of $\XX$. 
Our goal in this section is to show a similar result in \cite[Theorem 1.1]{MT}.

First we prove
\begin{lem}\label{lem_quotient_generating_function}
Let us form the generating function 
$$Z_{P}^{\chi}(\YY/E; q, v):=\sum_{\alpha=(\beta,n)}\chi(P_n(\YY,\iota_*\beta)/E, \nu_P)q^n v^{\beta},$$
where $\YY=\SS\times E$ and we consider the quotient of the moduli stack of stable pairs.  Then we have:
$$\Z_{P}^{\red}(\YY/E; q, v)=-\log(1+\Z_{P}^{\chi}(\XX; q, v)).$$
\end{lem}
\begin{proof}
The proof is the same as in \cite[Proposition 3.2]{MT}, since in that proof  the key points only happen at the elliptic curve 
$E$. We sketch it here. 

By the $\cc^*$-action on $\XX=\SS\times \cc$, the induced action on $P_n(\XX,\iota_*\beta)$ preserves the Behrend function $\nu_P$
and let $P^0_{\alpha}\subset P_n(\XX,\iota_*\beta)$ be the fixed point locus, which consists of stable pairs with set-theoretic supports on 
$\SS\times\{0\}\subset \SS\times\cc$.  Therefore the weighted Euler characteristic $\chi(P_n(\XX,\iota_*\beta), \nu_P)$ can be localized to 
$P_{\alpha}=\chi(P_{\alpha}^0, \nu_P|_{P_{\alpha}^0})$. 

For the Calabi-Yau DM stack $\YY=\SS\times E$, the exponential map $e^{(-)}$ gives an isomorphism on an analytic neighborhood of any point $p\in E$ and a neighborhood of $0\in\cc$. Then this translates 
stable pairs from $\SS\times\{0\}\subset \XX$ to 
$\SS\times\{p\}\subset \YY$, and 
$P_{\alpha}^0\times E$ is the moduli stack of stable pairs on $\YY$ supported set theoretically on a single fiber $\SS\times\{p\}$.
Then we stratify the moduli stack $P_n(\YY,\iota_*\beta)$ by the minimal number $k$ of fibers $\SS$ on which the stable pairs are set-theoretically supported.  Let the distinct charges be 
$\alpha_1,\cdots, \alpha_l$.  Let $k_i$ be the number of fibers $\SS$ such that the charge is $\alpha_i=(\beta_i, n_i)$, then 
$$\sum_{i=1}^{l}k_i=k; \quad  \sum_{i=1}^{l}k_i\alpha_i=\alpha.$$
Then the same argument in \cite[Proposition 3.2]{MT} shows that the stratum of $P_n(\YY,\iota_*\beta)$ with this data is:
\begin{equation}\label{eqn_lem1}
\left(P^0_{\alpha_1}\times\cdots\times P^0_{\alpha_1}\right)\times\cdots\times 
\left(P^0_{\alpha_l}\times\cdots\times P^0_{\alpha_l}\right)\times \left(E^k\setminus \Delta_k\right)/(S_{k_1}\times\cdots\times S_{k_l})
\end{equation}
where the action acts freely and $\Delta_k$ is the big diagonal. 
Therefore the weighted Euler characteristic of (\ref{eqn_lem1}) is:
$$-\frac{P_{\alpha_1}^{k_1}}{k_1!}\cdot \frac{P_{\alpha_2}^{k_2}}{k_2!}\cdots \frac{P_{\alpha_l}^{k_l}}{k_l!}
\chi\left(\frac{E^k\setminus \Delta_k}{E}\right)=
(-1)^k\frac{1}{k}\mat{c} k\\
k_1,k_2,\cdots, k_l\rix P_{\alpha_1}^{k_1}\cdots P_{\alpha_l}^{k_l}.
$$
Then summing over all strata and all $\alpha$, we have:
\begin{align*}
&\sum_{k=1}^{\infty}\sum_{\substack{l, k_i, \alpha_i  \text{distinct}\\
\sum_{i=1}^{l}k_i=k}} \frac{(-1)^k}{k}\mat{c} k\\
k_1,k_2,\cdots, k_l\rix (P_{\alpha_1} Q^{\alpha_1})^{k_1}\cdots  (P_{\alpha_l} Q^{\alpha_l})^{k_l} \\
&=\sum_{k=1}^{\infty}\frac{1}{k}\left(-\sum_{\alpha}P_{\alpha}Q^{\alpha}\right)^{k}\\
&=-\log (1+\sum_{\alpha}P_{\alpha}Q^{\alpha}).
\end{align*}
\end{proof}

\begin{thm}\label{thm_red_Behrend_MT}
We have
$$\Z_{P}^{\red}(\XX; q, v)=-\log(1+\Z_{P}^{\chi}(\XX; q, v)).$$
\end{thm}
\begin{proof}
From Lemma \ref{lem_quotient_generating_function}, we need to prove:
$$\Z_P^{\chi}(\YY/E; q,v)=\Z_P^{\red}(\XX; q,v).$$
We show that 
$$P_{n,\beta}^{\red}(\XX)=\chi(P/E, \nu_P)=N_{n,\beta}^{\YY},$$
where $P:=P_n(\YY, \iota_*\beta)$.  From Theorem \ref{thm_oberdieck_reduced_Behrend}, $N_{n,\beta}^{\YY}=\widetilde{N}_{n,\beta}^{\YY}$.  Thus we need to show 
$$P_{n,\beta}^{\red}(\XX)=\widetilde{N}_{n,\beta}^{\YY}.$$
Here $\widetilde{N}_{n,\beta}^{\YY}=\widetilde{N}_{n,(\beta,0)}^{\YY}$ and 
\begin{equation}\label{eqn_formula_thm_1}
\widetilde{N}_{n,(\beta,0)}^{\YY}=\int_{[P_n(\YY, \iota_*\beta)]^{\red}}\tau_0(\iota_*\beta^{\vee})
\end{equation}
where we let $\beta^{\vee}\in H^2(\SS,\qq)$ such that $\int_{\SS}\beta\cup \beta^{\vee}=1$.
We take the degeneration of the elliptic  $E$ to a one-nodal rational elliptic curve. The degeneration formula 
\cite{LW}, \cite{Zhou} expresses the right side of (\ref{eqn_formula_thm_1}) as:
\begin{equation}\label{eqn_formula_thm_2}
\int_{[P_n(\SS\times\pp^1/(\SS_0\cup \SS_{\infty}),\iota_*\beta)]^{\red}}\tau_0(\beta^{\vee})
\end{equation}
We use the degeneration formula again to calculate (\ref{eqn_formula_thm_2}), and degenerate $\SS\times\pp^1/\SS_{\infty}$ to:
$$(\SS\times\pp^1)/\SS_{\infty}\bigcup_{\SS_{\infty}\sim \SS_0}(\SS\times\pp^1)/(\SS_0\cup\SS_{\infty})$$ and the degenerate formula gives:
\begin{equation}\label{eqn_formula_thm_3}
\int_{[P_n(\SS\times\pp^1/\SS_{\infty},\iota_*\beta)]^{\red}}\tau_0(\iota_*\beta^{\vee})
=
\sum_{(\beta_i, n_i)}\int_{[P_{n_1}(\SS\times\pp^1/\SS_{\infty},\iota_*\beta_1)\times 
P_{n_2}(\SS\times\pp^1/(\SS_0\cup \SS_{\infty}),\iota_*\beta_2)]^{\red}}1\times \tau_0(\iota_*\beta^{\vee})
\end{equation}
where the sum is over all $(\beta_1, n_1), (\beta_2, n_2)$ whose sum is $(\beta,n)$. Same analysis as in \cite[Theorem 4.2]{MT} shows that only the case 
$(\beta_2, n_2)=(\beta,n)$ and $(\beta_1, n_1)=0$ contributes to (\ref{eqn_formula_thm_2}). 
Therefore from  (\ref{eqn_formula_thm_3}), 
\begin{equation}\label{eqn_formula_thm_4}
\widetilde{N}_{n,(\beta,0)}^{\YY}=\int_{[P_n(\SS\times\pp^1/\SS_{\infty},\iota_*\beta)]^{\red}}\tau_0(\iota_*\beta^{\vee})
\end{equation}
Now we use virtual localization to the $\cc^*$-action on $\pp^1$ with weight $+1$ on tangent space at $0$. Lift $\iota_*\beta^{\vee}$ to $H^*_{\cc^*}(\SS\times\pp^1)$ by the map
$\iota: \SS\times\{0\}\hookrightarrow \SS\times\pp^1$.  Then localization formula reduces the integral (\ref{eqn_formula_thm_4}) to:
$$\int_{[P_n(\XX, \iota_*\beta)^{\cc^*}]^{\red}}\frac{1}{e(N^{\vir})}\tau_0(\iota_*\beta^{\vee}).$$
Since over $\SS_0$, $\tau_0(\iota_*\beta^{\vee})$ is 
$c_1(t)\int_{\SS}\beta\cup\beta^{\vee}=t$, we have 
$$\int_{[P_n(\XX, \iota_*\beta)^{\cc^*}]^{\red}}\frac{1}{e(N^{\vir})}=P_{n,\beta}^{\red}(\XX).$$
\end{proof}

\begin{rmk}
A similar formula as in Theorem \ref{thm_red_Behrend_MT} for  twisted Joyce-Song stable pair invariants also holds with the same proof 
in  \cite[Proposition 2.6]{Jiang-Tseng-K3}.
\end{rmk}

%%%%%%%%%%%%%%%%
\section{Wall crossing for  Calabi-Yau gerbes and Toda's multiple cover formula}\label{sec_KKV_Toda}

In this section we generalize Toda's wall crossing formula for the D0-D2-D6 bound states to $\mu_r$-gerbes and get the multiple cover formula for 
counting semistable coherent sheaves on the gerbe. 

\subsection{Wall-crossing in D0-D2-D6 bound states for $\mu_r$-gerbes}\label{subsec_D0-D2-D6}

In this section we let $\XX$ be a $\mu_r$-gerbe over a smooth projective Calabi-Yau threefold $X$.  We study the wall crossing formula of \cite[Formula (28)]{Toda_Kyoto}, and also \cite{Toda_JAMS} on the $\mu_r$-gerbe 
$\XX$.  We follow the techniques in Toda's papers. The arguments should work for any Calabi-Yau threefold DM stack $\XX$. 

\subsubsection{Toda's weak stability conditions on the category of D0-D2-D6 bound states}\label{subsubsec_Toda_stability)D0-D2-D6}

Let $\XX\to X$ be a $\mu_r$-gerbe over $X$, which is given by a class $[\XX]\in H^2(X,\mu_r)$.  Let 
$\Coh_{\leq 1}(\XX)$ be the subcategory of the category of coherent sheaves $\Coh(\XX)$ on $\XX$ consisting of coherent sheaves 
$E$ on $\XX$ such that 
$\supp(E)\leq 1$.  Let 
$$\sA_{\XX}:=\langle \sO_{\XX}, \Coh_{\leq 1}(\XX)[-1]\rangle_{\ex}.$$
Then $\sA_{\XX}$ is the heart of a bounded $t$-structure on a derived category $D_{\XX}$ (this is proved in \cite[Lemma 3.5]{Toda_JAMS} which works for the gerbe $\XX$):
$$\sD_{\XX}:=\langle \sO_{\XX}, \Coh_{\leq 1}(\XX)\rangle_{\tr}\subset D^b(\Coh(\XX)).$$ 
So $\sA_{\XX}$ is an abelian category. From physics, $\sD_{\XX}$ is called the category of D0-D2-D6 bound states for the gerbe $\XX$. 

\begin{rmk}
It is possible to apply stability conditions and wall crossing in \cite{Toda_JAMS}, \cite{Bridgeland_JAMS} to this triangulated category $\sD_{\XX}$ to prove the 
DT/PT-correspondence for the counting invariants for $\XX\to X$.  Since the decomposition conjecture of  DT and PT invariants for the $\mu_r$-gerbe $\XX\to X$ is proved in 
\cite{GT}, the DT/PT-correspondence is already known. 
\end{rmk}

Let us recall how we construct the category $\sA_{\XX}$. Let 
$\Coh_{\geq 2}(\XX)$ be the subcategory of $\Coh(\XX)$ given by 
$$\Coh_{\geq 2}(\XX):=\{E\in \Coh(\XX)| \Hom(\Coh_{\leq 1}(\XX), E)=0\}.$$
Then $(\Coh_{\leq 1}(\XX), \Coh_{\geq 2}(\XX))$ is a torsion pair in the sense of \cite[Definition 3.1]{Toda_JAMS} or \cite{HRS}. There is an abelian category 
$\Coh^{\dag}(\XX)$
with respect to this torsion pair and 
$$\sA^{p}_{\XX}:=\Coh^{\dag}(\XX)=\langle\Coh_{\geq 2}(\XX)[1], \Coh_{\leq 1}(\XX)\rangle_{\ex}.$$
The subcategory $\Coh_{\leq 1}(\XX)\subset \sA^{p}_{\XX}$ is closed under subobjects and quotients in  $\sA^{p}_{\XX}$. The category $\sA_{\XX}$ is the intersection
$$\sA_{\XX}=\sD_{\XX}\cap \sA^{p}_{\XX}[-1].$$
Let us construct the weak stability conditions on $\sD_{\XX}$ following Toda. 

\begin{defn}\label{defn_NS_group}
For the Calabi-Yau threefold DM stack $\XX\to X$, set 
$$\Gamma_0:=\qq\oplus H_2(\XX,\qq)$$
and the group homomorphism 
$$\cl_0: K(\Coh_{\leq 1}(\XX))\to \Gamma_0$$
is given by
$E\mapsto (\Ch_3(E), \Ch_2(E))$.   Define
$$\Gamma:=\qq\oplus H_2(\XX,\qq)\oplus\zz=\Gamma_0\oplus\zz.$$
\end{defn}

The general Chern character (not orbifold Chern character)
$$\cl: K(\sA_{\XX})\to \Gamma$$
is given by: $E\mapsto (\Ch_3(E), \Ch_2(E), \Ch_0(E))$ and defines a group homomorphism. This is because $\sA_{\XX}$ is generated by
$\sO_{\XX}$, $E\in\Coh_{\leq 1}(\XX)[-1]$, and by Poincare duality, 
$\Ch_3(E)\in H^6(\XX,\qq)$, $\Ch_2(E)\in H^4(\XX,\qq)\cong H_2(\XX,\qq)$.
We recall the weak stability condition in \cite{Toda_JAMS} and \cite{Toda_Kyoto}.
Let $\sA_{\XX}$ be an abelian category and $\Gamma$ be a finitely generated abelian group. 
Fix
$\hh\subset \cc$ such that 
$$\hh=\{\rho\cdot e^{\pi i \phi}| \rho>0, 0<\phi\leq 1\}.$$

\begin{defn}\label{defn_Bridgeland_stability}
A (Bridgeland) stability condition on $\sA_{\XX}$ is a group homomorphism 
$$\sZ: \Gamma\to\cc$$
satisfying the following conditions:
\begin{enumerate}
\item For any nonzero $E\in\sA_{\XX}$, $\sZ(E)=\sZ(\cl(E))\in\hh$ and the argument:
$\arg(\sZ(E))\in (0,\pi]$ is well-defined.  The object $E\in\sA_{\XX}$ is called $\sZ$-(semi)stable if for any nonzero subobject 
$0\neq F\subsetneq E$, $\arg\sZ(F)< (\leq)\arg\sZ(E)$.
\item For any object $E\in\sA_{\XX}$, there is a Harder-Narasimhan filtration:
$$0=E_0\subset E_1\subset\cdots\subset E_m=E$$
such that each $F_i=E_i/E_{i-1}$ is $\sZ$-semistable and 
$\arg\sZ(F_1)>\arg\sZ(F_2)>\cdots>\arg\sZ(F_m)$. 
\end{enumerate}
\end{defn}

Now for the group $\Gamma$, we fix a filtration:
$$0=\Gamma_{-1}\subsetneq \Gamma_0\subsetneq\Gamma_1\subsetneq\cdots\subsetneq\Gamma_m=\Gamma$$
such that the quotient $\Gamma_i/\Gamma_{i-1}$ is a free abelian group. 

\begin{defn}[\cite{Toda_JAMS}]\label{defn_weak_stability_Toda}
A weak stability condition on $\sA_{\XX}$ is given by 
$$\sZ=\{\sZ_i\}_{i=0}^{m}\in \prod_{i=0}^{m}\Hom(\Gamma_i/\Gamma_{i-1}, \cc)$$
such that the following conditions are satisfied:
\begin{enumerate}
\item If for any nonzero $E\in\sA_{\XX}$, $\cl(E)\in \Gamma_i/\Gamma_{i-1}$, then 
$$\sZ(E):=\sZ_i([\cl(E)])\in\hh$$
where 
$[\cl(E)]$ is the class of $\cl(E)$ in $\Gamma_i/\Gamma_{i-1}$. 
The objects $E\in \sA_{\XX}$ is $\sZ$-(semi)stable if for any exact sequence
$$0\to F\to E\to G\to 0$$
in $\sA_{\XX}$, and 
$$\arg\sZ(F)< (\leq) \arg\sZ(G).$$
\item There is a Harder-Narasimhan filtration for any $E\in \sA_{\XX}$.
\end{enumerate}
\end{defn}
It is not hard to see that if $m=0$, any weak stability condition is a stability condition. 
We follow Toda to construct the following weak stability condition on $\sA_{\XX}$. Recall our
$\Gamma=\Gamma_0\oplus \zz$, and we take the following 2-step filtration on $\Gamma$:
$$0=\Gamma_{-1}\subsetneq \Gamma_0\subsetneq \Gamma_1=\Gamma.$$
The embedding $\Gamma_0\hookrightarrow \Gamma$ is given by $(n,\beta)\mapsto (n,\beta,0)$.
So 
$$
\begin{cases}
\Gamma_0/\Gamma_{-1}=\qq\oplus H_2(\XX,\qq);\\
\Gamma_1/\Gamma_0=\qq.
\end{cases}
$$
We give the following data:
$$
\begin{cases}
\omega\in H^2(\XX,\qq)\cong H^2(X,\qq), \omega \text{~is ample};\\
0<\theta<1.
\end{cases}
$$
and let
$$\sZ_{\omega,\theta}=\left(\sZ_{\omega,\theta,0}, \sZ_{\omega,\theta,1}\right)\in \prod_{i=0}^{1}
\Hom(\Gamma_i/\Gamma_{i-1},\cc)=\Hom(\Gamma_0/\Gamma_{-1},\cc)\times\Hom(\Gamma_1/\Gamma_{0},\cc)$$
which is defined as:
$$
\begin{cases}
\sZ_{\omega,\theta,0}(n,\beta)=n-(\omega\cdot \beta)\sqrt{-1};\\
\sZ_{\omega,\theta,1}(\rk)=\rho\cdot e^{\pi i \theta},
\end{cases}
$$
where $(n,\beta)\in\Gamma_0$ and $\rho\in\zz$, $(n,\beta,\rk)\in \Gamma$.

\begin{lem}\label{lem_weak_stability_sAXX}(\cite[Lemma 5.2]{Toda_Kyoto})
The homomorphisms $(\sZ_{\omega,\theta,i})$ give a weak stability condition on the category $\sA_{\XX}$.
\end{lem}
\begin{proof}
The proof is similar to \cite[Lemma 5.2]{Toda_Kyoto}. For completeness, we provide a proof here. 
We check Definition \ref{defn_weak_stability_Toda}. Let $E\in\sA_{\XX}$ be such that 
$\cl(E)\in\Gamma_i/\Gamma_{i-1}$.  In the case $i=1$, 
$\sZ_{\omega,\theta}(E)\in \rr_{>0}e^{i\pi \theta}\subset \hh$ and in the case $i=0$, 
$\sZ_{\omega,\theta}(E)=\sZ_{\omega}(E[1])\in \hh$ since 
$E\in\Coh_{\leq 1}(\XX)[-1]$.  Here 
$\sZ_{\omega}(E)=n-(\omega\cdot \beta)\sqrt{-1}$.
Also $\sZ_{\omega,\theta}(E)$ defined in this way is a stability condition on $\Coh_{\leq 1}(\XX)$, see
\cite[Example 3]{Toda_Kyoto}. Therefore condition $(1)$  in  Definition \ref{defn_weak_stability_Toda} is satisfied. 

To  check the Harder-Narasimhan property, we introduce a torsion pair $(\sA_{1}^{p}, \sA_{1/2}^{p})$ on the category $\sA^p_{\XX}$, see \cite[Lemma 2.16]{Toda_Duke}, which is defined by
$$\sA_1^{p}:=\langle F[1], \sO_{x}| F \text{~is pure 2-dimensional~}, x\in \XX\rangle_{\ex}$$
and 
$$\sA_{1/2}^{p}:=\langle E\in \sA^p_{\XX}| \Hom(F, E)=0 \text{~for any~} F\in \sA_1^p\rangle.$$
From the definition of torsion pair, 
for any $T\in\sA_1^p$, and $F\in \sA_{1/2}^{p}$,  we have  $\Hom(T,F)=0$ and any $E\in \sA^p_{\XX}$ fits into the exact sequence:
$$0\to T\to E\to F\to 0$$
such that  $T\in\sA_1^p$ and $F\in \sA_{1/2}^{p}$.
Let 
$$\sA_{\XX,1}:=\sA_1^{p}[-1]\cap\sA_{\XX}=\langle\sO_{x}[-1]: x\in\XX\rangle_{\ex}$$
and 
$$\sA_{\XX,1/2}:=\sA_{1/2}^{p}[-1]\cap\sA_{\XX}=\{E\in\sA_{\XX}|\Hom(\sA_{\XX,1},E)=0\}.$$
Then $(\sA_{\XX,1}, \sA_{\XX,1/2})$ is a torsion pair on $\sA_{\XX}$. Also if 
$E\in\sA^p_{1/2}[-1]$, and $\rk(E)=0$ or $1$, and $c_1(E)=0$, then from 
\cite[Lemma 5.1]{Toda_Kyoto}, $E\in\sA_{\XX,1/2}$. The proof is as follows. 
Rank zero case is obvious.  In the rank one case, $\hH^{-1}(E)$ is torsion free of rank one, so there exists a gerby curve $C\subset \XX$ such that 
$\hH^{-1}(E)\to \hH^{-1}(E)^{\vee\vee}\cong\sO_{\XX}$ and an exact sequence 
$I_C\to E\to F[-1]$. $I_C, F[-1]\in\sA_{\XX}$ imply that $E\in\sA_{\XX}$.
Thus for any $E\in\sA_{\XX}$, there exists an exact sequence:
\begin{equation}\label{eqn_ecaxt_sequence_thm1}
0\to T\to E\to F\to 0
\end{equation}
such that $T\in\sA_{\XX,1}$ and $F\in\sA_{\XX,1/2}$. Also the categories $\sA_{\XX,1}, \sA_{\XX,1/2}$ are of finite length with respect to  strict 
epimorphism and strict monomorphism. This is due to \cite[Lemma 2.19]{Toda_Duke} and the category of twisted sheaves on $X$ (which is the same as category of coherent sheaves on $\XX$) is equivalent to the categories of untwisted shaves 
as in \cite{HS}. 
Since $\sZ_{\omega,\theta}$ is defined by:
$$
\begin{cases}
\sZ_{\omega,\theta,0}(n,\beta)=n-(\omega\cdot \beta)\sqrt{-1};\\
\sZ_{\omega,\theta,1}(\rk)=\rk\cdot e^{i\pi\theta}.
\end{cases}
$$
The same proof as in \cite[Lemma 2.27]{Toda_Duke} provides that $E\in\sA_{\XX}$ is $\sZ_{\omega,\theta}$-semistable if and only if 
one of the following conditions holds:
\begin{enumerate}
\item $E\in\sA_{\XX,1}$;
\item $E\in\sA_{\XX,1/2}$ and for any exact sequence
$$0\to F\to E\to G\to 0$$
in $\sA_{\XX}$ with $F,G\in \sA_{\XX,1/2}$, and 
$\arg\sZ_{\omega,\theta}(F)\leq \arg\sZ_{\omega,\theta}(G)$.
\end{enumerate}
Thus the Harder-Narasimhan filtration for any $E\in \sA_{\XX}$ is given by the exact sequence (\ref{eqn_ecaxt_sequence_thm1}) for the torsion pair and the  
Harder-Narasimhan filtration for any $F$ is also given by the exact sequence (\ref{eqn_ecaxt_sequence_thm1}). 
\end{proof}

\subsubsection{Moduli stack of semistable objects in $\sA_{\XX}$}

We use the following big moduli stack $\widehat{\sM}_{\XX}$ on $\sA_{\XX}$ counting perfect complexes satisfying certain conditions, see 
\cite{Lieblich_JAG}, \cite{Toda_Kyoto}. Let us define 
$$\widehat{\sM}_{\XX}: (\Sch_{\cc})\to (\text{groupoids})$$
to be the stack that sends 
$$S\mapsto \{\sE| \sE\in D(\Coh(\XX\times S))| \text{Condition*} \}/\cong$$
where $\text{Condition*}$ is:
$$
\begin{cases}
\bullet  \,\,\sE \text{~is relatively perfect};\\
\bullet  \,\,\sE_s=Li_s^*\sE\in D^b(\Coh(\XX));\\
\bullet  \,\,\Ext^i(\sE_s, \sE_s)=0, i<0 \text{~for any~} s\in S.
\end{cases}
$$
Then $\widehat{\sM}_{\XX}$ is an Artin stack locally of finite type. We define a substack 
$$\widehat{\sM}(\sA_{\XX})\subset \widehat{\sM}_{\XX}$$
to be the substack consisting of all $S$-valued point $\sE\in \widehat{\sM}(S)$ such that 
$\sE_s\in\sA_{\XX}$ for all the points $s\in S$.  Then we may write 
\begin{equation}\label{eqn_decomposition_stacks}
\widehat{\sM}(\sA_{\XX})=\bigsqcup_{v\in\Gamma}
\widehat{\sM}^{v}(\sA_{\XX})
\end{equation}
where $\widehat{\sM}^{v}(\sA_{\XX})$ is the substack of objects $E\in\sA_{\XX}$ such that 
$\cl(E)=v$. 
Note here that $\cl(E)=(\Ch_3(E), \Ch_2(E), \Ch_0(E))\in\Gamma$, since we work on $\mu_r$-gerbes $\XX\to X$.
In general the Chern character $\Ch$ should be taken as the  orbifold Chern character $\widetilde{\Ch}: K(\XX)\to H_{\CR}^*(\XX)$. 
Since $H_{\CR}^*(\XX)=H^*(\XX)\sqcup\cdots\sqcup H^*(\XX)$, up to the action of $\mu_r$ on $E$, $\cl(E)$ keeps the same on each component. Later on 
we only care about the first component given by $\XX$-twisted sheaves.

Then $\widehat{\sM}^{v}(\sA_{\XX})\subset \widehat{\sM}_{\XX}$ is an open immersion for 
$v=(n,\beta,\rk)\in\Gamma$ with $\rk=0$ or $1$; and 
$\widehat{\sM}^{v}(\sA_{\XX})$ is an Artin stack locally of finite type. 

\begin{defn}\label{defn_moduli_stack_stable}
Define
$$\widehat{\sM}_{n,\beta}(\omega,\theta)\subset \widehat{\sM}^{v}(\sA_{\XX})$$
to be the stack of $\sZ_{\omega,\theta}$-semistable objects $E\in\sA_{\XX}$ with 
$\cl(E)=v=(-n,-\beta,1)$.
\end{defn}

Here is a similar result as in \cite[Proposition 5.4]{Toda_Kyoto}:

\begin{prop}\label{prop_Toda_Kyoto_5.4}
We have:
\begin{enumerate}
\item The stack $\widehat{\sM}_{n,\beta}(\omega,\theta)$ is an Artin stack of finite type. 
\item In the case $\theta\to 1$, $\widehat{\sM}_{n,\beta}(\omega,\theta)\cong [P_n(\XX,\beta)/\cc^*]$ which is the trivial 
$\cc^*$-gerbe.
\item  
$$\widehat{\sM}_{n,\beta}(\omega,\theta)\cong \widehat{\sM}_{-n,\beta}(\omega, 1-\theta)$$ is given by
$E\mapsto R\Hom(E,\sO_{\XX})$.
\item $\widehat{\sM}_{n,\beta}(\omega,\frac{1}{2})=\emptyset$ for $|n|>>0$.
\end{enumerate}
\end{prop}
\begin{proof}
We first recall the $\mu$-limit stability of Toda in \cite{Toda_Duke} and generalize it to $\mu_r$-gerbes $\XX\to X$. 
Let 
$B+i\omega\in H^2(\XX,\cc)$, where $\omega\in H^2(X,\rr)$ is an ample divisor. 
Then in \cite{Toda_Duke} Toda defined the so called $\mu_{B+i\omega}$-limit stability on $\sA^p_{\XX}$. 
We recall it here. Let 
$\sigma=B+i\omega\in A(\XX)_{\cc}=\{B+i\omega | \omega \text{~ample~}\}$, and let 
$$\sZ_{\sigma}: K(\XX)\to \cc$$
be given by 
\begin{equation}\label{eqn_integration_Toda_Kyoto_5.1}
E\mapsto -\int_{\XX}e^{-(B+i\omega)}\Ch(E)\sqrt{\Td_{\XX}}\in \cc.  
\end{equation}
Here we only use the original Chern character on $\XX$ and therefore Toda's arguments in \cite[\S 3]{Toda_Duke}, \cite[\S 3]{Toda_ASPM} go through for sheaves on $\sA^p_{\XX}$
for $\XX$.  Then one can write down 
$$\sZ_{\sigma}(E)=\left(-v_3^B(E)+\frac{1}{2}\omega^2 v_1^B(E)\right)
+\left(\omega v_2^B(E)-\frac{1}{6}\omega^3 v_0^B(E)\right)i,
$$
where 
$$v^B(E)=e^{-B}\Ch(E)\cdot\sqrt{\td_{\XX}}=(v_0^B(E), v_1^B(E), v_3^B(E))\in H^{\even}(\XX,\rr)\cong H^{\even}(X,\rr). $$

\begin{rmk}
The integration (\ref{eqn_integration_Toda_Kyoto_5.1}) may depend on the gerbe structure $\XX$.
\end{rmk}

Then from \cite[Lemma 2.20]{Toda_Duke}, for any nonzero $E\in\sA^p_{\XX}$, and $\sigma_m=B+im\omega$, 
$$\sZ_{\sigma_m}(E)\in\{\rho\cdot e^{i\pi\phi} | \rho>0, \frac{1}{4}<\phi<\frac{5}{4}\}$$
for $m>>0$. Therefore 
$$\phi_{\sigma_m}(E)=\frac{1}{\pi}\Im \log\sZ_{\sigma_m}(E)\in (\frac{1}{4}, \frac{5}{4}).$$
Then a nonzero $E\in\sA^p_{\XX}$ is called $\sigma$-limit stable (or $\sigma$-limit semistable) if for any nonzero
$F\subsetneq E$, $\phi_{\sigma}(F)<\phi_{\sigma}(E)$, (or $\phi_{\sigma}(F)\leq \phi_{\sigma}(E)$).

Here a lemma as in \cite[Lemma 3.8]{Toda_Duke}:

\begin{lem}\label{lem_Toda_Duke_3.8}
For $\sigma=B+i\omega\in A(\XX)_{\cc}$, $E\in\sA^p_{1/2}$, such that 
$\det(E)=\sO_{\XX}$, and 
$\Ch(E)=(-1,0,\beta,n)$. Let 
$F\in \Coh_{\leq 1}(\XX)$. Then $\phi_{\sigma}(F)<\phi_{\sigma}(E)$, (or $\phi_{\sigma}(F)>\phi_{\sigma}(E)$)
if and only if one of the following conditions holds:
\begin{enumerate}
\item $\mu_{\sigma}(F)<-\frac{3B\omega^2}{\omega^3}$, (resp. $\mu_{\sigma}(F)>-\frac{3B\omega^2}{\omega^3}$).
\item  $\mu_{\sigma}(F)=-\frac{3B\omega^2}{\omega^3}$ and 
$\omega v_2^B(E)\mu_{\sigma}(F)<v_3^B(E)$, (resp. $\omega v_2^B(E)\mu_{\sigma}(F)>v_3^B(E)$).
\end{enumerate}
\end{lem}
\begin{proof}
This is from similar calculations of \cite[Lemma 3.8]{Toda_Duke}.
\end{proof}

Then from Lemma \ref{lem_Toda_Duke_3.8} we have:
let $B=k\omega$,  
$E[1]\in\sA^p_{\XX}$ is $\mu_{B+i\omega}$-limit semistable if and only if $E\in\sA^p_{1/2}$ and 
\begin{enumerate}
\item for any one dimensional sheaf $F\neq 0$, $F\hookrightarrow E[1]$ is in $\sA^p_{1/2}$, we have 
$\frac{\Ch_3(F)}{\omega\cdot \Ch_2(F)}\leq -2k$;
\item  if we have $E[1]\twoheadrightarrow G$ for $G$ a one-dimensional object in $\sA^p_{1/2}$, we have $\frac{\Ch_3(G)}{\omega\cdot \Ch_2(G)}\geq -2k$.
\end{enumerate}
So following Toda, let 
$$
k=
\begin{cases}
\frac{1}{2\tan(\theta)}, & \theta\neq \frac{1}{2};\\
0, & \theta=\frac{1}{2}.
\end{cases}
$$
Then from the above argument:
For $E\in \sA^p_{\XX}[-1]$, $E[1]\in \sA^p_{\XX}$ is $\mu_{k\omega+i\omega}$-semistable if and only if $E\in \sA_{\XX}$
and $E$ is $\sZ_{\omega,\theta}$-semistable such that 
$\cl(E)=v=(-n,-\beta,1)\in\Gamma$. 

Proof of the proposition now can be done using the same method of Toda. 
\cite[Proposition 3.17]{Toda_ASPM} showed that $\widehat{\sM}_{n,\beta}(\omega,\theta)$ is an open substack of the stack 
$\widehat{\sM}_{\XX}$, hence an Artin stack locally of finite type. 
When $\theta\to 1$, the stack 
$\widehat{\sM}_{n,\beta}(\omega,\theta)=[P_n(\XX,\beta)/\cc^*]$ comes from 
\cite[Theorem 3.21]{Toda_ASPM} and \cite[Theorem 4.7]{Toda_Duke}. We omit the details.  All of the results in $(3), (4)$ are from \cite[Lemma 2.28]{Toda_Duke} and 
\cite[Lemma 4.4]{Toda_ASPM}.
\end{proof}

\subsubsection{Counting invariants and wall crossing}

We define the counting invariants in the abelian category $\sA_{\XX}$. 
If $\widehat{\sM}_{n,\beta}(\omega,\theta)=[\widehat{M}_{n,\beta}(\omega,\theta)/\cc^*]$  is the moduli stack of 
$\sZ_{\omega,\theta}$-stable $E\in\sA_{\XX}$ satisfying $\cl(E)=(-n,-\beta, 1)$, we define

\begin{defn}\label{defn_invariants_stable}
$$\DT_{n,\beta}(\omega,\theta)=\chi(\widehat{M}_{n,\beta}(\omega,\theta), \nu_{\widehat{M}}),$$
is the weighted Euler characteristic, 
where $\nu_{\widehat{M}}$ is the Behrend function. 
\end{defn}

Our goal is to count the semistable objects in $\sA_{\XX}$.  We use Joyce-Song method to consider the Hall algebra 
$$H(\sA_{\XX})=K(\St/\widehat{\sM}_{\XX})$$
with $\star$-product.  We use the definition of Hall algebra of Joyce and Bridgeland \cite{Bridgeland10}. 
First for a pair $(n,\beta)\in \Gamma_0$, let $\sM_{n,\beta}(\omega)\subset \widehat{\sM}_{\XX}$ be the substack 
parametrizing $\sZ_{\omega}$-semistable $E\in \Coh_{\leq 1}(\XX)$ such that the data is given by $[E]=\beta$ and 
$\int_{\XX}\Ch(E)\cdot \Td_{\XX}=n$.  This is an open substack of finite type. 

The elements in the Hall algebra $H(\sA_{\XX})$ are given by 
$$\widehat{\delta}_{n,\beta}(\omega)=[\sM_{n,\beta}(\omega)\stackrel{i}{\hookrightarrow} \widehat{\sM}_{\XX}]$$
where $i$ sends $E$ to $E[-1]\in\sA_{\XX}$, 
and 
$$\widehat{\delta}_{n,\beta}(\omega,\theta)=[\widehat{\sM}_{n,\beta}(\omega,\theta)\to  \widehat{\sM}_{\XX}].$$
We have its logarithm:
$$\widehat{\epsilon}_{n,\beta}(\omega,\theta)=$$
$$\sum_{\substack{\ell\geq 1, 1\leq e\leq\ell\\
(n_i,\beta_i)\in\qq\oplus H_2(\XX,\qq)\\
n_1+\cdots+n_{\ell}=n\\
\beta_1+\cdots+\beta_{\ell}=\beta\\
\sZ_{\omega,\theta}(-n_i,-\beta_i)\in\rr_{>0}e^{i\pi\theta},i\neq e}}
\frac{(-1)^{\ell-1}}{\ell}
\widehat{\delta}_{n_1,\beta_1}(\omega)\star\cdots\star\widehat{\delta}_{n_{e-1},\beta_{e-1}}(\omega)\star
\widehat{\delta}_{n_e,\beta_e}(\omega,\theta)\star\cdots\star\widehat{\delta}_{n_{\ell},\beta_{\ell}}(\omega)
$$
\begin{defn}\label{defn_DT_generalized_JS}
Define 
$$\DT_{n,\beta}(\omega,\theta)=
\lim_{t\to 1}(t^2-1)P_t(-\nu\cdot \widehat{\epsilon}_{n,\beta}(\omega,\theta))$$
Here 
\begin{enumerate}
\item $\nu: H(\sA_{\XX})\to H(\sA_{\XX})$ is the map by inserting the Behrend function given by
$$[\sY\stackrel{\rho}{\longrightarrow}\widehat{\sM}_{\XX}]\mapsto \sum_{i\in \zz}i\cdot [\rho|_{\sY_i}: \sY_i\to \widehat{\sM}_{\XX}] $$
where $\sY_i=(\nu_{\sM}\circ \rho)^{-1}(i)$.
\item $P_t: K(\St/\widehat{\sM}_{\XX})\to \qq(t)$ is the map 
$$P_t\left([\rho: [Y/GL_m(\cc)]\to \widehat{\sM}_{\XX}]\right)
\mapsto \frac{P_t([Y])}{P_t([GL_m(\cc)])}$$
and $P_t([Y])$ is the Poincar\'e polynomial of the quasi-projective variety $Y$.
\end{enumerate}
\end{defn}

\subsubsection{Limit stable invariants $L_{n,\beta}$}  As in \cite{Toda_Kyoto}, we define 
\begin{defn}\label{defn_limit_stable_invariants}
Define
$$L_{n,\beta}:=\DT_{n,\beta}(\omega, \theta=\frac{1}{2}).$$
\end{defn}

We have the following fact:
\begin{enumerate}
\item When $\theta\to 1$, $\DT_{n,\beta}(\omega,\theta)=P_{n,\beta}$;
\item $L_{n,\beta}=L_{-n,\beta}$ and they are zero when $|n|>>0$.
\end{enumerate}

\subsubsection{Rank zero invariants $N_{n,\beta}$} 
In the category of $\Coh_{\leq 1}(\XX)$, recall that we have 
$$\widehat{\delta}_{n,\beta}(\omega)=[\sM_{n,\beta}(\omega)\stackrel{i}{\hookrightarrow} \widehat{\sM}_{\XX}]$$
where $i$ sends $E$ to $E[-1]\in\sA_{\XX}$, 
we define 
$$\widehat{\epsilon}_{n,\beta}(\omega)=\sum_{\substack{\ell\geq 1, 1\leq i\leq\ell\\
(n_i,\beta_i)\in\qq\oplus H_2(\XX,\qq)\\
n_1+\cdots+n_{\ell}=n\\
\beta_1+\cdots+\beta_{\ell}=\beta\\
\sZ_{\omega}(-n_i,-\beta_i)=\sZ_{\omega}(-n,-\beta)}}
\frac{(-1)^{\ell-1}}{\ell}
\widehat{\delta}_{n_1,\beta_1}(\omega)\star\cdots\star\widehat{\delta}_{n_{\ell},\beta_{\ell}}(\omega)
$$
\begin{defn}\label{defn_DT_generalized_JS_N-nbeta}
Define 
$$N_{n,\beta}(\omega)=
\lim_{t\to 1}(t^2-1)P_t(-\nu\cdot \widehat{\epsilon}_{n,\beta}(\omega)).$$
\end{defn}

The invariants $N_{n,\beta}(\omega)$ counts rank zero $\sZ_{\omega}$-semistable coherent sheaves in $\Coh_{\leq 1}(\XX)$. From \cite[Lemma 4.8]{Toda_Kyoto}, it is independent to the choice of $\omega$ and we just use $N_{n,\beta}:=N_{n,\beta}(\omega)$. 

\subsubsection{Wall-crossing formula of Toda}

The wall crossing formula of Toda in \cite[Theorem 5.7]{Toda_Kyoto} and \cite[Theorem 5.8, Theorem 8.10]{Toda_JAMS} works for coherent sheaves and stable pairs $[\sO_{\XX}\to F]$
on $\XX$. For this we let 
$$\DT(\omega,\theta):=\sum_{n,\beta}\DT_{n,\beta}(\omega,\theta)q^n t^{\beta}$$
be the generating series for $0<\theta<\frac{1}{2}$. 
The wall and chamber structure for weak stability conditions are given by \cite[\S 5.1]{Toda_JAMS}. It is defined in a subset 
$\sV\subset \Stab_{\Gamma_{\bullet}}(\sD_{\XX})$ introduced in \cite[\S 5.1]{Toda_JAMS}. Then we have the limit:
$$\DT(\omega, \phi_{\pm}):=
\lim_{\theta\to \phi\pm 0}\DT(\omega,\theta).$$
The following result is from \cite[Theorem 5.8, Theorem 8.10]{Toda_JAMS} and 
\cite[Theorem 5.7]{Toda_Kyoto}.

\begin{thm}\label{thm_wall_crossing}
Let $0<\phi<\frac{1}{2}$, then we have the wall crossing formula:
$$DT(\omega,\phi_{+})=\DT(\omega,\phi_-)\cdot 
\prod_{\substack{n>0, \beta>0\\
-n+(\omega\cdot \beta)i\in \rr_{>0}e^{i\pi\phi}}}\exp((-1)^{n-1}n N_{n,\beta}q^nt^{\beta})$$
\end{thm}
\begin{proof}
Since we work on the category $\sA_{\XX}$ of objects in the derived category of coherent sheaves on the gerbe $\XX$, and from \cite{TT-Adv}, the category of coherent sheaves on $\XX$
is isomorphic to the category of twisted sheaves on the rigidified inertia stack $X\sqcup\cdots\sqcup X$ by some $\cc^*$-gerbe determined by the $\mu_r$-gerbe $\XX$.  The category of  $\cc^*$-gerbe twisted sheaves is equivalent 
 \cite{HS} to untwisted sheaves on $X\sqcup\cdots\sqcup X$. Therefore the method of Toda in \cite[Theorem 5.8, Theorem 8.10]{Toda_JAMS} work in this case and we get the wall-crossing formula. 
\end{proof}

\begin{cor}\label{cor_PT_wall-crossing}
we have 
\begin{equation}\label{eqn_Toda66}
\PT(\XX)=
\prod_{\substack{n>0, \beta>0}}e^{(-1)^{n-1}n N_{n,\beta}q^nt^{\beta}}
\cdot\left(\sum_{n,\beta}L_{n,\beta}q^nt^{\beta}\right).
\end{equation}
\end{cor}
\begin{proof}
Proposition \ref{prop_Toda_Kyoto_5.4} tells us that 
$$\lim_{\theta\to 1}\DT(\omega,\theta)=\PT(\XX).$$
The elements $F\in \Coh_{\leq 1}(\XX)$ such that 
$\sZ_{\omega,1/2}(F[-1])\in\rr_{>0}\sqrt{-1}$ will have 
$\chi(F)=0$. 
Then the elements with phase $\frac{1}{2}$ are pure dimensional one sheaves. The wall crossing formula \cite[Theorem 5.8, Theorem 8.10]{Toda_JAMS}
tells us $\lim_{\theta\to\frac{1}{2}}\DT(\omega,\theta)=\DT(\omega,\theta=\frac{1}{2})=\sum_{n,\beta}L_{n,\beta}q^n t^{\beta}$. Then the result is obtained by applying 
$\theta=\frac{1}{2}$ to $\theta=1$ from the Theorem above. 
\end{proof}

\subsection{Decomposition formula for \'etale gerbes}\label{subsec_decomposition_gerbes}

\subsubsection{GW theory for $\mu_r$-gerbes}

The gerbe $\XX\to X$ is given by an element in $H^2(X,\mu_r)$, therefore is banded.  We recall some results in \cite{AJT_JDG} and \cite[\S 6]{AJT}. 
Let $\sK_g(\XX,\beta)$ be the moduli stack of twisted stable maps of genus $g$ twisted curves to $\XX$ of degree 
$\beta\in H_2(\XX,\qq)=H_2(X,\qq)$.  We borrow the following diagram from \cite[Formula (41)]{AJT}:
\[
\xymatrix{
\sK_g(\XX,\beta)\ar[r]^{t}\ar[dr]\ar@/^2pc/[rr]^{p} & P_g\ar[r]^--{q^\prime}\ar[d]& \overline{M}_g(X,\beta)\ar[d]\\
&\mathfrak{M}_g^{\tw}\ar[r]^{q}& \mathfrak{M}_g
}
\]
where 
\begin{enumerate}
\item $\mathfrak{M}_g^{\tw}$ is the moduli stack of prestable twisted curves of genus $g$;
\item $\mathfrak{M}_g$ is the moduli stack of prestable  curves of genus $g$;
\item $q$ maps a twisted curve to the underlying prestable curve;
\item $\overline{M}_g(X,\beta)$ is the moduli space of stable genus $g$ stable maps to $X$ of degree $\beta$;
\item The right vertical arrow is the forgetful morphism sending a stable map 
$f: C\to X$ to $C$;
\item The stack $P_g$ is defined as the fiber product;
\item The natural morphism
$$p: \sK_g(\XX,\beta)\to \overline{M}_g(X,\beta)$$
sends a twisted stable map $[f: \sC\to\XX]$ to the underlying stable map
$[f:C\to X]$ between the coarse moduli spaces;
\item The morphism $\sK_g(\XX,\beta)\to \mathfrak{M}_g^{\tw}$ is defined by the forgetful morphism again;
\item The morphism $t$ is defined by the universal property of fiber product. 
\end{enumerate}
In \cite[Proposition 5.1, Lemma 5.2]{AJT}, the authors show that $t$ is \'etale and factors through 
$$\sK_g(\XX,\beta)\to\sK_g(\XX,\beta)\myfatslash\mu_r\to P_g$$
where $\sK_g(\XX,\beta)\myfatslash\mu_r$ is the rigidification of $\sK_g(\XX,\beta)$ and $\sK_g(\XX,\beta)\to\sK_g(\XX,\beta)\myfatslash\mu_r$ is 
a $\mu_r$-gerbe.  In \cite[\S 6]{AJT}, the authors talked about the pushforward formula for virtual fundamental classes. 

\begin{thm}\label{thm_pushforward}(\cite[Theorem 6.8]{AJT})
We have:
$$p_*[\sK_g(\XX,\beta)]^{\vir}=r^{2g-1}\cdot [ \overline{M}_g(X,\beta)]^{\vir}.$$ 
\end{thm}

We are working on the twisted stable maps to $\XX$ without marked points.  Therefore 
$$N^{\GW}_{g,\beta}(\XX):=\int_{[\sK_g(\XX,\beta)]^{\vir}}1=
r^{2g-1}\cdot \int_{[\overline{M}_g(X,\beta)]^{\vir}}1=:r^{2g-1}N^{\GW}_{g,\beta}(X).$$
We let 
$$F_{\GW}(\XX; u, v)=\sum_{g\geq 0}\sum_{\beta\neq 0}N^{\GW}_{g,\beta}(\XX)u^{2g-2}v^{\beta}$$
and 
$$F_{\GW}(X; u, v)=\sum_{g\geq 0}\sum_{\beta\neq 0}N^{\GW}_{g,\beta}(X)u^{2g-2}v^{\beta}$$
be the generating function of Gromov-Witten invariants for $\XX$ and $X$. 

\begin{thm}\label{thm_GW_gerbe_formula}
$$F_{\GW}(\XX; u, v)=r\cdot F_{\GW}(X; ru, v).$$
\end{thm}
\begin{proof}
We calculate 
\begin{align*}
F_{\GW}(\XX; u, v)&=\sum_{g\geq 0}\sum_{\beta\neq 0}N^{\GW}_{g,\beta}(\XX)u^{2g-2}v^{\beta}\\
&=\sum_{g\geq 0}\sum_{\beta\neq 0}r^{2g-1}N^{\GW}_{g,\beta}(X)u^{2g-2}v^{\beta}\\
&=r\cdot \sum_{g\geq 0}\sum_{\beta\neq 0}N^{\GW}_{g,\beta}(X)(ru)^{2g-2}v^{\beta}\\
&=r\cdot F_{\GW}(X; ru, v).
\end{align*}
\end{proof}

\subsubsection{Gopakumar-Vafa invariants and the formula of Toda}

The GW invariants $N^{\GW}_{g,\beta}(\XX)$ and $N^{\GW}_{g,\beta}(X)$ defined above are rational numbers.  From Gopakumar-Vafa \cite{GV}, we have the following 
integrality conjecture which gives the Gopakumar-Vafa invariants. This is given by string duality between type IIA string theory and M-theory. 

\begin{con}\label{con_GV}(Gopakumar-Vafa conjectural invariants)
There exist integers $n_g^{\beta}\in\zz$, for $g\geq 0$ and $\beta\in H_2(X,\zz)$ such that 
$$\sum_{g\geq 0,\beta\geq 0}N^{\GW}_{g,\beta}(X)u^{2g-2}t^{\beta}=\sum_{\substack{g\geq 0,\beta> 0\\
k\in\zz_{\geq 1}}}\frac{n_g^{\beta}}{k}\left(2\sin\left(\frac{ku}{2}\right)^{2g-2}\right)t^{k\beta}.$$
\end{con}

We recall the GW/PT-correspondence  in \cite{PT}, \cite{MPT}.
Let 
$$\sZ_{\GW}(X;u,t):=\exp(F_{\GW}(X; u, t))=1+\sum_{\beta\neq 0}\sZ_{\GW}(X;u)_{\beta}t^{\beta}.$$
Then the GW/PT-correspondence is given by:
\begin{equation}\label{eqn_GW/PT}
\sZ_{\GW}(X;u,t)=\sZ_{\PT}(X; q,t),
\end{equation}
for $q=-e^{iu}$. Here 
$$\sZ_{\PT}(X; q,t)=\PT(X)^{\prime}=1+\sum_{\beta\neq 0}\PT(X)^\prime_{\beta}$$
and 
$\PT(X)^{\prime}=\sum_{\beta\neq 0}\sum_{n\in\zz}P_{n,\beta}(X)q^n t^{\beta}$. 
Let us define 
$$\PT(X):=\sum_{\beta\geq 0}\sum_{n\in\zz}P_{n,\beta}(X)q^n t^{\beta}=1+\sum_{\beta>0}\PT_{\beta}(X)$$
for $\PT_{\beta}(X)=\sum_{n\in\zz, \beta>0}P_{n,\beta}(X)q^n t^{\beta}$. 
\cite[Conjecture 6.2]{Toda_Kyoto} made the following conjecture:
\begin{con}\label{con_Toda_68}
There exist integers $n_g^{\beta}\in\zz$, for $g\geq 0$ and $\beta\in H_2(X,\zz)$ such that 
\begin{equation}\label{eqn_Toda_68}
\PT(X)=\prod_{\beta>0}\left(\prod_{j=1}^{\infty}(1-(-q)^{j}t^{\beta})^{j n_0^{\beta}}\cdot 
\prod_{g=1}^{\infty}\prod_{k=0}^{2g-2}(1-(-q)^{g-1-k}t^{\beta})^{(-1)^{k+g}\cdot n_g^{\beta}\tiny\mat{c} 2g-2\\
k\rix}\right).
\end{equation}
\end{con}

For a $\mu_r$-gerbe $\XX\to X$, the GW potential of $\XX$ is related to $\PT(X)$ through the GW potential of $X$ by:

\begin{prop}\label{prop_GW_gerbe_XX}
Let $\sZ_{\GW}(\XX;u,t):=\exp(F_{\GW}(\XX; u, t))$. Then we have
$$\sZ_{\GW}(\XX;u,t)=(\sZ_{\GW}(X;ru,t))^{r}.$$
\end{prop}
\begin{proof}
This is calculated by
\begin{align*}\sZ_{\GW}(\XX;u,t)&=\exp(F_{\GW}(\XX; u, t))\\
&= \exp(r\cdot F_{\GW}(X; ru, t))\\
&=(\sZ_{\GW}(X;ru,t))^{r}.
\end{align*}
\end{proof}

\subsection{Multiple cover formula.}

In this section we fix $X$ to be a smooth projective Cababi-Yau threefold; and $p:\XX\to X$ a 
$\mu_r$-gerbe. 

Recall the wall crossing formula (\ref{eqn_Toda66}) in Corollary \ref{cor_PT_wall-crossing}.  If the gerbe $\XX\to X$ is trivial, or we just count the stable pairs 
$[\sO_{\XX}\to F]$ for the generating sheaf $\Xi=\sO_{\XX}$, then the wall crossing  Formula  (\ref{eqn_Toda66}) is just the same as Toda's formula in \cite[Formula (66)]{Toda_Kyoto} based on the following:

\begin{prop}\label{prop_trivial_gerbe}
If the gerbe $\XX\to X$ is trivial,  then every  coherent sheaf on $\XX$ is pulled back from its coarse moduli space $X$. 
\end{prop}

We follow \cite[\S 6.2]{Toda_Kyoto} to derive the multiple cover formula. 
First take  logarithm of RHS of  (\ref{eqn_Toda_68}) yields:
\begin{align}\label{eqn_RHS_68}
&\log \prod_{\beta>0}\left(\prod_{j=1}^{\infty}(1-(-q)^{j}t^{\beta})^{j n_0^{\beta}}\cdot 
\prod_{g=1}^{\infty}\prod_{k=0}^{2g-2}(1-(-q)^{g-1-k}t^{\beta})^{(-1)^{k+g}\cdot n_g^{\beta}\tiny\mat{c} 2g-2\\
k\rix}\right) \\ \nonumber
&=\sum_{\beta>0}\sum_{j=1}^{\infty}j\cdot n_0^{g}\sum_{k\geq 1}\frac{(-1)^{jk-1}}{k}q^{jk} t^{k\beta}+
\sum_{\beta>0}\sum_{g=1}^{\infty}\sum_{a\geq 1}\frac{n_g^{\beta}}{a}\sum_{k=0}^{2g-2}\tiny\mat{c} 2g-2\\
k\rix(-(-q)^a)^{g-1-k}t^{a\beta}\\  \nonumber
&=\sum_{\beta>0}\sum_{n=1}^{\infty}\sum_{\substack{k\geq 1\\ 
k| (\beta,n)}}\frac{(-1)^{n-1}n}{k^2}n_0^{\beta/a}q^{n} t^{\beta}+
\sum_{\beta>0}\sum_{g=1}^{\infty}\sum_{\substack{a\geq 1\\ 
a|\beta}}\frac{n_g^{\beta/a}}{a} f_g(-(-q)^a)t^{\beta}, 
\end{align}
where 
$$f_g(q):=\sum_{k=0}^{2g-2}\tiny\mat{c} 2g-2\\
k\rix q^{g-1-k}=q^{1-g}(1+q)^{2g-2}.$$
The second term is a polynomial in $q^{\pm 1}$ and is invariant under $q\leftrightarrow \frac{1}{q}$. 

Now we take the logarithm of   Formula (\ref{eqn_Toda66}) :
\begin{align}\label{eqn_log_66}
\log (\PT(\XX))&=\sum_{\beta>0}\sum_{n>0}(-1)^{n-1}n\cdot N_{n,\beta}q^n t^{\beta}+
\log\left(\sum_{n,\beta}L_{n,\beta}q^n t^{\beta}\right). 
\end{align}
Here we write:
$$\sum_{\beta>0}L_\beta(q)t^{\beta}:=\log\left(\sum_{n,\beta}L_{n,\beta}q^n t^{\beta}\right).$$
Then we equal these two formula (\ref{eqn_RHS_68}) and   (\ref{eqn_log_66}) and get:

\begin{equation}\label{eqn_key_formula1}
\sum_{n>0}(-1)^{n-1}n N_{n,\beta}q^n=\sum_{n=1}^{\infty}\sum_{\substack{k\geq 1\\
k|(\beta,n)}}\frac{(-1)^{n-1}n}{k^2}n_0^{\beta/k} q^n
\end{equation}
and 
\begin{equation}\label{eqn_key_formula2}
L_\beta(q)=\sum_{g=1}^{\infty}\sum_{\substack{a\geq 1\\
a|\beta}}\frac{n_g^{\beta/a}}{a}f_g(-(-q)^a).
\end{equation}

Formula (\ref{eqn_key_formula1}) gives the multiple cover formula:

\begin{con}\label{eqn_multiple_cover_formula}
$$N_{n,\beta}=\sum_{\substack{k\geq 1\\
k|(\beta,n)}}\frac{1}{k^2}n_0^{\beta/k}$$
where $N_{1,\beta}=n_0^{\beta}$.   
\end{con}
This conjecture is equivalent to Conjecture \ref{con_Toda_68}, and 
$n_g^{\beta}$ for $g\geq 1$ are written down by the invariants $L_{n,\beta}$. 
Since we don't need the higher genus Gopakumar-Vafa invariants in this paper, we leave  these invariants for the  future research.

\section{Multiple cover formula for  twisted K3 surfaces}\label{sec_twisted_K3}

In this section we prove a multiple cover formula for counting semistable sheaves in the category of twisted sheaves on 
a twisted K3 surface. 
Let $S$ be a smooth projective K3 surface.  We let $p:\SS_{\alpha}\to S$ be an optimal $\mu_r$-gerbe, which means that 
the order of the corresponding  $[\alpha]\in H^2(S,\sO_S^*)_{\tor}$ in the cohomological Brauer group is $r$.   The pair $(S,\alpha)$ or the 
gerbe $\SS_{\alpha}$ is called a twisted K3 surface.

\subsection{KKV formula for K3 surfaces and the theorem of Maulik-Thomas} 

Let us recall how Maulik-Thomas \cite{MT} prove the multiple cover formula for K3 surfaces. 
Let $X:=S\times \cc$ be the local K3 surface.  Let $\iota: S\to X$ be the inclusion and 
$$N_{g,\beta}^{\red}:=\int_{[\overline{M}_{g}(X,\iota_*\beta)]^{\vir}}\frac{1}{e(N^{\vir})}$$
be the reduced connected residue GW invariants of $X$ by $\cc^*$-localization. 
We write its generating series in terms of ``BPS" form as in Conjecture \ref{con_GV}: 
$$\sZ_{\GW}^{\red}(X;u,t)=
\sum_{g\geq 0,\beta\neq 0}N^{\red}_{g,\beta}(X)u^{2g-2}t^{\beta}=\sum_{\substack{g\geq 0,\beta\neq 0}}
n_g^{\beta} u^{2g-2}\sum_{k>0}\frac{1}{k}\left(\frac{\sin(ku/2)}{u/2}\right)^{2g-2}t^{k\beta}.$$
The Gopakumar-Vafa invariants $n_g^{\beta}$ are in fact integers $n_{g,h}\in\zz$ which depend only on 
$h$ and 
$$\int_{S}\beta^2=2h-2.$$
These invariants are nonzero only for $0\leq g\leq h$ and are determined by the KKV formula
\begin{equation}\label{eqn_KKV}
\sum_{g\geq 0}\sum_{h\geq 0}(-1)^g n_{g,h}(y^{\frac{1}{2}}-y^{-\frac{1}{2}})^{2g}q^h
=\prod_{n\geq 1}\frac{1}{(1-q^n)^{20} (1-yq^n)^{2} (1-y^{-1}q^n)^{2}}.
\end{equation}
This KKV formula is proved in \cite{PT_KKV}.
In \cite{MT}, Maulik-Thomas prove the following result:
\begin{thm}\label{thm_MT_6.3}(\cite[Theorem 6.3]{MT})
$$\sZ_{\GW}^{\red}(X;u,t)=\Z_{\PT}^{\red}(X;q,t)=-\log (1+\Z_{\PT}^{\chi}(X;q,t))=-\log (1+\Z_{\PT}^{\na}(X;-q,t))$$
after $q=e^{-iu}$ and $\Z_{\PT}^{\na}(X;-q,t)$ is the generating series of the naive Euler characteristic of the stable pair moduli spaces. Here  $Z_{\PT}^{\red}(X;q,t), \Z_{\PT}^{\chi}(X;q,t)$ were written as 
$Z_{P}^{\red}(X;q,t), \Z_{P}^{\chi}(X;q,t)$ in Section \ref{sec_Behrend_Oberdieck}. 
\end{thm}

The KKV formula of \cite{PT_KKV} implies that the PT-stable pair generating series $\Z_{\PT}^{\chi}(X;q,t)$ satisfies the ``BPS rationality" condition, and 
Toda \cite[Theorem 6.4]{Toda_Kyoto} shows that the ``BPS rationality" condition is equivalent to the multiple cover formula (\ref{eqn_multiple_cover_formula}). 
Thus we obtain:

\begin{prop}[\cite{MT}]\label{prop_multiple_cover_K3}
The multiple cover formula Conjecture \ref{eqn_multiple_cover_formula} holds for K3 surfaces. 
\end{prop}

\subsection{Category of twisted sheaves on $\overline{\XX}=\SS\times \pp^1$}

We let $X:=S\times \cc$ be the local K3 surface, and $\overline{X}=S\times\pp^1$. We have
$$H^2(S\times\pp^1,\mu_r)\cong H^2(S,\mu_r)\oplus H^2(\pp^1,\mu_r)\cong H^2(S,\mu_r)\oplus \zz_r.$$
Let $\SS:=\SS_{\alpha}\to S$ be an optimal $\mu_r$-gerbe for $\alpha\in H^2(S,\mu_r)$. 
Then  
$$p: \overline{\XX}:=\SS\times\pp^1\to S\times\pp^1$$
is a $\mu_r$-gerbe given by the class 
$\alpha\in H^2(S,\mu_r)\hookrightarrow H^2(S\times\pp^1, \mu_r)$, which is also optimal.  
Let 
\[
\xymatrix{
\overline{\XX}\ar[r]^{\pi}\ar[d]^{p}& \pp^1\\
S\times \pp^1\ar[ur]&
}
\]
be the projection.  
The gerbe twisted sheaves were introduced in \cite{Lieblich_Duke}, and was reviewed in \cite[\S 3]{Jiang_2019}.  We consider  $\overline{\XX}$-twisted sheaves on $\overline{\XX}=\SS\times \pp^1$.

\begin{prop}\label{prop_twisted_sheaf_barXX}
Any $\alpha$-twisted coherent sheaf $E$ on $\overline{\XX}$ is tensor product of a pullback $\alpha$-twisted sheaf on $\SS$ and a coherent sheaf on $\pp^1$. 
\end{prop}
\begin{proof}
Let $\pi_{\SS}: \overline{\XX}\to \SS$ be the projection.  Since $\alpha\in H^2(S,\mu_r)\hookrightarrow H^2(\overline{\XX},\mu_r)$, from the definition of twisted sheaves in 
\cite[\S 3.1]{Jiang_2019}, we have a commutative diagram:
\[
\xymatrix{
\mu_r\times E\ar[r]\ar[d]_{\chi}& E\ar[d]^{\id}\\
\cc^*\times E\ar[r]& E
}
\]
where $\chi: \mu_r\to \cc^*$ is the inclusion. 
Since any coherent sheaf 
$E=\pi_{\SS}^*E_1\otimes \pi^*E_2$ for $E_1$ and $E_2$ are coherent sheaves on $\SS$ and $\pp^1$,  the gerbe structure implies that $\mu_r$ only acts on $\pi_{\SS}^*E_1$.  Hence 
$E_1$ defines a $\SS$-twisted sheaf on $\SS$. 
\end{proof}

We define 
$$\Coh_{\pi}^{\tw}(\overline{\XX})\subset \Coh(\overline{\XX})$$
to be the subcategory of gerbe $\alpha$-twisted sheaves on $\overline{\XX}$ supported on the fiber of $\pi$, and let 
$$\sD_0^{\SS}:=D^b(\Coh_{\pi}^{\tw}(\overline{\XX}))$$
be the corresponding derived category.  As in \cite[Definition 2.1]{Toda_JDG}, we define 
$$\sD:=\langle \pi^*\Pic(\pp^1), \Coh_{\pi}^{\tw}(\overline{\XX})\rangle_{\tr}\subset D^b(\Coh^{\tw}(\overline{\XX}))$$
to be the triangulated category. 

\subsubsection{Twisted Chern character}\label{subsubsec_twisted_Chern_character}

We introduce the twisted Chern character on $\SS$ and twisted Mukai pairs. 
Let us first review the Brauer-Severi variety.  
The gerbe $\alpha\in \Br^\prime(S)$, which is the same as $\Br(S)$. Thus it determines a projective bundle $P\to S$ of rank $r-1$.  We have  a vector bundle $G$ over $P$ is defined by the Euler sequence
$$0\to \sO_P\rightarrow G\rightarrow T_{P/S}\to 0$$
of the projective bundle $P\to S$. 

Yoshioka \cite{Yoshioka2} defined the subcategory $\Coh(S;P)\subset \Coh(P)$ of coherent sheaves on $P$ to be the subcategory of sheaves 
$$E\in \Coh(S;P)$$
if and only if $$E|_{P_i}\cong p^*(E_i)\otimes \sO_{P_i}(\lambda_i)$$
where $\{U_i\}$ is an open covering of $S$, $P_i=U_i\times\pp^{r-1}$, $E_i\in\Coh(U_i)$; and there exists an equivalence:
\begin{equation}\label{eqn_equivalence_twist_untwist}
\Coh(S;P)\stackrel{\cong}{\longrightarrow}\Coh(S,\alpha)
\end{equation}
by
$$E\mapsto p_*(E\otimes L^{\vee}).$$
Here $\alpha=o([\XX])\in H^2(X,\sO_X^*)_{\tor}$ is the image of the $\mu_r$-gerbe $[\SS]$. The line bundle $L\in\Pic(P)$ is the line bundle with the property that 
$$L|_{p^{-1}(x)}=\sO_{p^{-1}(x)}(-1),$$
and $\Coh(S,\alpha)$ is the category of $\alpha$-twisted coherent sheaves on $S$.
We call $E\in \Coh(S;P)$ a $P$-sheaf.

A $P$-sheaf $E$ is of dimension $d$ if $p_*E$ is of dimension $d$ on $S$.  Yoshioka defined the Hilbert polynomial 
$$P_E^G(m)=\chi(p_*(G^\vee\otimes E)(m))=\sum_{i=0}^{d}\alpha_i^{G}(E)\cdot \mat{c} m+i\\
i\rix.$$
The stability and semistability can be defined using this Hilbert polynomial.

There is an integral structure on $H^*(S,\qq)$ following Huybrechts-Stellari \cite{HS}. 
The integer 
 $r$ is the minimal rank on $\SS$
such that there exists a rank $r$ $\SS$-twisted locally free sheaf $E$ on the generic scheme $S$. 
Recall that 
$\langle, \rangle$ the Mukai pairing on $H^*(S,\zz)$. 

\begin{defn}\label{defn_Mukai_vector_Psheaf}
For a $P$-sheaf $E$, define a Mukai vector of $E$ as:
$$v_{G}(E):=\frac{\Ch(Rp_*(E\otimes G^{\vee}))}{\sqrt{\Ch(Rp_*(G\otimes G^\vee))}}\sqrt{\td}_S=(\rk, \zeta, b)\in H^*(S,\qq),$$
where $p^*(\zeta)=c_1(E)-\rk(E)\frac{c_1(G)}{\rk(G)}, b\in \qq$.
\end{defn}
One can check that 
$$\langle v_{G}(E_1), v_{G}(E_2)\rangle=-\chi(E_1, E_2).$$
We define for $\xi\in H^2(S,\zz)$, an injective homomorphism:
$$T_{-\frac{\xi}{r}}: H^*(S,\zz)\longrightarrow H^*(S,\qq)$$
by
$$x\mapsto e^{-\frac{\xi}{r}}\cdot x$$
and 
$T_{-\frac{\xi}{r}}$ preserves the bilinear form  $\langle, \rangle$. The following result is from \cite[Lemma 3.3]{Yoshioka2}. 

\begin{prop}\label{prop_repn_Mukai_vector}(\cite[Lemma 3.3]{Yoshioka2})
Let $\xi\in H^2(S,\zz)$ be a representation of $\omega(G)\in H^2(S,\mu_r)$, where $\rk(G)=r$. Set 
$$(\rk(E), D, a):=e^{\frac{\xi}{r}}\cdot v_G(E).$$
Then $(\rk(E), D, a)\in H^*(S,\zz)$ and $D \mod r=\omega(E)$.
\end{prop}

In \cite{H-St}, Huybrechts and  Stellari defined a weight $2$ Hodge structure on the lattice 
$(H^*(S,\zz), \langle,\rangle)$ as:
$$
\begin{cases}
H^{2,0}(H^*(S,\zz)\otimes \aaa^1_{\kappa}):=T_{-\frac{\xi}{r}}^{-1}(H^{2,0}(S));\\
H^{1,1}(H^*(S,\zz)\otimes \aaa^1_{\kappa})):=T_{-\frac{\xi}{r}}^{-1}(\oplus_{p=0}^{2}H^{p,p}(S));\\
H^{0,2}(H^*(S,\zz)\otimes \aaa^1_{\kappa})):=T_{-\frac{\xi}{r}}^{-1}(H^{0,2}(S)).
\end{cases}
$$
and this polarized Hodge structure is denoted by 
$$\left(H^*(S,\zz), \langle, \rangle, -\frac{\xi}{r}\right).$$
From \cite[Lemma 3.4]{Yoshioka2}, this Hodge structure 
$\left(H^*(S,\zz), \langle, \rangle, -\frac{\xi}{r}\right)$
depends only on the Brauer class 
$o([\xi \mod r])$, where 
$o(\xi)$ is the image under the map $o: H^2(S,\mu_r)\to H^2(S,\sO_S^*)$.

\begin{defn}\label{defn_integral_structure_K3}
For the projective bundle $P\to S$ and $G$ the locally free $P$-sheaf. Let 
$\xi\in H^2(S,\zz)$ be a lifting of 
$\omega(G)\in H^2(S,\mu_r)$, where $\rk(G)=r$. 
\begin{enumerate}
\item Define an integral Hodge structure of $H^*(S,\qq)$ as:
$$T_{-\frac{\xi}{r}}\left(\left(H^*(S,\zz), \langle, \rangle, -\frac{\xi}{r}\right)\right).$$
\item $v=(\rk, \zeta, b)$ is a Mukai vector if $v\in T_{-\frac{\xi}{r}}\left(H^*(S,\zz)\right)$ and $\zeta\in \Pic(S)\otimes \qq$.
Moreover, if $v$ is primitive in  $T_{-\frac{\xi}{r}}\left(H^*(S,\zz)\right)$, then $v$ is primitive. 
\end{enumerate}
\end{defn}

Now we let 
$$\cl_0: K(\sD_0^{\SS})\stackrel{\pi_{\SS*}}{\longrightarrow}K(\SS)\stackrel{\Ch}{\longrightarrow}\Gamma_0$$
with 
$$\Gamma_0=\widetilde{H}(\SS,\qq)\cap \widetilde{H}^{1,1}(\SS)=\qq\oplus\NS(\SS)\oplus \qq$$
and $\pi_{\SS*}E\in K(\SS)$. 
The twisted Mukai vector is given by: 
\begin{equation}\label{eqn_twisted_Mukai_vector}
v_{G}: K(\sD_0^{\SS})\stackrel{\pi_{\SS*}}{\longrightarrow}K(\SS)\stackrel{v_G}{\longrightarrow}\Gamma_0.
\end{equation}

We have the general Chern character:
$$\Ch: K(\overline{\XX})\to H^*(\overline{\XX},\qq).$$
Let 
$$\Gamma:=H^0(\overline{\XX},\qq)\oplus (\Gamma_0\boxtimes H^2(\pp^1,\qq))\subset H^*(\overline{\XX},\qq).$$
Then 
$$\cl=\Ch: K(\sD)\to \Gamma=\qq\oplus\Gamma_0$$
and every element in $\Gamma$ can be written as: 
$v=(R,\rk,\beta,n)$, where $R, \rk$ are integers and $\beta\in\NS(\SS)$.  For any $E\in\sD$, we have:
$$\cl(E)=(\Ch_0(E),\Ch_1(E), \Ch_2(E), \Ch_3(E)).$$
Here $\Ch_3(E)\in\qq$ since we work on $H^*(\overline{\XX},\qq)$. The rank of 
$v$ is $\rk(v)=R$ and 
$\Ch_1(E)=c_1(E)=\rk\cdot [\SS]$ and $\rk$ is an integer. 

\subsubsection{Moduli stack of twisted sheaves}

In \cite{Yoshioka2}, the moduli stack $\sM^{P,G}_{H,ss}(v)$ (or $\sM^{P,G}_{H,s}(v)$)
of $G$-twisted semistable (or stable) $P$-sheaves $E$ with $v_G(E)=v$ is defined.  Let 
$M^{P,G}_{H,ss}(v)$ (or $M^{P,G}_{H,s}(v)$) be its coarse moduli space.  The moduli stack $\N^{P,G}_{H,ss}(v)$ (or $\N^{P,G}_{H,s}(v)$)
of $G$-twisted semistable (or stable) Higgs $P$-sheaves $(E,\phi)$ with $v_G(E)=v$ is similarly defined, where  
$E$ is a $G$-twisted $P$-sheaf and 
$$\phi: p_*E\to p_*(E)\otimes K_{\SS}$$
is the Higgs field.  This makes sense since $p_*(E\otimes L^{\vee})$ is a $G$-twisted sheaf on $S$. 
Let 
$N^{P,G}_{H,ss}(v)$ (or $N^{P,G}_{H,s}(v)$) be its coarse moduli space. 

 On the other hand, for the optimal  $\mu_r$-gerbe $\SS\to S$, using geometric stability as in \cite{Lieblich_Duke},  reviewed in \cite[\S 3]{Jiang_2019} Lieblich defined the 
 moduli stack of $\sM^{s,\tw}_{\SS}(v)$ of stable $\SS$-twisted sheaves with Mukai vector $v$. We have the following result:
 
 \begin{thm}\label{thm_moduli_mur_gerbe}
The natural map 
 $\sM^{s,\tw}_{\SS}(v)\to \sM^{P,G}_{H,s}(v)$ is a $\mu_r$-gerbe. 
 \end{thm}

\subsection{The invariants $N(v_{G})$}\label{subsec_invariants_Nv_Jv}

\subsubsection{The heart of a bounded $t$-structure on $\sD_0^{\SS}$ and on $\sD$}\label{subsubsec_heart_sD}

We first recall the geometric stability of $\alpha$-twisted sheaves on $\SS\to S$ or 
$\overline{\XX}\to \overline{X}=S\times\pp^1$.  We present twisted sheaves on $\overline{\XX}$, the case of $\SS$ is similar. The geometric Hilbert polynomial of a twisted sheaf 
$E$ is defined as:
$$P^g(E,m)=\chi^g(E\otimes\sO_{\overline{X}}(m))$$
where $\chi^g(E):=[I\overline{\XX}:\overline{\XX}]\deg(\Ch(E)\cdot \Td_{\overline{\XX}})$. 
The geometric Hilbert polynomial can be written as:
$$P^g(E,m)=\sum_{i=0}^{\dim(E)}a_i(E)\frac{m^i}{i!}.$$
The rank is defined as:
$$\rk(E)=\frac{a_d(E)}{a_d(\sO_{\overline{\XX}})}$$
and 
$\deg(E):=a_{d-1}(E)-\rk(E)\cdot a_{d-1}(\sO_{\overline{\XX}})$. The slope is:
$$\mu(E)=\mu_{\omega}(E)=\frac{\deg(E)}{\rk(E)}. $$
Since we are interested in twisted sheaves $E$, where the $\mu_r$ action on $E$ is given by $\lambda^1$-action. By Grothendieck-Riemann-Roch theorem for DM stacks, for a twisted sheaf on $\SS$, 
\begin{equation}\label{eqn_slope_twisted_sheaf}
\mu_{\omega}(E)=\frac{\int_{\SS}c_1(E)p^*\omega}{\rk(E)},
\end{equation}
where $\omega$ is an ample divisor on $S$ (usually take $\omega=\sO_S(1)$). 

On the category $\Coh_{\pi}^{\tw}(\overline{\XX})$, for an element $E$ such that 
$$\cl_0(E)=(\rk,\beta,n)\in \zz\oplus \NS(\SS)\oplus \qq$$
the Hilbert polynomial $P^g(E,m)$ and the slope $\mu(E)$ is defined before.  Then an object 
$E\in  \Coh_{\pi}^{\tw}(\overline{\XX})$ is called Gieseker (semi)stable and $\mu_{\omega}$-(semi)stable if for any subsheaf $F\subsetneq E$, 
$p^g(F,m)< (\leq) p^g(E,m)$, $\mu_{\omega}(F)< (\leq) \mu_{\omega}(E)$ respectively.  Here $p^g$ is the reduced Hilbert polynomial. 
As in \cite[\S 2.5]{Toda_JDG}, for the $\mu_{\omega}$-stability on  $\Coh_{\pi}^{\tw}(\overline{\XX})$,  we have the Harder-Narasimhan filtration
$$0\to E_0\subset E_1\subset\cdots\subset E_{\ell}=E$$
for $E\in  \Coh_{\pi}^{\tw}(\overline{\XX})$ which satisfies
$$\mu_{\omega}(E_1/E_0)>\cdots>\mu_{\omega}(E_{\ell}/E_{\ell-1}).$$
Using this we construct a torsion pair. Let 
$$\mu_{\omega,+}:=\mu_{\omega}(E_1/E_0); \quad  \mu_{\omega,-}:=\mu_{\omega}(E_{\ell}/E_{\ell-1}).$$

\begin{defn}\label{defn_torsion_pair}
Define $(\sT_{\omega}, \sF_{\omega})$ to be the torsion pair of full subcategories in $\Coh_{\pi}^{\tw}(\overline{\XX})$ as:
$$\sT_{\omega}:=\{E\in \Coh_{\pi}^{\tw}(\overline{\XX})| \mu_{\omega,-}(E)>0\}$$ and 
$$\sF_{\omega}:=\{E\in \Coh_{\pi}^{\tw}(\overline{\XX})| \mu_{\omega,+}(E)\leq 0\}.$$
\end{defn}
Then the pair $(\sT_{\omega}, \sF_{\omega})$ is a torsion pair in the sense that:
\begin{enumerate}
\item Any $T\in\sT_{\omega}$, $F\in \sF_{\omega}$, we have $\Hom(T,F)=0$;
\item For any $E\in \Coh_{\pi}^{\tw}(\overline{\XX})$, there exists an exact sequence
$$0\to T\to E\to F\to0$$
such that $T\in\sT_{\omega}$, and $F\in\sF_{\omega}$. 
\end{enumerate}
Let 
$$\sB_{\omega}:=\langle \sF_{\omega}, \sT_{\omega}[-1]\rangle_{\ex}\subset \sD_0^{\SS}$$
be the tilting in \cite{HRS}. Then $\sB_{\omega}$ is the heart of a bounded $t$-structure on $\sD_0^{\SS}$.
We remark that $\sB_{t\omega}$ is the same as $\sB_{\omega}$ for $t>0$.

\begin{prop}
Define 
$$\sA_{\omega}:=\langle\pi^*\Pic(\pp^1), \sB_{\omega}\rangle_{\ex}\subset \sD$$
then $\sA_{\omega}$ is the heart of a bounded $t$-structure on $\sD$. 
\end{prop}
\begin{proof}
The argument is the same as \cite[Proposition 2.9]{Toda_JDG}.
\end{proof}

\subsubsection{Toda's weak stability condition on $\sD_0^{\SS}$ and $\sD$}

Recall we have Toda's weak stability condition as in Definition \ref{defn_weak_stability_Toda}. For our $\Gamma=\zz\oplus \Gamma_0$, let 
$$0=\Gamma_{-1}\subset \Gamma_0\subset \Gamma_1=\Gamma$$
be the filtration.  For $\omega$ an ample divisor on $S$ and $t\in\rr_{>0}$, let 
$$\sZ_{t\omega}\in \prod_{i=0}^{1}\Hom(\Gamma_i/\Gamma_{i-1},\cc)$$
be given by:
$$
\begin{cases}
\sZ_{t\omega,0}(v):=\int_{\SS}e^{-it\omega}v, & v\in\Gamma_0;\\
\sZ_{t\omega,1}(R)=R\sqrt{-1}, & R\in \Gamma_1/\Gamma_0=\zz.
\end{cases}
$$
Then \cite[Lemma 3.4]{Toda_JDG} says that 
$$\sigma_{t\omega}:=(\sZ_{t\omega}, \sA_{\omega})$$
is a weak stability condition. This can be checked that \cite[\S 8.2]{Toda_JDG} 
works for sheaves in $D^b(\Coh_{\pi}^{\tw}(\overline{\XX}))$ and $\sD$. 
The wall and chamber structure and some comparison with $\mu_{i\omega}$-limit semistable objects are in 
\cite[\S 3.4, \S 3.5]{Toda_JDG}.

\subsubsection{Hall algebra $H(\sA_{\omega})$ and counting invariants}

The Hall algebra $H(\sA_{\omega})$ for the abelian category 
$\sA_{\omega}$ is defined in \cite{Bridgeland10}.  The stable pair invariants of \cite{PT} can be defined on 
$\sA_{\omega}\subset \sD$ as in \S \ref{subsec_stable_pair_invariants_DM}, and Toda proves a wall crossing formula for the stable pair invariants.
Here we only care about the rank zero invariants. 

Let 
$\widehat{\sM}(\sA_{\omega})$ be the stack of objects in $\sA_{\omega}$, and 
$$\sM_{t\omega}(v)\subset \widehat{\sM}(\sA_{\omega})$$
the substack parametrizing $\sZ_{t\omega}$-semistable objects $E\in \sA_{\omega}$ and 
$\cl(E)=v\in\Gamma$ with $\rk(v)\leq 1$. Then the moduli stack gives an element
$$\delta_{t\omega}(v):=[\sM_{t\omega}(v)\hookrightarrow \widehat{\sM}(\sA_{\omega})]\in H(\sA_{\omega}).$$
As in Joyce \cite{Joyce03}, let 
\begin{equation}\label{eqn_epsilon_K3}
\epsilon_{t\omega}(v):=\sum_{\substack{\ell\geq 1, v_1+\cdots+v_{\ell}=v\\
v_i\in\Gamma\\
\arg\sZ_{t\omega}(v_i)=\arg\sZ_{t\omega}(v)}}\frac{(-1)^{\ell-1}}{\ell}\delta_{t\omega}(v_1)\star\cdots\star
\delta_{t\omega}(v_{\ell})
\end{equation}
This sum is a finite sum \cite[Lemma 4.8]{Toda_JDG}, and $\delta_{t\omega}(v)$ is well-defined.  The proof is from Bogomolov inequality for semistable sheaves. 
We define a morphism 
$$P_q: H(\sA_{\omega})\to \qq(q^{\frac{1}{2}})$$
by the virtual Poincar\'e polynomial 
as in Definition \ref{defn_DT_generalized_JS} with $t$ replaced by $q^{\frac{1}{2}}$. 
\cite[Theorem 6.2]{Joyce_Motivic_Invariants} shows that 
\begin{equation}\label{eqn_limit_epsilon}
\lim_{q^{\frac{1}{2}}\to 1}(q-1)P_q(\epsilon_{t\omega}(v))\in\qq
\end{equation}
exists. 

\begin{defn}\label{defn_DT_K3_gerbe}
For $v\in\Gamma_0$, define
$$\DT^{\chi}_{t\omega}(v):=\lim_{q^{\frac{1}{2}}\to 1}(q-1)P_q(\epsilon_{t\omega}(1,-v))$$
for $(1,-v)\in\Gamma=\zz\oplus\Gamma_0$. This is the Euler characteristic version of the DT-invariants. 
\end{defn}

\subsubsection{Counting rank zero invariants}

\cite[\S 4.6]{Toda_JDG} uses the DT-invariants $\DT^{\chi}_{t\omega}(v)$ to recover the invariants $L_{\beta,n}$ of the limit stable objects and study the wall crossing formula for the stable pair invariants on $\overline{\XX}$. 
Since in this paper we don't need this $\PT^{\tw}(\overline{\XX})$, we leave this for later study. 

To define the counting invariants in $\sA_{\omega}$ and $\sA_{\omega}[-1]$, let 
$$C(\sB_{\omega}):=\Im(\cl_0: \sB_{\omega}\to \Gamma_0).$$
From \cite[Definition 4.17]{Toda_JDG}, 
\begin{defn}\label{defn_invariants_Nv}
Define $N(v)\in\qq$ for $v\in\Gamma_0$, to be
\begin{enumerate}
\item If $v\in C(\sB_{\omega})$, 
$$N(v)=\lim_{q^{\frac{1}{2}}\to 1}(q-1) P_q(\epsilon_{\omega}(0,v));$$
\item If $-v\in C(\sB_{\omega})$, then $N(v):=N(-v)$;
\item If $\pm v\notin C(\sB_{\omega})$, $N(v)=0$.
\end{enumerate}
\end{defn}
This definition is the same as the invariants of $\sZ_{\omega,0}$-semistable objects $E\in\sB_{\omega}$
satisfying $\cl_0(E)=v$. 

\begin{rmk}
In the case of $v=(0,\beta,n)$, the invariant $N(0,\beta,n)$ in the above definition is the same as the one we defined in 
(\ref{eqn_Toda66}).
\end{rmk}

\subsection{Twisted Hodge isometry}\label{subsec_twisted_Hodge_isometry}

Let $\Stab^\circ_{\Gamma_0}(\sD_0^{\SS})$ be the component in $\Stab_{\Gamma_0}(\sD_0^{\SS})$  containing the stability conditions $(\sZ_{t\omega,0}, \sB_{\omega})$. 

For the twisted K3 surface $\SS_{\alpha}$ or $(S,\alpha)$, \cite{HMS} has generalized Bridgeland \cite{Bridgeland_K3} stability conditions on K3 surface to twisted K3 surfaces; and proved that there 
exists a component of maximal dimension
$$\Stab^\circ(S,\alpha)\subset \Stab(S,\alpha)$$
 in $\Stab(S,\alpha)$. 
 
 \begin{thm}\label{thm_stability_twisted_K3_sD}
 We have 
 $$\Psi: \Stab^\circ_{\Gamma_0}(\sD_0^{\SS})\stackrel{\sim}{\longrightarrow}\Stab^\circ(S,\alpha).$$
 \end{thm}
 \begin{proof}
 We recall the stability condition $\Stab^\circ(S,\alpha)\subset \Stab(S,\alpha)$ in \cite{HMS}.
 We have
 $$\sZ: \Stab(S,\alpha)\to \NS(S,\alpha)\otimes\cc$$
 given by:
 $$\sigma\mapsto (\sZ,P_{\sigma})$$
 where $\NS(S,\alpha)=\widetilde{H}^{1,1}(S,\alpha,\zz)$. 
 Let $P(S,\alpha)\subset  \NS(S,\alpha)\otimes\cc$ be the open subset of vectors $\varphi$ such that the real part and imaginary part of $\varphi$
 generate a positive plane in  $\NS(S,\alpha)\otimes\cc$.
 $P(S,\alpha)$ has two connected components and let 
$P^{+}(S,\alpha)$ be the one containing 
$$\varphi=e^{B_0+i\omega}.$$
Here $B_0\in H^2(S,\qq)$ is a $B$-field lift of $\alpha$, i.e., 
$$H^2(S,\qq)\to H^2(S,\sO_S^*)$$
by
$$B\mapsto \alpha=\exp(B).$$
Let 
$B_1\in\NS(S)\otimes\rr$ and $B:=B_1+B_0$, then 
$$\varphi=e^{B+i\omega}\in \NS(S,\alpha)\otimes\cc$$
and 
$$\sZ_{\varphi}(E):=\langle v_{G}(E),\varphi\rangle.$$
This central charge 
$\sZ_{\varphi}(E)$ determines a torsion pair $(\sT,\sF)$, where $\sT\subset \Coh(S,\alpha)$ is given by twisted sheaves $E$ such that every 
nontrivial torsion free quotient $E\twoheadrightarrow E^\prime$ satisfies $\Im(\sZ_{\varphi}(E^\prime))>0$.  $\sF\subset \Coh(S,\alpha)$ is given by twisted sheaves 
$E\in\sF$ if $E$ is torsion free and every non-zero subsheaf $E^\prime\subset E$ satisfies 
$\Im(\sZ_{\varphi}(E^\prime))\leq 0$. Then every $G\in \Coh(S,\alpha)$ can be uniquely written as 
$$0\to E\to G\to F\to 0$$
for $E\in\sT$, $F\in \sF$. 
The heart of the induced $t$-structure is an abelian category:
$$\sA(\varphi):=\left\{E\in D^b(\Coh(S,\alpha))  \left| \begin{array}{l}
\text{$\bullet\,\, \hH^{i}(E)=0$ for $i\notin \{-1,0\}$}\\
\text{$\bullet\,\, \hH^{-1}(E)\in\sF$} \\
\text{$\bullet\,\, \hH^{0}(E)\in\sT$}
\end{array}\right
\}\right.$$
Then from \cite[Lemma 3.4]{HMS}, let 
$\varphi=e^{B+i\omega}$, the induced homomorphism
$$\sZ_{\varphi}: \sA(\varphi)\to \cc$$
is a stability function on $\sA(\varphi)$ if and only if for any spherical twisted sheaf
$E\in \Coh(S,\alpha)$, $\sZ_{\varphi}(E)\notin \rr_{\leq 0}$. 
From \cite[Proposition 3.6]{HMS}, the pair $(\sZ_{\varphi}, \sA(\varphi))$ defines a stability condition on 
$D^b(\Coh(S,\alpha))$.  Then let 
$\Stab^{\circ}(S,\alpha)\subset \Stab(S,\alpha)$ be the connected component of $\Stab(S,\alpha)$
that contains the stability conditions described above. Then since the category $\Coh_{\pi}^{\tw}$ can be taken as a family of K3 gerbes 
$\SS\to S$, the same argument as in \cite[Theorem 6.5]{Toda_JDG}
gives
\[
\xymatrix{
 \Stab^{\circ}(\sD_0^{\SS})\ar[r]^{\theta^{\circ}}\ar[d]_{\sZ}& \Stab^{\circ}(S,\alpha)\ar[d]^{\sZ}\\
 \NS(\overline{\XX})\otimes\cc\ar[r]^{i^*}& \NS(S)\otimes\cc
}
\]
We explain the notation:
$$\theta: \Stab(\sD_0^{\SS})\to \Stab(S,\alpha)$$
is a map defined in \cite[Proposition 5.1]{Toda_JDG}, and 
$\theta^{\circ}$ is the restriction to $\Stab^{\circ}(\sD_0^{\SS})$. Let 
$$i: \SS\hookrightarrow \SS\times\pp^1=\overline{\XX},$$
then 
$i^*: \NS(\overline{\XX})\to \NS(\SS)$ is the pullback. 
This diagram gives a homeomorphism between 
$\Stab^{\circ}(\sD_0^{\SS})$ and one of the connected components of $\Stab^{\circ}(S,\alpha)\times_{\NS(\SS)\otimes\cc}\NS(\overline{\XX})$.
\end{proof}

Now on the derived category $\sD_0^{\SS}$, if we have a stability condition 
$$\sigma=(\sZ,\sA)\in\Stab^{\circ}_{\Gamma_0}(\sD_0^{\SS})$$
we define the invariants $N_{\sigma}(v)\in\qq$ as the Euler characteristic version of the generalized Donaldson-Thomas invariants counting $\sZ$-semistable objects $E\in\sA$ or $\sA[-1]$ with 
$\cl_0(E)=v$ in Definition \ref{defn_invariants_Nv}. 

Here is a result proved in \cite[Theorem 4.21]{Toda_JDG}:

\begin{thm}\label{thm_invariants_Nv_changing}
The counting invariants $N_{\sigma}(v)$ do not depends on a choice of stability condition $\sigma\in\Stab^{\circ}_{\Gamma_0}(\sD_0^{\SS})$ and $N_{\sigma}(v)$ is also independent of $\omega$.
\end{thm}
\begin{proof}
The proof is the same as \cite[Theorem 4.21]{Toda_JDG} by using Joyce's wall crossing formula:
$$N_{\sigma_1}(v)=N_{\sigma_0}(v)+
\sum_{v_1+v_2=v}U_{v_1,v_2}\chi(v_1,v_2)N_{\sigma_0}(v)\cdot N_{\sigma_1}(v)+\cdots\cdots$$
for some other terms (which involve product of $\chi(v_i, v_j)$). Here $\sigma_0$ and $\sigma_1$ are two different stability conditions and 
$\sigma_0$ is close to $\sigma_1$. It suffices to prove that 
$\chi(v_i, v_j)=0$ for any $v_1, v_2\in\Gamma_0$.
Let 
$$\widetilde{v}_1=(0,v_1), \quad \widetilde{v}_2=(0,v_2)\in\Gamma=\zz\oplus\Gamma_0.$$
Then by Grothendieck-Riemann-Roch theorem for stacks
$$\chi(\widetilde{v}_1, \widetilde{v}_2)=
\chi(\overline{v}_1, \overline{v}_2)=0$$
where $\overline{v}_1, \overline{v}_2$ are the corresponding Mukai vectors on $S$. 
\end{proof}

\subsubsection{Counting invariants on $\XX:=\SS\times\cc$}

We are actually  interested in the counting invariants on the open local  K3  gerbe $\XX:=\SS\times\cc$. Let 
$$\Coh_{\pi}^{\tw}(\XX)\subset \Coh_{\pi}^{\tw}(\overline{\XX})$$
be the subcategory consisting of sheaves supported on the fibers of $\pi|_{\XX}: \XX\to\cc$. 
Still let $\widehat{\sM}(\XX)$ be the stack of coherent twisted sheaves in $\Coh_{\pi}^{\tw}(\XX)$, which is an Artin stack locally of finite type. 
Let 
$\sA_{\XX}:=\Coh_{\pi}^{\tw}(\XX)$ and $H(\sA_{\XX})$ the Hall algebra in \cite{Bridgeland10}, which consists of elements over the stack $\widehat{\sM}(\XX)$. 
For $v\in\Gamma_0$, we let $\sM_{\omega,\XX}(v)\subset \widehat{\sM}(\XX)$ be the substack of $\omega$-Gieseker semistable twisted sheaves 
$E$ with $v_G(E)=v$.  This gives an element:
$$\delta_{\omega,\XX}(v):=[\sM_{\omega,\XX}(v)\hookrightarrow \widehat{\sM}(\XX) ]\in H(\sA_{\XX})$$
and define its ``logarithm":
$$
\epsilon_{\omega,\XX}(v):=
\sum_{\substack{\ell\geq 1, v_1+\cdots+v_{\ell}=v, v_i\in\Gamma_0\\
p_{\omega,v_i}=p_{\omega,v}(m)}}
\frac{(-1)^{\ell}-1}{\ell}\delta_{\omega,\XX}(v_1)\star\cdots\star \delta_{\omega,\XX}(v_{\ell})
$$
where $p_{\omega,v_i}(m)$ is the reduced geometric Hilbert polynomial. The sum is a finite sum and 
$\epsilon_{\omega,\XX}(v)$ is well-defined.  Let 
$$C(\XX):=\Im(v_G: \Coh_{\pi}^{\tw}(\XX)\to\Gamma_0)$$
we define the invariants: 

\begin{defn}\label{defn_invariants_Jv_XX}
For $v\in\Gamma_0$, define the invariant $J(v)\in\qq$ as follows:
\begin{enumerate}
\item  If $v\in C(\XX)$, 
$$J(v):=\lim_{q^{\frac{1}{2}}\to 1}(q-1) P_q(\epsilon_{\omega,\XX}(v)).$$
\item If $-v\in C(\XX)$, 
$J(v):=J(-v)$.
\item  If $\pm v\notin C(\XX)$, $J(v);=0$.
\end{enumerate}
\end{defn}

\begin{rmk}
If $v_G(E)=v$ is primitive, then 
$$J(v)=\chi(\sM_{\omega,\XX}^{\tw}(v))=\chi(\sM_{\omega,\SS}^{\tw}(v)\times\cc)$$
where $\sM_{\omega,\SS}^{\tw}(v)$ is the moduli stack of $\SS$-twisted sheaves with 
Mukai vector $v$. From Yoshioka \cite{Yoshioka2},  $\sM_{\omega,\SS}^{\tw}(v)$ is deformation equivalent to the Hilbert scheme $\Hilb^{1-\frac{\chi(v,v)}{2}}(S)$. This is essential to the calculation of Vafa-Witten invariants  in \cite{Jiang_2019}.
\end{rmk}

\subsubsection{The comparison between $N(v)$ and $J(v)$}

We also need to compare $N(v)$ and $J(v)$.  Similar to Definition \ref{defn_invariants_Jv_XX}, we can define the counting invariants 
$$\overline{J}(v)\in\qq$$
to be the invariants counting geometric $\omega$-Gieseker semistable twisted sheaves 
$E\in \Coh_{\pi}^{\tw}(\overline{\XX})$ with $v_G(E)=v\in\Gamma_0$. 
We have a generalization of \cite[Theorem 4.24]{Toda_JDG}:

\begin{thm}\label{thm_Jbar_N}
For any $v\in\Gamma_0$, we have:
$$\overline{J}\left(v\cdot \frac{\sqrt{\td_{S}}}{\sqrt{\Ch(Rp_*(G\otimes G^{\vee}))}}\right)=N(v)$$
\end{thm}
\begin{proof}
We need to generalize \cite[Theorem 6.6]{Toda_Adv} to $P$-sheaves on $\Coh(S,P)$, where 
$P\to S$ is the Brauer-Severi variety corresponding to the optimal gerbe 
$\SS\to S$.   

Recall the definition of $N(v)$, given by a weak stability condition 
$$\sigma_{\varphi}=(\sZ_{\varphi}, \sA(\varphi))\in\Stab^{\circ}_{\Gamma_0}(\sD_0^{\SS})$$
and 
$N(v)$ is independent to the stability condition we choose.  Here 
$\varphi= e^{B+i\omega}$.  Let us take 
$$\sigma_k=(\sZ_{k\omega}, \sA_{\omega})$$
and $\varphi=e^{B+ik\omega}$, and 
$$\sZ_{\varphi}(E)=\langle e^{B+ik\omega}, v_G(E)\rangle.$$
From Theorem \ref{thm_J_N} in the Appendix, $E\in\sB_{\omega}$ is $\sZ_{\varphi}$-semistable with 
$\cl_0(E)=v$ if and only if $E$ is Gieseker semistable with 
$v_G(E)=v\cdot \frac{\sqrt{\td_{S}}}{\sqrt{\Ch(Rp_*(G\otimes G^{\vee}))}}$.
\end{proof}

Now we use Toda's method to show that 
$$J_{\omega}(v)=J_{\omega}(g (v))$$
for $g$ a Hodge isometry on $\widetilde{H}(S,\alpha, \zz)$. 
First on the gerbe $\overline{\XX}=\SS\times\pp^1$, we mimic the definition of $J(v)$ in 
Definition \ref{defn_invariants_Jv_XX} to define $\overline{J}(v)$. This invariant is defined by:
$$\delta_{\omega,\overline{\XX}}=
[\sM^{\tw}_{\omega,\overline{\XX}}(v)\hookrightarrow \Coh^{\tw}_{\pi}(\overline{\XX})]\in H(\sA_{\overline{\XX}})$$
and $\sM^{\tw}_{\omega,\overline{\XX}}(v)$ is the moduli stack of $\omega$-Gieseker semistable twisted sheaves $E\in \Coh^{\tw}_{\pi}(\overline{\XX})$ with $\cl_0(E)=v$. 

For $\overline{\XX}=\SS\times\pp^1$, and each point $p\in\pp^1$, let 
$$\sU_p:=\overline{\XX}\setminus \XX_p$$
where $\XX_p$ is the fiber of $\SS\times\cc\to\cc$ over $p$.  Denote by 
$\sM^{\tw}_{\omega,\sZ}(v)\subset \sM^{\tw}_{\omega,\overline{\XX}}(v)$ the locus of 
$E\in \Coh_{\pi}^{\tw}(\overline{\XX})$ such that $\supp(E)\subseteq \sZ$
for $\sZ\subset \overline{\XX}$ an open or closed substack. 
Let 
$$\delta_{\omega,\sZ}(v)=[\sM^{\tw}_{\omega,\sZ}(v)\hookrightarrow\widehat{\sM}(\overline{\XX})]\in H(\sA_{\overline{\XX}})$$
and define  $\epsilon_{\omega,\sZ}(v)$ accordingly.  The following lemma is \cite[Lemma 4.25]{Toda_JDG}. 

\begin{lem}\label{lem_decomposition_epsilon}
$$\epsilon_{\omega,\overline{\XX}}(v)=\epsilon_{\omega,\sU_p}(v)+ \epsilon_{\omega,\XX_p}(v)$$
\end{lem}
\begin{proof}
Since the definition of $\epsilon_{\omega,\overline{\XX}}(v)$, (resp. $\epsilon_{\omega,\sU_p}(v)$,  
$\epsilon_{\omega,\XX_p}(v)$) is given by 
$\delta_{\omega,\overline{\XX}}$ (resp. $\delta_{\omega,\sU_p}$, $\delta_{\omega,\XX_p}$) by the logarithm  with the same coefficients as before. The proof is the same as \cite[Lemma 4.5]{Toda_JDG}.
\end{proof}

\begin{lem}\label{lem_lemma2}
We have:
$$\overline{J}(v)=2 J(v).$$
\end{lem}
\begin{proof}
The proof uses the Behrend function techniques in \cite{JS}.  First the invariant
$$\overline{J}(v)=\chi(M^{\tw}_{\omega,\overline{\XX}}(v),\nu_{\sM})$$
is the weighted Euler characteristic of the coarse moduli space $M^{\tw}_{\omega,\overline{\XX}}(v)$ of 
$\omega$-Gieseker semistable twisted sheaves, and $\nu_{\sM}$ is the Behrend function. 
If $M^{\tw, p}_{\omega,\overline{\XX}}(v)\subset M^{\tw}_{\omega,\overline{\XX}}(v)$ is the closed subscheme consisting of semistable twisted sheaves $E$ with $\supp(E)\subset \XX_p$ for $p\in\pp^1$, then from Lemma \ref{lem_decomposition_epsilon}, the constructible function $\nu_{\sM}$ is zero outside $\XX_p$, and we have:
$$M^{\tw, p}_{\omega,\overline{\XX}}(v)\cong M^{\tw}_{\omega,\XX_0}(v)\times\pp^1$$
for $\XX_0=\SS$.
Any $g$ in the automorphism group $\Aut(\pp^1)$ sends a point $p$ to $q$, and the element 
$\epsilon_{\XX_p}(v)$ is mapped to $\epsilon_{\XX_q}(v)$ by $g$. Therefore
$$\overline{J}(v)=\chi(\pp^1)\cdot \chi(M^{\tw}_{\omega,\XX_0}(v)\times\{0\}, \nu_{\sM})$$
and 
$$J(v)=\chi(\cc)\cdot \chi(M^{\tw}_{\omega,\XX_0}(v)\times\{0\}, \nu_{\sM}).$$
\end{proof}

So we have:
$${J}\left(v\cdot \frac{\sqrt{\td_{S}}}{\sqrt{\Ch(Rp_*(G\otimes G^{\vee}))}}\right)=\frac{1}{2}N(v),$$
from Theorem \ref{thm_Jbar_N} and Lemma \ref{lem_lemma2}.

\subsubsection{Automorphic property}

Let $G$ be the group of Hodge isometries
$$G:=O_{\Hodge}(\widetilde{H}(S,\alpha,\zz), (\star,\star))$$
on the twisted K3 surface $(S,\alpha)$. 
We prove the automorphic property for $\overline{J}(v)$. 
For two twisted K3 surfaces $(S,\alpha)$ and $(S^\prime,\alpha^\prime)$, we recall the twisted Fourier-Mukai transform 
in \cite{HMS}, \cite{HS}. For the 
$B$-field $B$ such that $\exp(B)=\alpha$; $B$-field $B^\prime$ such that $\exp(B^\prime)=\alpha^\prime$, we have
$(-B)\boxplus B^\prime:=p_{S}^*(-B)+p_{S^\prime}^*(B^\prime)\in H^2(S\times S^\prime, \qq)$, where 
$p_{S}: S\times S^\prime\to S$, and $p_{S^\prime}: S\times S^\prime\to S^\prime$ are projections and induce
$\alpha^{-1}\boxtimes \alpha^\prime\in H^2(S\times S^\prime, \sO^*)$.  
Let 
$\rE\in D^b(S\times S^\prime, \alpha^{-1}\boxtimes \alpha^\prime)$ be a kernel, we have:
$$\Phi: D^b(S,\alpha)\to D^b(S^\prime,\alpha^\prime)$$
given by
$$E\mapsto p_{S^\prime *}(p_{S}^*(E)\otimes\rE).$$
\cite[Proposition 4.3]{HS} implies that if $\Phi$ is a derived equivalence, then the induced map 
$$\Phi_*^{B,B^\prime}: \widetilde{H}(S,\alpha,\zz)\stackrel{\sim}{\longrightarrow}\widetilde{H}(S^\prime, \alpha^\prime,\zz)$$
is a Hodge isometry of integral weight-2 Hodge structures, and 
$$\Phi_*^{B,B^\prime}(-)=p_{S^\prime *}p_{S}^*\left(-\cdot \Ch(\rE)\sqrt{\td_{S\times S^\prime}}\right).$$
We have the following commutative diagram:
\[
\xymatrix{
D^b(\Coh(S^\prime,\alpha^\prime))\ar[r]^{\Phi}\ar[d]^{v_{\alpha^\prime}}& D^b(\Coh(S,\alpha))\ar[d]_{v_{\alpha}}\\
\widetilde{H}(S^\prime, \alpha^\prime,\zz)\ar[r]^{\Phi_*}& \widetilde{H}(S,\alpha,\zz)
}
\]
Therefore this equivalence $\Phi$ gives an isomorphism
$$\Phi_{\St}: \Stab(S^\prime,\alpha^\prime)\stackrel{\sim}{\longrightarrow}\Stab(S,\alpha)$$
on the stability manifolds.  We have a similar proposition as in \cite[Proposition 4.29]{Toda_JDG}.

\begin{prop}\label{prop_auto_property_Jbar}
Let $\sD_0^{\SS}$, $\Gamma^{S}_0$, $\overline{J}_S(v)$ be the invariants defined before for the twisted K3 surface 
$(S,\alpha)$. Assume that $\Phi_{\St}$ sends the connected component $\Stab^{\circ}(S^\prime, \alpha^\prime)$ to 
$\Stab^{\circ}(S,\alpha)$. Then 
$$\overline{J}_{S^\prime}(v)=\overline{J}(\Phi_*(v))$$
for any $v\in \Gamma_0^{S^\prime}$. 
\end{prop}
\begin{proof}
First $\Phi$ induces an equivalence
$$\widetilde{\Phi}: D^b(\overline{\XX}^\prime, \alpha^\prime)\stackrel{\sim}{\longrightarrow} D^b(\overline{\XX}, \alpha)$$
and the kernel is:
$$\rE\boxtimes \sO_{\Delta_{\pp^1}}\in D^b(\Coh(S^\prime\times S\times\pp^1\times\pp^1, \alpha^{-1}\boxtimes\alpha^\prime)).$$
$\widetilde{\Phi}$ restricts to give the equivalence:
$$\Phi: \sD_0^{\SS^\prime}\stackrel{\sim}{\longrightarrow}\sD_0^{\SS}.$$
Thus we have a commutative diagram:
\[
\xymatrix{
\sD_0^{\SS^\prime}\ar[r]^{{\Phi}}\ar[d]_{\cl_0\sqrt{\td_{S^\prime}}}& \sD_0^{\SS}\ar[d]^{\cl_0\sqrt{\td_S}}\\
\Gamma_0^{S^\prime}\ar[r]^{\Phi_*}& \Gamma_0^{S}
}
\]
Also from Theorem \ref{thm_stability_twisted_K3_sD}, 
$\Stab(\sD_0^{\SS})\stackrel{\sim}{\longrightarrow}\Stab(S,\alpha)$, 
and 
$$\Phi_{\St}: \Stab^{\circ}_{\Gamma_0^{S^\prime}}(\sD_0^{\SS^\prime})\stackrel{\sim}{\longrightarrow}
\Stab^{\circ}_{\Gamma_0^{S}}(\sD_0^{\SS}).$$
So from Theorem \ref{thm_Jbar_N} and the diagram above, we have
\begin{align*}
\overline{J}_{S_{\alpha}}(\Phi_* v)&=N_{\sigma}(\Phi_* v\cdot \sqrt{\td_{S_{\alpha}}}^{-1})\\
&=N_{\Phi_{\St}\sigma^\prime}(\Phi_* v\cdot \sqrt{\td_{S_{\alpha}}}^{-1})\\
&=N_{\sigma^\prime}( v\cdot \sqrt{\td_{S_{\alpha^\prime}}}^{-1})\\
&=\overline{J}_{S_{\alpha^\prime}}(v).
\end{align*}
\end{proof}

\begin{cor}\label{cor_Jv_Jgv}
Let $g\in G$ and $v\in\Gamma_0$, we have 
$$\overline{J}(gv)=\overline{J}(v)$$
and hence 
$$J(gv)=J(v).$$
\end{cor}

\subsubsection{Multiple cover formula}

In \cite[Conjecture 1.3]{Toda_JDG}, let $v\in\Gamma_0$ be an algebraic class, then the multiple cover formula for K3 surface is:
\begin{equation}\label{eqn_multiple_cover_K3_Toda}
 J(v)=\sum_{\substack{k\geq 1\\
 k|v}}\frac{1}{k^2}\chi(\Hilb^{\langle v/k,v/k\rangle+1}(S))
\end{equation}
which is the multiple cover formula Conjecture \ref{eqn_multiple_cover_formula} for 
$v=(0, \beta, n)$. 

In the category $\Coh^{\tw}_{\pi}(\XX)$ of twisted sheaves on $\XX\to X=S\times\pp^1$, given by the cohomology class 
$$\alpha\in H^2(X,\mu_r)\cong H^2(S,\mu_r),$$
if we forget about the $\mu_r$-action, this is equivalent to the category $\Coh(S,\alpha)$ of twisted sheaves on the twisted K3 surface $(S,\alpha)$. 
Since a one dimension sheaf $E$ with $\cl_0(E)=(0,\beta,n)$ on $\SS$ or on $\XX$ is automatically 
$\SS$-twisted or $\XX$-twisted, the element $\epsilon_{\omega, \XX}(v)$ is the same as 
$$\epsilon_{\omega, X}(v)= \sum_{\substack{\ell\geq 1, v_1+\cdots+v_{\ell}=v, v_i\in\Gamma_0\\
p_{\omega,v_i}=p_{\omega,v}(m)}}
\frac{(-1)^{\ell}-1}{\ell}\delta_{\omega,X}(v_1)\star\cdots\star \delta_{\omega,X}(v_{\ell}).$$
This is because in the category  $\Coh(S,\alpha)$, the subcategory of one dimensional sheaves is always a twisted category and keeps the same as the subcategory of one dimensional sheaves  in the general 
untwisted category  $\Coh(\SS_0)$ for the trivial gerbe $\SS_0$.  Thus the subcategory $\Coh(S,\alpha)_{\leq 1}$ is equivalent to the subcategory 
$\Coh(S)_{\leq 1}$. 
Another explanation can be seen from the equivalence (\ref{eqn_equivalence_twist_untwist}) in \S \ref{subsubsec_twisted_Chern_character} using $P$-sheaves, and the one dimensional pure $P$-sheaves corresponds to 
one dimensional pure twisted sheaves in $\Coh(S,\alpha)$.  The Brauer class $\alpha$, when restricted to the one dimensional supported locus of the twisted sheaf, is zero. Therefore the one dimensional twisted sheaves on 
$S$ is actually equivalent to the one dimensional sheaves on $S$.
Then the geometric semistability of a one dimensional sheaf as in \S \ref{subsubsec_heart_sD} corresponds to the general Gieseker semistability 
of the corresponding untwisted sheaf. 
Therefore the invariants
$J(v)$ defined in the category $\Coh(S,\alpha)$ or $\Coh(\overline{\XX}, \alpha)$ is the same as 
$J_{X}(v)$ if $v=(0,\beta,n)$.  
So we get:

\begin{prop}\label{prop_multiple_cover_formula}
In the category $\Coh^{\tw}_{\pi}(\XX)$ of  twisted sheaves, the multiple cover formula is still
(\ref{eqn_multiple_cover_K3_Toda}). Thus we prove Theorem \ref{thm_twisted_multiple_cover_intro}.
\end{prop}

%%%%%%%%%%%%%
\appendix{}

\section{The  invariants for twisted sheaves on twisted K3 surfaces}\label{Appendix_twisted_K3}

In this appendix we generalize the result in \cite{Toda_Adv} for comparing the counting invariants of semistable objects and invariants of counting semistable sheaves  in the derived category of coherent sheaves on K3 surfaces to twisted K3 surfaces. 

\subsection{Review of stability conditions on twisted K3 surfaces}\label{subsec_stability_twisted_K3}

\subsubsection{Twisted K3 surfaces}

Let $S$ be a smooth projective K3 surface.  Let $A$ be an abelian group scheme on $S$. An $A$-gerbe $\SS\to S$ is a DM stack $\SS$ over $S$ such that for any open subset 
$U\subset S$, there exists a covering $U^\prime\to U$ such that $\SS(U^\prime)\neq \emptyset$, and any sections $s,s^\prime\in \SS(U)$ there exists $U^\prime\to U$ such that
$s|_{U}=s^{\prime}|_{U^\prime}$. 
Also let $(A)_{\SS}$ be the sheaf of abelian groups $A$ on $\SS$, we have 
$(A)_{\SS}\cong I\SS$, where $I\SS$ is the inertia stack of $\SS$. 

We mainly take $A=\mu_r$ and $\cc^*$. Consider the following exact sequence:
$$1\to \mu_r\to \cc^*\stackrel{(\cdot)^r}{\longrightarrow}\cc^*\to 1$$
and taking cohomology:
$$\cdots\to H^1(S,\cc^*)\stackrel{\psi}{\longrightarrow}H^2(S,\mu_r)\stackrel{\varphi}{\longrightarrow}H^2(S,\cc^*)\to \cdots$$
A $\mu_r$-gerbe $\SS\to S$ is given by a class $[\SS]\in H^2(S,\mu_r)$, and is called ``essentially trivial" if it is in the image of the map $\psi$. 
Therefore an essentially trivial $\mu_r$-gerbe $\SS\to S$ is given by a line bundle $L\in\Pic(S)$. 

The cohomological Brauer group $$\Br^\prime(S)=H^2(S,\sO_S^*)_{\tor},$$
  is by definition the torsion part of the cohomology $H^2(S,\sO_S^*)_{\tor}$. De Jong's theorem \cite{de_Jong} implies that the Brauer group 
$\Br(S)=\Br^\prime(S)$, and $\Br(S)$ is the group of isomorphism classes of Azumaya algebras $\sA$ on $S$, see \cite[Definition 2.10]{Jiang_2019}. 
Here an Azumaya algebra $\sA$ on $S$ is an associative (non-commutative) $\sO_S$-algebra $\sA$ which is locally isomorphic to a matrix algebra $M_r(\sO_S)$ for some $r>0$. 

\begin{defn}\label{defn_optimal_gerbe}
A $\mu_r$-gerbe $p:\SS\to S$ is called ``optimal" if the period $\per(\SS)$, which is defined as the order of $[\SS]$ in $\Br^\prime(S)=\Br(S)$, is equal to $r$. 
\end{defn}

Let $\alpha:=\varphi([\SS])\in H^2(S,\sO_S^*)_{\tor}$ be the class of the gerbe $\SS$ in $\Br^\prime(S)$. Then $\alpha$ is called a Brauer class.  A K3 surface $(S,\alpha)$ together with a Brauer class $\alpha$ is called a twisted K3 surface in \cite{HS}.
In particular an optimal $\mu_r$-gerbe $\SS\to S$ determines a twisted K3 surface. 

\begin{rmk}
We use the notation $(S,\alpha)$, for $\alpha\in H^2(S,\sO_S^*)_{\tor}$, or $\SS_{\alpha}$ (where $\alpha$ also determines one class in $H^2(S,\mu_r)$) to represent the twisted K3 surface, where we can view $p: \SS_{\alpha}\to S$ an optimal $\mu_r$-gerbe. 
We may exchange the notation $(S,\alpha)$ and $\SS_{\alpha}$ arbitrarily in the paper. 
\end{rmk}

\begin{rmk}
Since $\Br^\prime(S)=\Br(S)$, any optimal $\mu_r$-gerbe $\SS_{\alpha}\to S$ for 
$\alpha\in H^2(S,\mu_r)\to H^2(S,\sO_S^*)$ gives an element $[P]\in H^1(S,\PGL_r)$, which classifies $\PGL_r$-bundles on $S$.
\end{rmk}

\subsubsection{Bridgeland stability conditions on twisted K3 surfaces}

In this section we review the Bridgeland stability conditions on twisted K3 surfaces in \cite{HMS}. Huybrechts, Macri and Stellari \cite{HMS} studied the stability condition on any generic K3 category and we only fix to category of twisted sheaves on K3 surfaces. 

For the twisted K3 surface $(S,\alpha)$ or $\SS_{\alpha}$, we let $\Coh(S,\alpha)$ or $\Coh^{\tw}(\SS_{\alpha})$ be the category of twisted sheaves on $S$ or $\SS_{\alpha}$.  
Let $D^b(\Coh(S,\alpha))$ or $D^b(\Coh^{\tw}(\SS_{\alpha}))$ be the corresponding derived category.  These two categories $\Coh(S,\alpha)$ and $\Coh^{\tw}(\SS_{\alpha})$ are equivalent if we forget about the $\mu_r$-gerbe structures on the twisted sheaves. 
We denote by $\Stab(S,\alpha)$ (or $\Stab^{\tw}(\SS_{\alpha})$) the stability manifold.  From \cite{HS}, there is a twisted Hodge structure 
$\widetilde{H}(S,\alpha,\zz)=\left(H^*(S,\zz), \langle, \rangle, -\frac{\xi}{r}\right)$, reviewed in Definition \ref{defn_integral_structure_K3}. 
Here $\xi\in H^2(S,\zz)$ is a lift of $\alpha\in H^2(S,\mu_r)$. 

Since our $\mu_r$-gerbe $\SS_{\alpha}$ is optimal, there is a corresponding Brauer-Severi variety
$p: P\to S$, where $P$ is a projective $\pp^{r-1}$-bundle over $S$. Then in \cite{Yoshioka2}, and \S \ref{subsubsec_twisted_Chern_character} we defined the category of 
$P$-sheaves $\Coh(S,P)$ on $P$ which is equivalent to $\Coh(S,\alpha)$.  We also define the Mukai vector 
$v_G(E)$ for a twisted sheaf $E$, where the rank $r$ vector bundle $G$ on $P$ is determined by the Euler sequence
$$0\to \sO_P\rightarrow G\rightarrow T_{P/S}\to 0$$
Note that from Proposition \ref{prop_repn_Mukai_vector} (also \cite[Lemma 3.3]{Yoshioka2}), if we 
set 
$$(\rk(E), D, a):=e^{\frac{\xi}{r}}\cdot v_G(E).$$
Then $(\rk(E), D, a)\in H^*(S,\zz)$ and $D \mod r=\omega(E)$, and $\omega(E)\in H^2(S,\mu_r)$ is $c_1(E)\mod r$.  
The Mukai vectors in \cite{HS} are of the form $e^{\frac{\xi}{r}}\cdot v_G(E)$. We work on Yoshioka's Mukai vectors and consider the integral structure:
$$T_{-\frac{\xi}{r}}\left(\left(H^*(S,\zz), \langle, \rangle, -\frac{\xi}{r}\right)\right)$$
in $H^*(S,\qq)$ in Definition \ref{defn_integral_structure_K3}.
 
 Recall that in \cite{HMS}, Huybrechts, Macri and Stellari defined a twisted Chern character:
 $$\Ch^{B}: K(S,\alpha)\to H^*(S,\qq)$$
 as follows:
 Here $B\in H^2(S,\qq)$ is the B-field such that $\exp(B)=\alpha$. This is because from 
 $$0\to \zz\to \sO_S\stackrel{\exp}{\longrightarrow}\sO_S^*\to 0$$
 and $H^3(S,\zz)=0$ we have the $0,2$-part of $B$ maps to 
 $\alpha\in H^2(S,\sO_S^*)_{\tor}$.  The Chern character map 
 $\Ch^B: K(S,\alpha)\to H^*(S,\qq)$ satisfies:
 \begin{enumerate}
 \item $\Ch^B$ is additive, i.e., $\Ch^B(E_1\oplus E_2)=\Ch^B(E_1)+\Ch^B(E_2)$;
 \item If $B=c_1(L)\in H^2(S,\zz)$, $\Ch^B(E)=e^{c_1(L)}\cdot \Ch(E)$;
 \item $\Ch^{B_1}(E_1)\cdot \Ch^{B_2}(E_2)=\Ch^{B_1+B_2}(E_1\otimes E_2)$ for two $B$-fields $B_1, B_2$;
 \item Any $E\in K(S,\alpha)$, we have $\Ch^B(E)\in e^{B}\cdot \left(\oplus_{p,p}H^{p,p}(S)\right)$.
 \end{enumerate}
 This twisted Chern character $\Ch^B$ has the property:
 if $B_0=k\cdot B\in H^2(S,\zz)$ for some $k\in\zz$, then 
 $$\Ch^B(E)^k=\Ch^{B_0}(E^{\otimes k})=e^{B_0}\cdot \Ch(E^{\otimes k}).$$
 This implies that 
 $$(e^{-B}\cdot \Ch^B(E))^k=e^{-B_0}\cdot \Ch^B(E)^k=\Ch(E^{\otimes k})\in \bigoplus_{p}H^{p,p}(S).$$
 
 \begin{defn}\label{defn_Hodge_structure_HS}
 We denote by $c$ or  $\widetilde{H}(S,B,\zz)$ the weight-two Hodge structure on $H^*(S,\zz)$ with 
 $$\widetilde{H}^{2,0}(S,\alpha):=\exp(B)(H^{2,0}(S))$$
 and $\widetilde{H}^{1,1}(S)$ its orthogonal complement (with respect to the Mukai pairing). 
 \end{defn}
Then the twisted Chern character:
$$\Ch^{B}: D^b(S,\alpha)\to  \widetilde{H}(S,\alpha,\zz)$$
can be extended to the derived category.
Let 
$\NS(S,\alpha):= \widetilde{H}^{1,1}(S,\alpha,\zz)$ be the twisted N\'eron-Severi group. Then 
$\N(D^b(S,\alpha))=\NS(S,\alpha)$, where $\N(D^b(S,\alpha))$ is the numerical Grothendieck group and the isomorphism is identified by $\Ch^B$.

Following \cite{HMS}, 

\begin{defn}\label{defn_PSalpha}
Let $P(S,\alpha)\subset \NS(S,\alpha)\otimes\cc$ be the open subset of $\NS(S,\alpha)\otimes\cc$ consisting of elements $\varphi$ such that 
the real part and imaginary part of $\varphi$ generate a positive plane in  $\NS(S,\alpha)\otimes\rr$.
\end{defn}

Let $B_0\in H^2(S,\qq)$ be a $B$-field lift of $\alpha$, that is: $\exp(B_0)=\alpha$. Then for any real ample class 
$\omega\in H^{1,1}(S,\zz)\otimes\rr$, we let 
$$\varphi=e^{b_0+i\omega}=1+(B_0+i\omega)+\frac{(B_0^2-\omega^2)}{2}+i(B_0\cdot \omega)\in\NS(S,\alpha)\otimes\cc.$$
The subset $P(S,\alpha)\subset \NS(S,\alpha)\otimes\cc$ has two connected components and we shall denote the one that contains 
$\varphi=\exp(B_0+i\omega)$ by $P^+(S,\alpha)$
and 
 $P^+(S,\alpha)\subset P(S,\alpha)\subset \NS(S,\alpha)\otimes\cc$. Also $\pi_1(P^+(S,\alpha))\cong \zz$.
 If $B_1\in \NS(S)\otimes\qq$, and $B=B_1+B_0$, we have 
 $$\varphi=e^{B+i\omega}=\exp(B_1)\cdot \exp(B_0+i\omega)\in\NS(S,\alpha)\otimes\cc$$
 and $\varphi\in P^+(S,\alpha)$. 
 Now for $\varphi\in P^+(S,\alpha)$, recall for any twisted sheaf $E$ on $S$, we can define the slope as:
 $$\mu(E)=\frac{c_1(E)\cdot \omega}{\rk(E)}$$
 which defines the geometric  slope semistability. For a torsion free twisted sheaf 
 $E$, let 
 $$0=E_0\subset E_1\subset\cdots\subset E_{m-1}\subset E_{m}=E$$
 be the Harder-Narasimhan filtration, where $E_i/E_{i+1}=F_i$ is $\mu_{\omega}$-semistable and 
 $\mu(F_i)>\mu(F_{i+1})$.  Defne 
 $\sT\subset \Coh(S,\alpha)$ to be the subcategory consisting of twisted sheaves whose torsion free part has 
 $\mu-{\omega}$-semistable Harder-Narasimhan factors of slope $\mu_{\omega}(F_i)>B\cdot \omega$ (or $\Im \sZ_{\varphi}(F_i)>0$). 
 A nontrivial twisted sheaf $E$ is an object in $\sF\subset \Coh(S,\alpha)$ if $E$ is torsion free and 
 $\mu_{\omega}(F_i)\leq B\cdot \omega$ (or $\Im \sZ_{\varphi}(F_i)\leq 0$). 
 Then $(\sT, \sF)$ is a torsion pair on $ \Coh(S,\alpha)$ and let 
$$\sA(\varphi):=\left\{E\in D^b(\Coh(S,\alpha))  \left| \begin{array}{l}
\text{$\bullet \hH^{i}(E)=0$ for $i\notin \{-1,0\}$}\\
\text{$\bullet \hH^{-1}(E)\in\sF$} \\
\text{$\bullet \hH^{0}(E)\in\sT$}
\end{array}\right
\}\right.$$ 
Then from \cite[Lemma 3.4]{HMS}, let 
$\varphi=e^{B+i\omega}$, the induced homomorphism
$$\sZ_{\varphi}: \sA(\varphi)\to \cc$$
is a stability function on $\sA(\varphi)$ if and only if for any spherical twisted sheaf
$E\in \Coh(S,\alpha)$, $\sZ_{\varphi}(E)\notin \rr_{\leq 0}$. 
\begin{prop}\label{prop_stability_condition_HMS}(\cite[Proposition 3.6]{HMS})
The pair $(\sZ_{\varphi}, \sA(\varphi))$ defines a stability condition on 
$D^b(\Coh(S,\alpha))$.  
\end{prop}
Then let 
$\Stab^{\circ}(S,\alpha)\subset \Stab(S,\alpha)$ be the connected component of $\Stab(S,\alpha)$
that contains the stability conditions described above. Also \cite[Proposition 3.8]{HMS} proved that if 
$\sigma=(\sZ,P)$ is contained in a connected component of maximal dimension $\Stab^{\circ}(S,\alpha)$, and for any closed point
$x\in S$, the skyscraper sheaf $k(x)$ is $\sigma$-stable of phase one with 
$\sZ(k(x))=-1$. Then there exists $\varphi=\exp(B+i\omega)\in P^+(S,\alpha)$ such that the heart of $\sigma$ is $\sA(\varphi)$.

For any $\varphi=B+i\omega\in P^+(S,\alpha)$, we can calculate (see \cite[Lemma 3.4]{HMS})
$$\sZ_{\varphi}: \sA(\varphi)\to \cc$$
If $v_{G}(E)=(r,l,s)$ for $r>0$, then 
\begin{equation}\label{eqn_sZ_varphi_E}
\sZ_{\varphi}(E)=\frac{1}{2r}\left((l^2-2rs)+r^2\omega^2-(l-rB)^2\right)+\left(\omega\cdot l-r(\omega\cdot B\right)i
\end{equation}
Also for $r=0$, 
\begin{equation}\label{eqn_sZ_varphi_E_rk0}
\sZ_{\varphi}(E)=\left(-s+l\cdot B\right)+\left(l\cdot \omega\right)i
\end{equation}
which is the same as in \cite[Formula (31)]{Toda_Adv}.

Let $\sV\subset \Stab^{\circ}(S,\alpha)$ be the open subset in $\Stab(S,\alpha)$ such that it consists of all stability conditions
$\sigma_{\varphi}$ for $\varphi=B+i\omega\in P^+(S,\alpha)$.
Then $\sV$ satisfies the following properties:

$\bullet$ For any $\sigma_{\varphi}\in\Stab^{\circ}(S,\alpha)$, there exists $\Phi\in \mbox{Auteq}(D^b(S,\alpha))$ and 
$g\in\widetilde{GL}^+(2,\rr)$ such that $g\circ \Phi(\sigma_{\varphi})$ is also algebraic and is contained in $\sV$, see
\cite[Remark 3.9]{HMS}. 

Now we also introduce the geometric twisted Gieseker stability. We follow \cite[\S 4.2]{Toda_Adv} and let 
$\sL, \sM\in \Pic(S)$
be two line bundles, and $\sL$ is ample. Define for any $E\in \Coh(S,\alpha)$, the twisted Hilbert polynomial:
$$\chi^g(E\otimes\sM^{-1}\otimes \sL^n)=\sum_{i=0}^{d}a_i n^i$$
for $a_i\in\qq$, $a_d\neq 0$. Let 
$\omega=c_1(\sL)$, $\beta=c_1(\sM)$ and the twisted reduced Hilbert polynomial is:
\begin{equation}\label{eqn_reduced_twisted_Hilbert}
p(E, \beta, \omega, n)=\frac{\chi^g(E\otimes\sM^{-1}\otimes \sL^n)}{a_d}
\end{equation}
If $v_G(E)=(\rk, l,s)$ for $\rk>0$, we can calculate using the definition of geometric Hilbert polynomial, see \cite[Proposition 3.21]{Jiang_2019},
(\ref{eqn_reduced_twisted_Hilbert}) is:
\begin{equation}\label{eqn_reduced_twisted_Hilbert_result}
r\cdot \Big[n^2+\frac{2\omega(l-\rk\cdot \beta)}{\rk\cdot \omega^2}n-\frac{(l^2-2\rk s-(l-\rk\cdot \beta)^2)}{\rk\cdot \omega^2}\Big]
\end{equation}
and when $\rk=0$, $l\neq 0$,
(\ref{eqn_reduced_twisted_Hilbert}) is:
\begin{equation}\label{eqn_reduced_twisted_Hilbert_result2}
r\cdot \Big[n+\frac{(s-\beta\cdot l)}{l\cdot \omega}\Big]
\end{equation}
These formula are the same as \cite[Formula (28), (29)]{Toda_Adv} up to a factor $r$ due to the gerbe structure. The twisted geometric stability of $E$ can be similarly defined. 
The following lemma is based on \cite[Lemma 4.6]{Toda_Adv}.

\begin{lem}\label{lem_lemma4.6_Toda}
We have:
\begin{enumerate}
\item Let $E\in D^b(S,\alpha)$ be simple, then $v_{G}(E)^2\geq -2$.
\item For any $\varphi=B+i\omega$, $m\in \rr_{>0}$, the number 
$$\#\{v_{\alpha}\in\NS(S,\alpha)| v_{\alpha}^2\geq -2, |\langle\exp(\varphi), v_{\alpha}\leq m|$$
is finite.
\item Let $E\in\N(S,\alpha)$, and $v_{\alpha}(E)=(0,l,s)\in\NS^*(S,\alpha)$, $l\neq 0$, then 
$$p(E,\beta,\omega,n)=n-\frac{\Re \sZ_{\varphi}(E)}{\Im \sZ_{\varphi}(E)}\in\qq[n].$$
\item Let $E, E^\prime\in\NS(S,\alpha)$, $p(E,\beta,\omega,n)=p(E^\prime,\beta,\omega,n)$ if and only if 
$\Im\frac{\sZ_{\varphi_k}(E^\prime)}{\sZ_{\varphi_k}(E)}=0$ for infinitely many $k\in\qq$, where $\varphi_k=B+ik\omega$.
\end{enumerate}
\end{lem}
\begin{proof}
$(1)$ is from \cite[Proposition 3.6]{Yoshioka2}, where Yoshioka proves it for twisted sheaves, but the same argument as in \cite[Lemma 5.1]{Bridgeland_K3}
works for derived simple object by Serre duality. 
$(2)$ just come from \cite[Lemma 8.2]{Bridgeland_K3}, which works for twisted sheaves on K3, since there are finitely many integral points in $\N(S,\alpha)\otimes\rr$.
$(3)$ and $(4)$ are from the above calculations (\ref{eqn_reduced_twisted_Hilbert_result}) and  (\ref{eqn_reduced_twisted_Hilbert_result2}).
\end{proof}

\subsection{The moduli stack counting semistable objects}\label{subsec_moduli_stack_App}

We list some boundedness of semistable objects following \cite[\S 4.5]{Toda_Adv}. The boundedness results are used to show that the moduli stack 
$\sM^{\textbf{c}}_{\sigma(\beta,\omega)}(S,\alpha)$ with topological invariant $\textbf{c}$ is an Artin stack of finite type.  Since for now, that the moduli stack of Bridgeland semistable objects is an Artin stack  is a well-known fact, we only need the 
boundedness results to prove our later results we interested in. 

Let $E\in\sA(\varphi)$ be an object. Set
$H^0(E)_{\tor}\subset H^0(E)$ to be the maximal torsion subsheaf of $H^0(E)$, and 
$$H^0(E)_{\free}=H^0(E)/H^0(E)_{\tor}.$$
We take Toda's notations:
$$T_1,\cdots, T_{a(E)}\in \Coh(S,\alpha)$$
$$F_1,\cdots, F_{d(E)}, F_{d(E)+1},\cdots, F_{e(E)}\in \Coh(S,\alpha)$$
represent the $\mu_{\omega}$-stable factors of $H^0(E)_{\free}$ and $H^{-1}(E)$
respectively. Let 
$$T_{a(E)+1},\cdots, T_{b(E)}, T_{b(E)+1},\cdots, T_{c(E)}\in \Coh(S,\alpha)$$
be the $\varphi=B+i\omega$-twisted stable factors of $H^{0}(E)_{\tor}$.
Set 
$$
\begin{array}{ll}
\dim(T_i)=2 & (1\leq i\leq a(E));\\
\dim(T_i)=1 & (a(E)< i\leq b(E));\\
\dim(T_i)=0 & (b(E)< i\leq c(E));\\
\Im \sZ_{\varphi}(F_i[1])>0 & (1\leq i\leq d(E)));\\
\Im \sZ_{\varphi}(F_i[1])=0 & (d(E)< i\leq e(E))).
\end{array}
$$
Define for $\textbf{c}\in \N(S,\alpha)$, 
$$\sM^{\textbf{c}}(\varphi=(\beta,\omega))=\{E\in\sA(\varphi)| \Im\sZ_{\varphi}(E)\leq \Im\sZ_{\varphi}(\textbf{c})\}.$$
The following results are from \cite[Lemma 4.8, Lemma 4.9]{Toda_Adv} which works for twisted sheaves. 

\begin{lem}\label{lem_lemma1_App}
The maps 
$\sM^{\textbf{c}}(\varphi)\to \zz$ given by 
$E\mapsto b(E)$ and $E\mapsto d(E)$ are bounded, and the sets 
$$\{\Im\sZ_{\varphi}(T_i)\in\qq | 1\leq i\leq c(E), E\in \sM^{\textbf{c}}(\varphi=(\beta,\omega))\}$$
and 
$$\{\Im\sZ_{\varphi}(F_i[1])\in\qq | 1\leq i\leq e(E), E\in \sM^{\textbf{c}}(\varphi=(\beta,\omega))\}$$
are finite sets. 
\end{lem}

\begin{lem}\label{lem_lemma2_App}
There exist constants $C, C^\prime, N$ (depending only on $\textbf{c}$, $B$, $\omega$), such that 
$$\frac{1}{k}\Re\sZ_{\varphi_k}(T_i)\geq \Re\sZ_{\varphi}(T_i)\geq C; (1\leq i\leq a(E))$$
and 
$$\frac{1}{k}\Re\sZ_{\varphi_k}(F_i[1])\geq \Re\sZ_{\varphi}(F_i[1])\geq C^\prime; 1\leq i\leq e(E)$$
for any $E\in \sM^{\textbf{c}}(\varphi)$ and 
$k\geq N$, where $\varphi_k=B+ik\omega$. 
\end{lem}

\subsection{Counting twisted semistable objects and twisted semistable sheaves}\label{subsec_objects_sheaves_comparison_App}

Let $H(\sA(\varphi))$ be the Hall algebra of the category $\sA(\varphi)$.  One can define Bridgeland stability $\sZ_{\varphi_k}$ and define the counting invariants 
$N(v)$ which count semistable twisted objects with Mukai vector $v$, see Definition \ref{defn_invariants_Nv}.

We also mimic the definition of invariants $J(v)$ for $\XX$ to define the invariants counting semistable twisted sheaves. 
Let $\Lambda:=\qq(q^{\frac{1}{2}})$ be a $\qq$-algebra, and 
$$\gamma: K(\Var)\to \Lambda$$
be the motivic invariants defined by Poincar\'e polynomial. 
Fixing a numerical invariant $\textbf{c}\in\N(S,\alpha)$, let 
$\sM^{\tw}_{\omega}(\textbf{c})\subset \widehat{\sM}(\sA(\varphi))$ be the stack of $\omega$-Gieseker semistable twisted sheaves of numerical invariant $\textbf{c}$. Then there exists an element 
$$\delta_{\omega, \SS_{\alpha}}(\textbf{c})=[\sM^{\tw}_{\omega}(\textbf{c})\to  \widehat{\sM}(\sA(\varphi))]$$
in the Hall algebra $H(\sA_{\SS_{\alpha}})=H(\Coh(S,\alpha))$. Define 
$$\epsilon_{\omega,\SS_{\alpha}}(\textbf{c})=
\sum_{\substack{\ell\geq 1, c_1+\cdots+c_{\ell}=c,\\
p(\textbf{c}_i,\omega,n)=p(\textbf{c}, \omega,n)}}
\frac{(-1)^{\ell-1}}{\ell}\delta_{\omega,\SS_{\alpha}}(\textbf{c}_1)\star\cdots\star\delta_{\omega,\SS_{\alpha}}(\textbf{c}_{\ell})
$$
Let 
$$C(S,\alpha):=\Im(\Coh(S,\alpha)\to \N(S,\alpha)).$$
If $\textbf{c}\in C(S,\alpha)$, define 
$$J(\textbf{c})=\lim_{q^{\frac{1}{2}}\to 1}(q-1)P_q(\epsilon_{\omega,\SS_{\alpha}}(\textbf{c})).$$
If $-\textbf{c}\in C(S,\alpha)$, define 
$$J(\textbf{c})=J(-\textbf{c}).$$
The invariant $J(\textbf{c})$ does not depend on the choice of $\omega$.
Our aim is to compare $J(\textbf{c})$ with the invariant $N(\textbf{c})$ counting semistable objects in 
$\sA(\varphi)$. For this purpose, we let 
$$
\begin{cases}
\varphi_k=B+ik\omega;\\
\sA(\varphi_k)=\sA(\varphi);\\
\sigma_k=(\sZ_{\varphi_k}, \sA(\varphi)).
\end{cases}
$$
Let us fix a $\textbf{c}=(\rk, l,s)\in H^*(S,\qq)$. We first have a generalization of \cite[Proposition 6.4]{Toda_Adv}.

\begin{prop}\label{prop_semistable_objects_sheaves}
Assume that 
$\omega\cdot l>0$ or $\rk=l=0$. Let us choose $0<\phi_k\leq 1$ such that $\sZ_{\varphi_k}(\textbf{c})\in\rr_{>0}e^{i\pi \phi_k}$.
Then there exists a $N>0$ such that for all $k\geq N$, and $\textbf{c}^\prime$ satisfying 
$$\textbf{c}^\prime\in C^{\sigma_k}(\phi_k); \quad  |\Im\sZ_{\varphi}(\textbf{c}^\prime)|\leq |\Im\sZ_{\varphi}(\textbf{c})|,$$
then any $E\in \sM^{(\textbf{c}^\prime, \phi_k)}(\sigma_k)$ (which is $\sigma_k$-semistable with $(\textbf{c}^\prime, \phi_k)$) is $\omega$-Gieseker semistable as a coherent sheaf. 
\end{prop}
\begin{proof}
In the case $\rk=l=0$, any object $E\in \sA_{\omega}$ of numerical invariant $\textbf{c}$ is a zero dimensional sheaf, so must be a semistable sheaf.
Now let $\omega\cdot l>0$. From the formula of $\sZ_{\varphi}(E)$ in (\ref{eqn_sZ_varphi_E}) and (\ref{eqn_sZ_varphi_E_rk0}), when $\rk>0$, the phase 
$\phi_k\to 0 (k\to \infty)$, and when $\rk=0$, $\phi_k\to \frac{1}{2} (k\to \infty)$. Therefore there exists $N>0$ such that 
$\phi_k\leq \frac{3}{4}$ for all $k\geq N$. Let $E\in  \sM^{(\textbf{c}^\prime, \phi_k)}(\sigma_k)$ and $\textbf{c}^\prime$ satisfies the condition the the proposition, we have
$$\phi_k(H^{-1}(E)[1])\leq \phi_k\leq \frac{3}{4}.$$
Look at 
\begin{align*}
E\mapsto \frac{\Re\sZ_{\varphi_k}(H^{-1}(E)[1])}{\Im\sZ_{\varphi_k}(H^{-1}(E)[1])}
&=\frac{\frac{1}{2\rk}(l^2-2\rk s+\rk^2k^2\omega^2-(l-\rk B)^2)}{(k\omega\cdot l)-\rk(k\omega\cdot B)}\\
&= \frac{\Re\sZ_{\varphi_k}(H^{-1}(E)[1])}{k\cdot \Im\sZ_{\varphi}(H^{-1}(E)[1])}.
\end{align*}
On the moduli spaces $\bigcup_{k\geq N, \textbf{c}^\prime}\ \sM^{(\textbf{c}^\prime, \phi_k)}(\sigma_k)$ which is bounded below, we have 
$E\in \sM^{\textbf{c}}(\varphi)$ since $|\Im\sZ_{\varphi}(E)|\leq |\Im\sZ_{\varphi}(\textbf{c})|$. Then the map 
$E\mapsto \Im \sZ_{\varphi}(E)$ on $\sM^{\textbf{c}}(\varphi)$ is bounded by Lemma \ref{lem_lemma1_App}. Therefore the map
$$E\mapsto \frac{1}{k}\Re\sZ_{\varphi_k}(H^{-1}(E)[1])$$
on $\bigcup_{k\geq N, \textbf{c}^\prime}\ \sM^{(\textbf{c}^\prime, \phi_k)}(\sigma_k)$  is bounded below. So by Lemma \ref{lem_lemma2_App}, 
$$E\mapsto \Re\sZ_{\varphi}(H^{-1}(E)[1])$$
is bounded below on $\bigcup_{k\geq N, \textbf{c}^\prime}\ \sM^{(\textbf{c}^\prime, \phi_k)}(\sigma_k)$. Then this imples (\cite[Lemma 4.10]{Toda_Adv}) that the set
$$\left\{v_{\alpha}(H^{-1}(E)[1])\in\NS^*(S,\alpha)| E\in \bigcup_{k\geq N, \textbf{c}^\prime}\ \sM^{(\textbf{c}^\prime, \phi_k)}(\sigma_k)\right\} $$
is a finite set.  If we let $\{v_1,\cdots,v_n\}$ be this set. Then $\lim_{k\to \infty}\phi_k(v_i)=1$, so we can make $\phi_k(v_i)>\frac{3}{4}$ for $k\geq N$  for some $N>0$. This implies that 
$E\in \sM^{(\textbf{c}^\prime, \phi_k)}(\sigma_k)$ has that $H^{-1}(E)=0$ meaning that $E$ is a coherent sheaf. 

The next step is to follow \cite[Proposition 6.4]{Toda_Adv} to show that $E$ is actually a Gieseker semistable twisted sheaf. 
If not, then let $T$ be the $\omega$-Gieseker semistable factor of $E$ of smallest reduced geometric Hilbert polynomial. Let 
$$v_{\alpha}(E)=((\rk)^{\prime}, l^\prime, s^\prime);\quad v_{\alpha}(T)=((\rk)^{\prime\prime}, l^{\prime\prime}, s^{\prime\prime}).$$
If $(\rk)^\prime=0$ which implies that $(\rk)^{\prime\prime}=0$, then from $(3)$ in  Lemma \ref{lem_lemma4.6_Toda}, $E$ is Gieseker twisted semistable. 
So we assume that $(\rk)^\prime>0$, $(\rk)^{\prime\prime}>0$.
In this case $\phi_k\to 0 (k\to \infty)$, we have $E\to T$ is surjective and since $E$ is $\sigma_k$-semistable for $k\geq N$, 
$\phi_k(E)\leq \phi_k(T)$.  So we have
\begin{equation}\label{eqn_key_1_App}
\frac{\Re\sZ_{\varphi_k}(E)}{\Im\sZ_{\varphi_k}(E)}\geq \frac{\Re\sZ_{\varphi_k}(T)}{\Im\sZ_{\varphi_k}(T)}.
\end{equation}
We calculate above as:
\begin{equation}\label{eqn_key_2_App}
\frac{\omega l^{\prime\prime}-(\rk)^{\prime\prime}\omega B}{\omega l^\prime-(\rk)^{\prime}\omega B}\left(-s^\prime+\frac{1}{2}(\rk)^{\prime}k^2\omega^2+l^{\prime}B-\frac{1}{2}(\rk)^{\prime}B^2\right)
\geq -s^{\prime\prime}+\frac{1}{2}(\rk)^{\prime\prime}k^2\omega^2+l^{\prime\prime}B-\frac{1}{2}(\rk)^{\prime\prime}B^2.
\end{equation}
We also have from the Mukai vector property: 
\begin{equation}\label{eqn_key_3_App}
0<(\rk)^{\prime\prime}<(\rk)^\prime, \quad 0<\omega l^{\prime\prime}\leq \omega l^\prime, \quad  (l^{\prime\prime})^2-2(\rk)^{\prime\prime}s^{\prime\prime}\geq -2.
\end{equation}
Then   (\ref{eqn_key_2_App}) and (\ref{eqn_key_3_App}) imply that the set 
$$\left\{ v_{\alpha}(T)\in \NS^*(S,\alpha)| E\in\bigcup_{k\geq N, \textbf{c}^\prime}\ \sM^{(\textbf{c}^\prime, \phi_k)}(\sigma_k)\right\} $$
is a finite set.  Therefore the set 
$$\left\{ v_{\alpha}(E_{\tor})\in \NS^*(S,\alpha)| E\in\bigcup_{k\geq N, \textbf{c}^\prime} \sM^{(\textbf{c}^\prime, \phi_k)}(\sigma_k)\right\} $$
is a finite set. Assume that this set is:
$$\{v_1^\prime,\cdots, v_m^\prime\}.$$
Then $\phi_k(v_i^\prime)\to \frac{1}{2}$ when $k\to \infty$. So 
$\phi_k(v_i^\prime)>\phi_k$ for all $1\leq i\leq m$ and $k\geq N$ after replacing 
$N$ if necessary. Hence for $k\geq N$, $E\in \sM^{(\textbf{c}^\prime, \phi_k)}(\sigma_k)$ must be torsion free.  By definition of $T$, 
$$\mu_{\omega}(E)>\mu_{\omega}(T)$$
or
$$\mu_{\omega}(E)=\mu_{\omega}(T),  \quad  \frac{s^\prime}{(\rk)^\prime}-\frac{l^\prime B}{(\rk)^{\prime}}>\frac{s^{\prime\prime}}{(\rk)^{\prime\prime}}-\frac{l^{\prime\prime} B}{(\rk)^{\prime\prime}},$$
since we can calculate
$$\frac{\sZ_{\varphi_k}(E)}{(\rk)^\prime}-\frac{\sZ_{\varphi_k}(T)}{(\rk)^{\prime\prime}}
=-\left( \frac{s^\prime}{(\rk)^\prime}-\frac{s^{\prime\prime}}{(\rk)^{\prime\prime}}\right)+ \left(\frac{l^\prime B}{(\rk)^{\prime}}-\frac{l^{\prime\prime} B}{(\rk)^{\prime\prime}}\right)
+ik(\mu_{\omega}(E)-\mu_{\omega}(T)).
$$
Therefore replacing $N$ if necessary, we have $\phi_k(E)>\phi_k(T)$ for $k\geq N$. 
This $N$ is only determined by the numerical class of $T$, and the finiteness of $v_{\alpha}(T)$ can make this $N$ uniformly such that $\phi_k(E)>\phi_k(T)$
for $k\geq N$, which contradicts $E$ is $\sigma_k$-semistable. 
\end{proof}

Next we have a similar result as in \cite[Lemma 6.5]{Toda_Adv}. 
\begin{lem}\label{lem_Toda_6.5_App}
If $\omega\cdot l>0$ or $\rk=l=0$, then there exists a $N>0$ such that for $k\geq N$, and $\textbf{c}^\prime\in C(S,\alpha)$, 
\begin{equation}\label{eqn_lem_6.5}
p(\textbf{c}^\prime, \omega, n)=p(\textbf{c}, \omega, n)
\end{equation}
then any Gieseker semistable twisted sheaf $E$ of numerical type $\textbf{c}^\prime$ is $\sigma_k$-semistable. 
\end{lem}
\begin{proof}
First the set  $\textbf{c}^\prime\in C(S,\alpha)$ satisfying (\ref{eqn_lem_6.5}) is finite. 
So we take $\textbf{c}=\textbf{c}^\prime$. The case of $\rk=l=0$ is obvious. For the case $\rk>0, \omega\cdot l>0$, the smooth surface case is proved in 
\cite[Proposition 14.2]{Bridgeland_K3}. Since the twisted stability in \cite{HMS} is similar to the construction of Bridgeland, the proof of \cite[Proposition 14.2]{Bridgeland_K3}
works for twisted sheaves. 

For the case $\rk=0, l\neq 0$, $E$ is $\omega$-Gieseker semistable. In this case $\phi_k\to \frac{1}{2}$ when $k\to \infty$. Then Toda's proof 
in \cite[Lemma 6.5]{Toda_Adv} works in this case. 
\end{proof}

We show:
\begin{thm}\label{thm_J_N}
For $\textbf{c}\in C(S,\alpha)$, we have
$$N_{\sigma}(\textbf{c})=J(\textbf{c})$$
and $N_{\sigma}(\textbf{c})$ does not depend on the stability condition $\sigma$.
\end{thm}
\begin{proof}
First for the category of twisted sheaves $\Coh(S,\alpha)$, the $K$-theory \cite{Yoshioka2}
$$K(\Coh(S,\alpha))=\qq[E_0]\oplus K_{\leq 1}(\Coh(S,\alpha))$$
where $E_0$ is the minimal rank $\rk>0$ (in this case is $r$) of $\SS_{\alpha}$-twisted locally free sheaves. 
So tensoring with $E_0$ gives an equivalence on $K(\Coh(S,\alpha))$ and hence on the derived category $D^b(S,\alpha)$, thus the derived equivalence 
result of \cite[Corollary 5.26]{Toda_Adv} implies that:
$$N(\textbf{c}\otimes E_0)=N(\textbf{c}); \quad   J(\textbf{c}\otimes E_0)=J(\textbf{c}).$$
Thus we can assume that 
$\omega\cdot l>0$, or $\rk=l=0$.
We follow Toda to show that:
$$N_{\sigma_k}(\textbf{c})=J_{\omega}(\textbf{c})$$
 and $N$ is chosen before as in Proposition \ref{prop_semistable_objects_sheaves}. 
 Let $\textbf{c}_1,\cdots, \textbf{c}_m\in C^{\sigma_k}(\phi_k)$ be such that 
 $$\textbf{c}_1+\cdots +\textbf{c}_m=\textbf{c}; \quad  \prod_{i=1}^{m}N_{\sigma_k}(\textbf{c}_i)\neq 0.$$
 Here we consider $\sigma_k\in\sB^{\circ}$ as an open set such that 
 $\overline{\sB}^{\circ}=\sB$ is compact in the stability manifold. There exists a wall and chamber structure $\{\sW_{\gamma}\}_{\gamma\in\Gamma}$ on $\sB$
 with property:
 $$\sS:=\{E\in D^b(S,\alpha)| E \text{~is semistable for some~}\sigma^\prime=(\sZ^\prime, P^\prime)\in\sB; 
 |\sZ^\prime(E)|\leq |\sZ^\prime(\textbf{c})|\}.$$
 $\sS$ is a bounded mass, which means 
 there exists $m>0$ such that 
 $m_{\sigma}(E)\leq M$ for any $E\in \sS$, where 
 $m_{\sigma}(E)=\sum_{i=1}^{n}|\sZ(A_i)|$
 for 
 \[
 \xymatrix{
 0=E_0\ar[rr]&& E_1\ar[rr]\ar[dl]&& E_2\ar[r]\ar[dl]\ar[r]&& \cdots\cdots & E_{n}=E \ar[dl]\\
 &A_1\ar[ul]&& A_2\ar[ul]&&& A_n\ar[ul]
 }
 \]
 So this means that there exists $\Gamma^\prime\subset \Gamma$ and a connected component $\sC$ such that 
 $$\sC\subset \bigcap_{\gamma\in\Gamma^\prime}(\sB\cap\sW_{\gamma})\setminus \bigcup_{\gamma\notin \Gamma^\prime}\sW_{\gamma}.$$
 Infinitely many $\sigma_{k^\prime}$ for $k^\prime\in \qq_{\geq N}$ are contained in $\sC$.  Then $\sigma_k\in\sC$. If $\textbf{c}_i$ and $\textbf{c}_j$ are not proportional in $\N(S,\alpha)$, then
 $$\Im \frac{\sZ_{\varphi_{k^\prime}}(\textbf{c}_j)}{\sZ_{\varphi_{k^\prime}}(\textbf{c}_i)}=0$$
 for infinitely many $k^\prime\in \qq_{\geq N}$. Therefore from Lemma \ref{lem_lemma4.6_Toda} $(4)$, 
 $$p(\textbf{c}_i,\omega,n)=p(\textbf{c}_j, \omega, n)=p(\textbf{c}, \omega, n)$$
 for $i, j$. From Lemma \ref{lem_Toda_6.5_App}, 
 \begin{equation}\label{eqn_last_formula}
 \sM^{(\textbf{c}_i, \phi_k)}(\sigma_k)=\sM^{\textbf{c}_i}(\omega)
 \end{equation}
 and $\prod_{i=1}^{m}N_{\sigma_k}(\textbf{c}_i)=\prod_{i=1}^{m}J_{\omega}(\textbf{c}_i)$.
 
 On the other hand, if $\textbf{c}_1, \cdots,  \textbf{c}_m\in C(S,\alpha)$ such that 
 $\prod_{i=1}^{m}J_{\omega}(\textbf{c}_i)\neq 0$ and $\textbf{c}_1+\cdots+\textbf{c}_m=\textbf{c}$, 
 $p(\textbf{c}_i,\omega,n)=p(\textbf{c}, \omega, n)$,  (\ref{eqn_last_formula}) still holds for $k\geq N$
 by Proposition \ref{prop_semistable_objects_sheaves} and Lemma \ref{lem_Toda_6.5_App} above. Hence 
$\prod_{i=1}^{m}N_{\sigma_k}(\textbf{c}_i)=\prod_{i=1}^{m}J_{\omega}(\textbf{c}_i)$.  Also 
$\textbf{c}_i\in C^{\sigma_k}(\phi_k)$ so 
$J_{\omega}(\textbf{c})=N_{\sigma_k}(\textbf{c})$.  That $N_{\sigma_k}(\textbf{c})$ does not depend on $\sigma_k$ is just 
from a former argument or \cite[Proposition 5.17]{Toda_Adv}.
\end{proof}

%%%%%%%%%%%%%%%%%%

%%%----------------------------------------------------------------------
%%%----------------------------------------------------------------------

%\subsection*{}

% ------------------------------------------------------------------------
\end{document}